\documentclass[a4paper,11pt,reqno]{article}
\usepackage{geometry}
\usepackage{amsthm} 

\usepackage{amsmath,amstext,amscd,amssymb,euscript,mathrsfs}
\usepackage{calrsfs}
\usepackage{epsfig}
\usepackage{mathtools}
\usepackage[colorlinks=true]{hyperref} 
\usepackage[english]{babel}
\usepackage{verbatim} 
\usepackage{enumerate}
\usepackage{bbm}
\usepackage{color}
\usepackage{wrapfig} 

\newcommand{\SortNoop}[1]{}

\usepackage{enumitem}
\setenumerate{label={\normalfont(\alph*)},topsep=4pt,itemsep=0pt} 

\theoremstyle{plain}
\newtheorem{theorem}{Theorem}[section]
\newtheorem{lemma}[theorem]{Lemma}

\newtheorem{corollary}[theorem]{Corollary}

\theoremstyle{definition}
\newtheorem{definition}[theorem]{Definition}

\newtheorem{remark}[theorem]{Remark}

\newtheorem{obs}[theorem]{\!\!}

\usepackage{times}

\usepackage{stmaryrd}

\newcommand{\1}{\mathbbm{1}}

\newcommand{\dd}{\, \mathrm{d}}

\newcommand{\C}{\mathbb C}
\newcommand{\T}{\mathbb T}
\newcommand{\E}{\mathbb E}

\newcommand{\Z}{\mathbb Z}
\newcommand{\R}{\mathbb R}
\newcommand{\N}{\mathbb N}
\newcommand{\Q}{\mathbb Q}

\newcommand{\cC}{\mathcal{C}}
\newcommand{\cD}{\mathcal{D}}

\newcommand{\cF}{\mathcal{F}}
\newcommand{\cG}{\mathcal{G}}

\newcommand{\cR}{\mathcal{R}}
\newcommand{\cS}{\mathcal{S}}
\newcommand{\cT}{\mathcal{T}}

\newcommand{\cW}{\mathcal{W}}
\newcommand{\cX}{\mathcal{X}}
\newcommand{\cY}{\mathcal{Y}}
\newcommand{\cZ}{\mathcal{Z}}

\newcommand{\fd}{\mathfrak{d}}

\newcommand{\fn}{\mathfrak{n}}

\newcommand{\fp}{\mathfrak{p}}  
\newcommand{\fq}{\mathfrak{q}}
\newcommand{\fr}{\mathfrak{r}}
\newcommand{\fs}{\mathfrak{s}}

\newcommand{\fC}{\mathfrak{C}}   
\newcommand{\fD}{\mathfrak{D}}

\newcommand{\fR}{\mathfrak{R}}
\newcommand{\fS}{\mathfrak{S}}

\newcommand{\fX}{\mathfrak{X}}

\newcommand{\sH}{\mathscr{H}}

\renewcommand{\P}{\mathbb{P}}

\newcommand{\supp}{\operatorname{supp}} 
\newcommand{\ran}{\operatorname{ran}} 

\newcommand{\Dir}{\fd}
\newcommand{\Even}{\fn}

\usepackage{mathtools} 

\newcommand{\rD}{\mathrm{D}}

\renewcommand{\i}{\mathrm i}

\newcommand{\para}{\varolessthan}
\newcommand{\arap}{\varogreaterthan}
\newcommand{\reso}{\varodot}

\newcommand{\bxi}{{\boldsymbol{\xi}}}
\newcommand{\bzeta}{{\boldsymbol{\zeta}}}
\newcommand{\btheta}{{\boldsymbol{\theta}}}
\newcommand{\blambda}{{\boldsymbol{\lambda}}}

\newcommand{\I}{\mathbb{I}}

\newcommand{\limn}{\lim_{n\rightarrow \infty}}

\renewcommand{\epsilon}{\varepsilon}

\usepackage[disable]{todonotes}

\usepackage[normalem]{ulem} 

\newenvironment{colorwil}
    {\color[rgb]{0,0.5,0}    }

\newcommand{\ew}{\normalcolor}

\newcommand{\bwil}{\begin{colorwil}}
\newcommand{\ewil}{\end{colorwil}}

\newcommand{\towil}{\todo[color=green!50]}

\newenvironment{calc}    {\color{orange}     }    {     }
\newenvironment{calc2}    {     }    {     }
\newcommand{\cnewline}{\\}
\newcommand{\cand}{&}

\newenvironment{hide}    {\color{purple}     }    {     }

\usepackage{environ}

\RenewEnviron{calc}{}  
\renewcommand{\cnewline}{}  
\renewcommand{\cand}{}  

\RenewEnviron{hide}{}  

\usepackage{tocloft}
\newcommand{\bfold}{$}
\newcommand{\efold}{$ }

\begin{document}

\renewcommand{\thefootnote}{\Roman{footnote}}

\title{Asymptotics of the eigenvalues of the Anderson Hamiltonian with white noise potential in two dimensions}

\author{
\renewcommand{\thefootnote}{\Roman{footnote}}
Khalil Chouk
\footnotemark[1]
\\
\renewcommand{\thefootnote}{\Roman{footnote}}
Willem van Zuijlen
\footnotemark[2]
}

\footnotetext[1]{School of Mathematics,
University of Edinburgh, United Kingdom.}
\footnotetext[2]{
Weierstrass Institute for Applied Analysis and Stochastics, Berlin, Germany.
}

\date{December 6, 2020}

\maketitle 

\renewcommand{\thefootnote}{\arabic{footnote}} 

\begin{abstract}
In this paper we consider the Anderson Hamiltonian with white noise potential on the box $[0,L]^2$ with Dirichlet boundary conditions. 
We show that all the eigenvalues divided by $\log L$ converge as $L\rightarrow \infty$ almost surely to the same deterministic constant, which is given by a variational formula. 


\end{abstract}

\setcounter{tocdepth}{1} 
\tableofcontents

\section{Introduction}

We consider the Anderson Hamiltonian (also called random Schr\"odinger operator), formally defined by
\bfold 
\sH = \Delta + \xi,
\efold 
under Dirichlet boundary conditions on the two-dimensional box $[0,L]^2$, where $\xi$ is considered to be white noise. \towil{PUT DATE BY HAND!!}
We are interested in the behaviour of this operator as the size of the box, $L$, tends to infinity. 
In this paper we prove the following asymptotics of the eigenvalues. 
Let $\lambda(L)=\lambda_{1}(L) > \lambda_{2}(L) \ge \lambda_3(L) \cdots$ be the  eigenvalues of the Anderson Hamiltonian on $[0,L]^2$.  For all $n\in\N$, almost surely 
\begin{align*}
\lim_{ \overset{  L \in \Q }{ L \rightarrow \infty}}
\frac{\lambda_n(L)}{\log L} 
& = 4
\sup_{ \overset{ \psi \in C_c^\infty(\R^2) }{ \|\psi \|_{L^2}^2 =1} } 
\|\psi\|_{L^4}^2 - 
\int_{\R^2} | \nabla \psi|^2 = \chi,
\end{align*} 
where $\chi$ is the smallest $C>0$ such that $\|f\|_{L^4}^4 \le C \|\nabla f\|_{L^2}^2 \|f\|_{L^2}^2$ for all $f\in H^1(\R^2)$ (this is Ladyzhenskaya's inequality).

\subsection{Main challenge and literature}

In the one dimensional setting, i.e., on the box $[0,L]$, the Anderson Hamiltonian can be defined using the associated Dirichlet form as the white noise is sufficiently regular, see Fukushima and Nakao \cite{FuNa76} (see \cite{Ve11} for the regularity of white noise).  
In dimension two the regularity of white noise is too small to allow for the same approach. 
A naive way to tackle the problem of the construction is to take a smooth approximation of the white noise $\xi_\epsilon$ so that the operator $\sH_\epsilon=\Delta+\xi_\epsilon$ is well-defined as an unbounded self-adjoint operator, and then take the limit $\epsilon \downarrow 0$. 
However, $\sH_\epsilon$ does not converge, but $\sH_\epsilon - c_\epsilon$ does converge to an operator $\sH$ for certain renormalisation constants $c_\epsilon \nearrow_{\epsilon \downarrow 0} \infty$. 
This has been shown by Allez and Chouk \cite{AlCh15} for periodic boundary conditions, using the techniques of paracontrolled distributions introduced by Gubinelli, Imkeller and Perkowski \cite{GuImPe15} in order to study singular stochastic partial differential equations. In this paper we extend this to Dirichlet boundary conditions. 

Recently, also Labb\'e \cite{La19} constructed the Anderson Hamiltonian  with both periodic and Dirichlet boundary conditions, using the tools of regularity structures. 
Gubinelli, Ugurcan and Zachhuber \cite{GuUgZa20} extend the work of Allez and Chouk to define the Anderson Hamiltonian with periodic boundary conditions also for dimension $3$. 

One of the main interests in the study of this operator is due to its universal property, more precisely, it was proved by Chouk, Gairing and Perkowski \cite[Theorem 6.1]{ChGaPe17}
that under periodic boundary conditions the operator $\mathscr H$ is the limit \begin{calc}(in the resolvent sense)\end{calc} under a suitable renormalisation of the discrete Anderson Hamiltonian 
\bfold
\mathscr H_N=\Delta_N+\tfrac{1}{N} \eta_N
\efold
defined on the periodic lattice $(\frac{1}{N} \mathbb Z/N\mathbb Z)^2$ where $\Delta_N$ is discrete Laplacian and $(\eta_N(i),i\in\mathbb Z^2)$ are centred I.I.D. random variables with normalised variance and finite $p$-th moment, for some $p>6$. 

Recently, Dumaz and Labb\'e \cite{DuLa17}  proved the Anderson localization  for the one dimensional case for the largest eigenvalues and they obtain the exact fluctuation of the eigenvalue and the exact behaviour of the eigenfunctions near their maxima. 
Unfortunately, their approach used to tackle the Anderson localization in the one dimensional setting is strongly attached to the SDE obtained by the so-called Riccati transform and cannot be adapted to the two dimensional setting. 
Also Chen  \cite{Ch14} considers the one dimensional setting for the white noise (and shows $\lambda(L) \approx (\log L)^{\frac23}$), but also a higher dimensional setting for the more regular fractional  white noise (where $\lambda(L) \approx (\log L)^{\beta}$ for some $\beta \in (\frac12,1)$ (and $\beta \in (\frac12, \frac{2}{3})$ for $d=1$), where $\beta$ is a function of the degree of singularity of the covariance at zero). The techniques in his work do not allow for an extension to a higher dimensional setting with a white noise potential. 
\begin{calc}
In \cite[Lemmas 2.3 and 4.1]{Ch14} the almost sure convergence of the principal eigenvalue is stated. 
\end{calc}

The asymptotics of the principal eigenvalue is of particular interest for the asymptotics of the total mass of the solution to the parabolic Anderson model: $\partial_t u = \Delta u + \xi u= \sH u$. 
Chen \cite{Ch14} shows that with $U(t)$ the total mass of $u(t,\cdot)$, one has $\log U(t) \approx t \lambda(L_t)$ for some almost linear $L_t$, so that the asymptotics of $\lambda(L)$ leads to asymptotics of $\log U(t)$: In $d=1$ with $\xi$ white noise, $\log U(t) \approx t (\log t)^{\frac23}$; for $d\ge 1$ with $\xi$ a fractional white noise $\log U(t) \approx t (\log t)^\beta$, with $\beta$ as above. 
For smooth Gaussian fields $\xi$, Carmona and Molchanov \cite{CaMo95} show $\log U(t) \sim t (\log t)^{\frac12}$.
In a future work by K\"onig, Perkowski and van Zuijlen, the following asymptotics of the total mass of the solution to the parabolic Anderson model with white noise potential in two dimensions will be shown: $\log U(t) \approx t \log t$.

For a general overview about the parabolic Anderson model and the Anderson Hamiltonian 
we refer to the book by K\"onig \cite{Ko16}.

Let us mention that our main result is already applied in \cite{PeRo19} to prove that the super Brownian motion in static random environment is almost surely super-exponentially persistent.

\begin{calc}
\begin{remark}
About defining the operator using Dirichlet forms. 
\cite[Theorem VIII.15]{ReSi75} states that every closed semi-bounded quadratic form is the form of a unique self-adjoint operator 
Considering one dimension, white noise is of regularity a little less than $-\frac12$ in the sense that $\xi \in B_{\infty,\infty}^{-\frac12-\epsilon}$ for all $\epsilon>0$. 
For $u,v \in H_0^1$ one has $uv \in B_{1,1}^1$ (by Cauchy-Schwarz). Therefore the pairing with $\xi$ is (almost surely) well-defined and continuous by \cite[Theorem 2.76]{BaChDa11} and so $q(u,v) := \langle \nabla u, \nabla v \rangle + \langle \xi, uv \rangle$ defines a semi-bounded quadratic form on $H_0^1$.
Note that $q(u,u)$ is equivalent to $\|u\|_{H_0^1}^2$ by Poincar\'e's inequality, so that $q$ is also closed and hence is the form of a unique self-adjoint operator. 

In two dimensions, this does not work as the product $uv$ is still in $B_{1,1}^1$ but $\xi$ does not have values in $B_{\infty,\infty}^{-1}$ (the dual of $B_{1,1}^1$) but in $B_{\infty,\infty}^{-1-\epsilon}$ for all $\epsilon>0$. 
\end{remark}

\begin{theorem} \cite[Theorem 5.1]{CaMo95}
Let $V$ be a mean zero stationary Gaussian field on $\R^d$ with covariance function $\gamma$, i.e., $V(x)$ is a mean zero Gaussian random variable and $\E[V(x) V(0)] = \gamma(x)$. 
With $u$ the solution to the parabolic Anderson model, 
$ 
\partial_t u = \Delta u + V  u, 
$ 
for all $x\in \R^d$
\begin{align*}
\frac{\log u(t,x)}{ t \sqrt{\log t}} = \sqrt{2 d \gamma(0)} \qquad a.s. 
\end{align*}
\end{theorem}

\begin{remark}
This then leads to the asymptotics of the total mass, as mentioned in the introduction. 
In their paper they need not mention the asymptotics of the principal eigenvalue, as their approach does not use the eigenvalue expansion. 
However, by using the heuristics mentioned above that $\log U(t) \approx t \lambda(L_t)$, one expects $\lambda(L) \approx (\log L)^{\frac12}$. 
This implies that one cannot interchange limits in $L$ and $\epsilon$ for $\lambda(Q_L,\xi_\epsilon)$, where $\xi_\epsilon$ denotes a mollification of $\xi$. 
\end{remark}

\begin{theorem}[ McKean \cite{Mc94}]
\begin{align*}
\lim_{L\rightarrow \infty} \P\left( 
\tfrac{L}{\pi} \lambda(L)^\frac12 \exp( - \tfrac83 \lambda(L)^\frac32 ) > x 
\right) 
= 
\begin{cases}
1 & x <0, \\
e^{-x} & x \ge 0 . 
\end{cases}
\end{align*}
\end{theorem}

\begin{theorem} \cite[Lemmas 2.3 and 4.1]{Ch14} 
Let $V$ be a mean zero stationary generalised Gaussian field on $\R^d$ with covariance function $\gamma: \R^d \setminus \{0\}$ with $\gamma(x) \sim c |x|^{-\alpha}$ as $x\rightarrow 0$ for some $\alpha \in (0,2\wedge d)$. 
This means that for all $\varphi,\psi \in \cS(\R^d)$, 
$\langle V, \varphi \rangle$ is a mean zero Gaussian random variable and 
\begin{align*}
\E[\langle V, \varphi \rangle \langle V, \psi \rangle] = \int_{\R^d} \int_{\R^d} \gamma(x-y) \varphi(x) \psi(y) \dd x \dd y .  
\end{align*}
Then $ (\log L)^{-\frac{2}{4-\alpha}} \lambda((-L,L)^d, V)$ converges almost surely to a deterministic scalar, which can be described in terms of $d$, $\alpha$ and $\gamma$. 

In case $V$ is white noise in dimension $1$ (formally, $\gamma = \delta_0$), then 
$ (\log L)^{-\frac{2}{3}} \lambda((-L,L), V)$ converges almost surely to a deterministic scalar. 
\end{theorem}

\end{calc}

\subsection{Outline}

In Section~\ref{section:main_results} we state the main results of this paper. 
In Section~\ref{section:tail_bounds} we give a proof of the tail bounds of the eigenvalues using the other ingredients presented in Section \ref{section:main_results}, and use this to prove the main theorem. 
The definitions of our Dirichlet and Neumann (Besov) spaces and para- and resonance products between those spaces are given in Section~\ref{section:dirichlet_besov}. 
With the definitions given we can properly define the Anderson Hamiltonian on its Dirichlet domain and state the spectral properties in Section~\ref{section:operator_def}. 
In Section \ref{section:white_noise} we prove the convergence to enhanced white noise, that will be used to extend properties for smooth potentials to analogue properties where enhanced white noise is taken. 
In Section~\ref{section:scaling_and_translation} we prove scaling and translation properties. 
In Section~\ref{section:eigenvalues_on_boxes} we compare eigenvalues on boxes of different size.
In Section~\ref{section:ldp_enhancement} we prove the large deviation principle of the enhanced white noise. 
This leads to the large deviation principle for the eigenvalues. 
In Section~\ref{section:ldp_results} we study infima over the large deviation rate function, which are used to express the limit of the eigenvalues. 
The more cumbersome calculations needed to prove convergence to enhanced white noise are postponed to Section~\ref{section:proof_white_noise_conv} and Section~\ref{section:proof_convergence_diff_box_sizes}.

\textbf{Acknowledgements.} 
The authors are grateful to G. Cannizzaro, P. Gaudreau Lamarre, C. Labb\'e,  W. K\"onig, A. Martini, T. Orenshtein, N. Perkowski, A.C.M. van Rooij, T. Rosati and R.S. dos Santos  for discussions and valuable feedback. 
KC contributed to this paper when he was employed at the Technische Universit\"at Berlin and was supported by the European Research Council through Consolidator
Grant 683164.
WvZ is supported by the German Science Foundation (DFG) via the Forschergruppe FOR2402 ``Rough paths, stochastic partial differential equations and related topics".

\subsection{Notation}
\label{subsection:notation}

$\N = \{1,2,\dots\}$, $\N_0 = \{0\} \cup \N$, $\N_{-1} =\{-1\} \cup \N_0$. 
$\delta_{k,l}$ is the Kronecker delta, i.e., $\delta_{k,k}=1$ and $\delta_{k,l} =0$ for $k\ne l$. $\i = \sqrt{-1}$. 
For $f,g \in L^2(D)$, for some domain $D \subset \R^d$ we write $\langle f, g \rangle_{L^2(D)} =\int_D f  \overline g$. 
We write $\T_L^d$ for the $d$-dimensional torus of length $L>0$, i.e., $\R^d / L\Z^d$. $(\Omega, \P)$ will be our underlying complete probability space. In order to avoid cumbersome administration of constants, for families $(a_i)_{i\in \I}$ and $(b_i)_{i\in \I}$ in $\R$, we also write $a_i \lesssim b_i$ to denote that there exists a $C>0$ such that $a_i\le C b_i$ for all $i \in \I$ and $a_i \eqsim b_i$ to denote that both $a_i \lesssim b_i$ and $a_i \gtrsim b_i$ (i.e., $b_i \lesssim a_i$). 
We write $C_c^\infty(A)$ for those functions in $C^\infty(A)$ that have compact support in $A^\circ$.

\section{Main results}
\label{section:main_results}

In this section we give the main results of this paper without the technical details and definitions; the main theorem is Theorem \ref{theorem:as_convergence_all_eigenvalues}. 

We build on the methods on the construction of the Anderson Hamiltonian in \cite{AlCh15}. 
In that paper the operator is considered on the torus or differently said, on a box with periodic boundary conditions. 
In order to consider Dirichlet boundary conditions we will consider the domain to be a subset of $H_0^1$. 
The construction in \cite{AlCh15} relies on Bony estimates for para- and resonance products. 
We therefore have to find the right space in which we take $\xi$ in order to be able to take para- and resonance products of $\xi$ with elements in the domain. 
For this reason we construct the framework of Dirichlet, $B_{p,q}^{\fd,\alpha}$, and Neumann Besov spaces, $B_{p,q}^{\fn,\alpha}$ in Section~\ref{section:dirichlet_besov}. 
We will show that $H_0^\gamma$ agrees with $B_{2,2}^{\fd,\gamma}$ and show that the Bony estimates extend to products between elements of Dirichlet and Neumann spaces. 
Basically the idea is as follows, for $d=1$ and $L=1$.
Instead of the basis for the periodic Besov space $L^2$, given by $x\mapsto e^{2\pi \i k x }$ we build the Dirichlet Besov space by the basis of $L^2$ given by $x \mapsto \sin(\pi k x)$ and the Neumann Besov space by $x \mapsto \cos(\pi k x)$. 
The elements of the Dirichlet/Neumann Besov space on $[0,L]$ then extend oddly/evenly to elements of the periodic Besov space on $\T_{2L}$. 
We show that the extension of a product is the same as the product of the respective extensions, which allows us to obtain the Bony estimates from the periodic spaces. Moreover, this also allows us to extend the main theorem in \cite{AlCh15}  to Dirichlet boundary conditions on $Q_L=[0,L]^2$, as we present in the following theorem. 
We will consider $\xi$ in $\cC_\fn^\alpha$ and its enhancement in $\fX_\fn^\alpha$, which are the Neumann analogues of $\cC^\alpha$ and $\fX^\alpha$.

\begin{theorem}[Summary of Theorem \ref{theorem:dirichlet_summary}]
\label{theorem:summary_of_dirichlet_summary}
Let $\alpha \in (-\frac43,-1)$.
Let $y\in \R^2,L>0$ and $\Gamma = y + Q_L$. 
For an enhanced Neumann distribution $\bxi=(\xi,\Xi) \in \fX_\fn^\alpha(\Gamma)$ we construct a stongly paracontrolled Dirichlet domain $\fD_\bxi^\Dir(\Gamma)$, such that the Anderson Hamiltonian on $\fD_\bxi^\Dir(\Gamma)$ maps in $L^2(\Gamma)$ and is self-adjoint as an operator on $L^2(\Gamma)$ with a countable spectrum given by eigenvalues $\lambda(\Gamma,\bxi) =\lambda_1(\Gamma,\bxi) > \lambda_2(\Gamma,\bxi)\ge \cdots$ (counting multiplicities). 
For all $n\in\N$ the map $\bxi \mapsto \lambda_n(\Gamma,\bxi)$ is locally Lipschitz.
Moreover, a Courant-Fischer formula is given for $\lambda_n$ (see \eqref{eqn:min-max_dir}). 
\end{theorem}

In Section \ref{section:white_noise} we show that there exists a canonical enhanced white noise in $\fX_\fn^\alpha$: 

\begin{theorem}[See Theorem \ref{theorem:convergence_in_enhanced_space_to_white_noise} and  \ref{obs:shift_enhanced_white_noise}]
Let $\alpha \in (-\frac43,-1)$. 
For all $y\in \R^2$ and $L>0$ there exists a canonical $\bxi_L^y = (\xi_L^y,\Xi_L^y) \in \fX_\fn^\alpha(y+Q_L)$ such that $\xi_L^y$ is a white noise (in the sense that is described in that theorem). 
\end{theorem}

We will write $\bxi_L= \bxi_L^0, \xi_L=\xi_L^0, \Xi_L=\Xi_L^0$ and for $\beta>0$
\begin{align*}
\blambda_n(y+Q_L,\beta) = \lambda_n(y+Q_L, (\beta \xi_L^y, \beta^2 \Xi_L^y)), 
\qquad \blambda_n(y+Q_L)=\blambda_n(y+Q_L,1). 
\end{align*}

Now we have the framework set and can get to the key ingredients, of which two are given in Section \ref{section:scaling_and_translation}, the scaling and translation properties:

\begin{obs}
\label{obs:scaling_and_translation}
\begin{enumerate}
\item 
\label{item:scaling}
(Lemma \ref{lemma:scaling_eigenvalues_with_white_noise}) For 
$L,\beta,\epsilon>0$,  
\bfold 
\blambda_n(Q_L,\beta) \overset{d}{=} 
 \tfrac{1}{\epsilon^2} 
 \blambda_n (Q_{\frac{L}{\epsilon}}, \epsilon \beta  )
+ \tfrac{1}{2\pi} \log \epsilon.
\efold 
\item 
\label{item:translation}
(Lemma \ref{lemma:translation_of_eigenvalues}) 
For $y\in \R^2$ and $L,\beta>0$, 
\bfold 
\blambda_n(Q_L, \beta  )
 \overset{d}{=} \blambda_n(y+Q_L, \beta  ). 
\efold 
Moreover, if $y +Q_L^\circ \cap Q_L^\circ = \emptyset$, then $\blambda_n(Q_L, \beta )
$ and $\blambda_n(y+Q_L, \beta  )$ are independent. 
\end{enumerate}
\end{obs}

In  \cite[Proposition 1]{GaKo00} 
and 
\cite[Lemma 4.6]{BiKo01LTT}
 the principal eigenvalue on a large box is bounded by maxima of principal eigenvalues on smaller boxes. 
We extend these results from smooth potentials to enhanced potentials:

\begin{theorem}[Consequence of Theorem \ref{theorem:lower_and_upper_bounds_white_noise}\footnote{In this statement we have choosen $a=\frac12 r$.}]
\label{theorem:comparing_eigenvalues_different_boxes_summary}
There exists a $K>0$ such that for all $\epsilon>0$ and $L>r\ge 1$, the following inequalities hold almost surely 
\begin{align*}
\notag
 \max_{k\in \N_0^2, |k|_\infty< \frac{L}{r} - 1 }
\blambda(rk + & Q_{r}, \epsilon ) 
 \le 
\blambda( Q_L, \epsilon  )  \le \max_{k\in \N_0^2, |k|_\infty< \frac{L}{r} +1}
\blambda(rk + Q_{\frac32 r}, \epsilon ) + \tfrac{4K}{r^2}. 
\end{align*}
Moreover, for $n\in\N$ and $L>r \ge 1$; 
if $x,y \in \R^2$ and $x+Q_r \subset y+Q_L$, then $\blambda_n(x+Q_r,\epsilon) \le \blambda_n(y+Q_L,\epsilon)$; 
if $y,y_1,\dots,y_n \in \R^2$ are such that $(y_i + Q_r)_{i=1}^n$ are pairwise disjoint subsets of $y+Q_L$, then almost surely
$ 
\blambda_n(y+Q_L,\epsilon) \ge \min_{i\in \{1,\dots,n\}} \blambda(y_i +Q_r ,\epsilon). 
$ 
\end{theorem}

\begin{calc}
Note that $rk + [0,r]^d$ is indeed a subset of $[0,L]^d$ for $k\in \N_0^d$ if (and only if) $|k|_\infty < \frac{L}{r} - 1$. 
\end{calc}

Another important tool that we prove is the large deviations of the eigenvalues, which --by the contraction principle and continuity of the eigenvalues in terms of its enhanced distribution-- is a consequence of the large deviations of $(\sqrt{\epsilon}\xi_L, \epsilon \Xi_L)$, proven in Section \ref{section:ldp_enhancement}.

\begin{theorem}[See Corollary \ref{cor:ldp_eigenvalues}]
\label{theorem:LDP_summary}
$\blambda_n(Q_L,\sqrt{\epsilon})= \lambda_n(Q_L, (\sqrt{\epsilon}\xi_L, \epsilon \Xi_L) )$ satisfies the large deviation principle with rate $\epsilon$ and rate function $I_{L,n} : \R \rightarrow [0,\infty]$ given by 
\begin{align*}
I_{L,n}(x) = \inf_{\overset{V\in L^2(Q_L)}{ \lambda_n(Q_L,V) = x}} \tfrac12 \|V\|_{L^2}^2. 
\end{align*}
\end{theorem}

In Section \ref{section:ldp_results} we study infima over the large deviation rate function over half-lines, in terms of which the almost sure limit of the eigenvalues will be described: 

\begin{theorem}
\label{theorem:convergence_infima_rate_function}
There exists a $C>0$ such that for all $n\in\N$,   
\bfold
\varrho_n 
= \inf_{L>0} \inf I_{L,n}[1,\infty) 
= \lim_{L \rightarrow \infty } \inf I_{L,n}[1,\infty) 
>C
\efold 
and 
\begin{align}
\label{eqn:2_div_varrho_n_var_formula}
\frac{2}{\varrho_n} 
& = 4
 \sup_{ \overset{ V\in C_c^\infty(\R^2)}{ \|V\|_{L^2}^2 \le 1} } 
\sup_{ \overset{ F \sqsubset C_c^\infty(\R^2) }{ \dim F = n} } 
\inf_{ \overset{ \psi \in F }{ \|\psi \|_{L^2}^2 =1} } 
\int_{\R^2} - | \nabla \psi|^2 + V \psi^2  .
\end{align}
Moreover, 
\begin{align}
\label{eqn:2_div_varrho_1_var_formula_and_lady_constant}
\frac{2}{\varrho_1} 
= 4
\sup_{ \overset{ \psi \in C_c^\infty(\R^2) }{ \|\psi \|_{L^2}^2 =1} } 
\|\psi\|_{L^4}^2 - 
\int_{\R^2} | \nabla \psi|^2 = \chi, 
\end{align}
where 
$\chi$ is the smallest $C>0$ such that $\|f\|_{L^4}^4 \le C \|\nabla f\|_{L^2}^2 \|f\|_{L^2}^2$ for all $f\in H^1(\R^2)$ (this is Ladyzhenskaya's inequality). 
\end{theorem}

Using the scaling and translation properties of \ref{obs:scaling_and_translation}, the comparison of the eigenvalue with maxima of eigenvalues of smaller boxes in Theorem \ref{theorem:comparing_eigenvalues_different_boxes_summary} and the large deviations in Theorem \ref{theorem:LDP_summary} we obtain the following tail bounds in Section \ref{section:tail_bounds}.

\begin{theorem}
\label{theorem:tail_bounds}
Let $K>0$ be as in Theorem \ref{theorem:lower_and_upper_bounds_white_noise}.
Let $r,\beta>0$. 
We will abbreviate $I_{r,1}$ by $I_r$. 
For all $\mu > \inf I_{r}(1,\infty)$ 
and $\kappa < \inf I_{\frac32 r}[1-\frac{16K}{r^2})$ 
there exists an $M>0$ such that for all $L,x>0$ with $L\sqrt{x} >M$ 
\begin{align}
\label{eqn:tail_bound_le_x}
\P\left( \blambda( Q_{L},\beta) \le x  \right) 
& \le 
\exp\left(-\frac{ e^{2\log L-\frac{\mu}{\beta^2} x } x }{2r^2}   \right), \\
\label{eqn:tail_bound_ge_x}
\P\left( \blambda( Q_{L},\beta) \ge x  \right) 
& \le 
 \tfrac{ 2 }{ r^2}   x e^{2\log L -\frac{\kappa}{\beta^2} x }.
\end{align}
\end{theorem}

Using the tail bounds and the limit in Theorem \ref{theorem:convergence_infima_rate_function} we obtain our main result by a Borel-Cantelli argument and the `moreover' part of Theorem \ref{theorem:comparing_eigenvalues_different_boxes_summary}. For the details  see Section \ref{section:tail_bounds}. 

\begin{theorem}
\label{theorem:as_convergence_all_eigenvalues}
Let $\I \subset (1,\infty)$ be an unbounded countable set, and let $\beta>0$. 
For $L\in \I$ let $y_L \in \R^2$ be such that $y_r + Q_r \subset y_L + Q_L$ for $r,L \in \I$ with $L>r$.  Then for $n\in\N$
\begin{align*}
\lim_{ \overset{  L \in \I }{ L \rightarrow \infty}} \frac{ \blambda_n(y_L+Q_L,\beta)}{\log L} = \frac{2\beta^2 }{\varrho_1} 
= \beta^2 \chi
 \qquad a.s.
\end{align*}
\end{theorem}

\section{Proof of Theorem \ref{theorem:tail_bounds} and Theorem \ref{theorem:as_convergence_all_eigenvalues}}
\label{section:tail_bounds}

In this section we prove Theorem \ref{theorem:tail_bounds} and Theorem \ref{theorem:as_convergence_all_eigenvalues} by using \ref{theorem:summary_of_dirichlet_summary}--\ref{theorem:convergence_infima_rate_function}.

\begin{obs}
\label{lemma:probability_estimates}
Let $K>0$ be as in Theorem \ref{theorem:comparing_eigenvalues_different_boxes_summary}.
To simplify notation we take $\beta =1$. 
By consecutively applying 
the scaling in \ref{obs:scaling_and_translation}\ref{item:scaling},
the bounds in Theorem \ref{theorem:comparing_eigenvalues_different_boxes_summary} and 
then the independence and translation properties in \ref{obs:scaling_and_translation}\ref{item:translation},
we get for
$L,r,\epsilon>0$ with  $\frac{L}{\epsilon}> r\ge1 $ 
\begin{align}
\P\left( \epsilon ^2 \blambda( Q_{L}) \le 1 \right) 
\notag & =
\P\left( \blambda ( Q_{\frac{L}{\epsilon }}, \epsilon ) + \tfrac{\epsilon^2}{2\pi}\log \epsilon  \le 1 \right)  \\
\notag & \le
\P\left( 
 \max_{k\in \N_0^2, |k|_\infty< \frac{L}{\epsilon r} - 1 }
\blambda(rk + Q_{r}, \epsilon ) 
  \le 1 -\tfrac{\epsilon^2}{2\pi}\log \epsilon  \right) \\
 \cand \begin{calc} 
 \notag\le 
\prod_{k\in \N_0^2, |k|_\infty< \frac{L}{\epsilon r} - 1 }
\P\left( 
\blambda(rk + Q_{r}, \epsilon ) 
 \le 1 -\tfrac{\epsilon^2}{2\pi}\log \epsilon  \right) 
\end{calc} \cnewline
\label{eqn:bound_prob_le_1}
& = \P\left( 
\blambda( Q_{r}, \epsilon ) \le 1 - \tfrac{\epsilon^2}{2\pi}\log \epsilon  \right)^{\# \{k\in \N_0^2 : |k|_\infty< \frac{L}{\epsilon r} - 1\}}, 
\end{align}
and similarly 
\begin{align}
\notag 
\P &  \left( \epsilon ^2 \blambda( Q_{L}) \ge  1  \right) 
= \P\left( \blambda ( Q_{\frac{L}{\epsilon }}, \epsilon ) + \tfrac{\epsilon^2}{2\pi}\log \epsilon  \ge 1 \right) 
\\
\notag & \le 
\P\left( 
\max_{k\in \N_0^2, |k|_\infty< \frac{L}{\epsilon r} +1}
\blambda(rk + Q_{\frac32 r}, \epsilon  ) + \tfrac{4K}{r^2}
 + \tfrac{\epsilon^2}{2\pi}\log \epsilon  \ge 1 \right) \\
\cand \begin{calc}
\notag  \le 
\sum_{k\in \N_0^2, |k|_\infty \le \frac{L}{\epsilon r} +1}
\P\left( 
\blambda(rk + Q_{\frac32 r}, \epsilon )  \ge 1  - \tfrac{4K}{r^2}
 - \tfrac{\epsilon^2}{2\pi}\log \epsilon \right)
\end{calc} \cnewline
\label{eqn:bound_prob_ge_1} & \le 
\# \{ k\in \N_0^2 : |k|_\infty< \tfrac{L}{\epsilon r} +1\}
\P\left(
\blambda( Q_{\frac32 r}, \epsilon )  \ge 1  - \tfrac{4K}{r^2}  - \tfrac{\epsilon^2}{2\pi}\log \epsilon  \right). 
\end{align}
\begin{calc}
As $\#\{ k \in \N_0^2 : |k|_\infty \le n\} = (n+1)^2$ for $n\in\N$, we have 
\begin{align*}
 \lim_{M\rightarrow \infty} \frac{\# \{ k\in \N_0^2 : |k|_\infty < M  \pm 1\}}{ M^2 }  
  =1.
\end{align*}
\end{calc}
Observe that there exists an $M>0$ such that for all $L,r,\epsilon>0$ with $\frac{L}{\epsilon r} >M$
\begin{align*}
\tfrac12 (\tfrac{L}{\epsilon  r})^2 \le \# \{ k\in \N_0^2 : |k|_\infty < \tfrac{L}{\epsilon r} \pm 1\} 
\le 2 (\tfrac{L}{\epsilon  r})^2. 
\end{align*}
\end{obs}

By combining the above observations we have obtained the following.

\begin{lemma}
\label{lemma:prob_estimate_from_above_large_box_by_smaller}
Let $K>0$ be as in Theorem \ref{theorem:lower_and_upper_bounds_white_noise}. Let $\beta >0$. 
There exists an $M>1$ such that for all $L,r,\epsilon>0$ 
with $\frac{L}{\epsilon} > Mr > r \ge 1$
\begin{align}
\label{eqn:p_sq_lambda_le_q}
\P\left( \epsilon^2 \blambda( Q_{L},\beta) \le 1  \right) 
& \le 
\P\left( 
\blambda( Q_{r}, \epsilon \beta ) \le 1  - \tfrac{\epsilon^2}{2\pi}\log \epsilon  \right)^{\frac12 \left( \frac{L }{  \epsilon r}  \right)^2 }, \\
\label{eqn:p_sq_lambda_qe_q}
\P\left( \epsilon ^2 \blambda( Q_{L},\beta) \ge  1  \right) 
& \le 
2 \left( \tfrac{L  }{\epsilon r}  \right)^2
\P\left( 
\blambda( Q_{\frac32 r}, \epsilon \beta )   \ge 1 - \tfrac{4K}{r^2}   - \tfrac{\epsilon^2}{2\pi}\log \epsilon \right).
\end{align}
\end{lemma}

\begin{obs}
Let $r>0$.
Let us now use the large deviation principle in Corollary \ref{cor:ldp_eigenvalues}.
First, observe that as $\lim_{\epsilon\downarrow 0} \tfrac{\epsilon^2}{2\pi}\log \epsilon =0$,  also $\blambda(Q_r, \epsilon \beta) +\tfrac{\epsilon^2}{2\pi}\log \epsilon$ satisfies the large deviation principle with the rate function $\beta^{-2} I_{r,n}$ (by exponential equivalence, see \cite[Theorem 4.2.13]{DeZe10}). 
Hence for all $\mu > \inf I_{r,n}(1,\infty)$ and $\kappa < \inf I_{\frac32 r,n}[1-\frac{4K}{r^2},\infty)$ there exists a $\epsilon_0$ such that for $\epsilon \in (0,\epsilon_0)$ we have the following bound on the probability appearing in \eqref{eqn:p_sq_lambda_le_q} (using that $1-x \le e^{-x}$ for $x\ge 0$):
\begin{align}
\label{eqn:ldp_bound_le}
& \P\left( 
\blambda( Q_{r}, \epsilon \beta ) \le 1 - \tfrac{\epsilon^2}{2\pi}\log \epsilon  \right) 
\le 1- e^{-\frac{\mu}{\epsilon^2 \beta^2 }} \le e^{- e^{-\frac{\mu}{\epsilon^2 \beta^2 }}},\\
\label{eqn:ldp_bound_ge}
& \P\left( 
\blambda( Q_{\frac32 r}, \epsilon \beta ) 
\ge 1 - \tfrac{4K}{r^2} - \tfrac{\epsilon^2}{2\pi}\log \epsilon  \right) 
\le e^{-\frac{\kappa}{\epsilon^2 \beta^2 }} .
\end{align}
\end{obs}

\begin{proof}[Proof of Theorem \ref{theorem:tail_bounds}]
This now follows by Lemma \ref{lemma:prob_estimate_from_above_large_box_by_smaller} and the bounds \eqref{eqn:ldp_bound_le} and \eqref{eqn:ldp_bound_ge}. 
\begin{calc}
We obtain 
\begin{align}
\notag 
& \forall r >0 
\ \forall \mu > \inf I_r(1,\infty)
\ \exists M >0 
\ \forall L,x> 0  \mbox{ with } \tfrac{L\sqrt{x}}{r}>M: \\
\label{eqn:tail_bound_le_x_for_integer_restriction}
& \P\left(  \blambda( Q_{L}) \le x  \right) 
 \le 
e^{- \frac{L^2 x}{ 2 r^2}  e^{-\mu x} }
= e^{- \frac{ x}{ 2 r^2}  e^{2\log L -\mu x} }
, \\
\notag & \forall L,r >0 
\ \forall \kappa < \inf I_{\frac32 r}[1- \tfrac{4K}{r^2},\infty)
\ \exists M >0 
\ \forall L,x> 0 \mbox{ with }  \tfrac{L\sqrt{x}}{r} >M: \\
\label{eqn:tail_bound_ge_x_for_integer_restriction}
& \P\left(  \blambda( Q_{L}) \ge x  \right) 
 \le 
2 \tfrac{L^2 x}{r^2} e^{-\kappa x}. 
\end{align}
\end{calc}
\end{proof}




First we prove the convergence of the eigenvalues along the set $\{2^m: m\in \N\}$, before proving Theorem \ref{theorem:as_convergence_all_eigenvalues}.
Observe that in Theorem \ref{theorem:as_convergence_eigenvalues_dyadics}, contrary to Theorem \ref{theorem:as_convergence_all_eigenvalues}, we do not impose a condition on the sequence $(y_m)_{m\in\N}$. 

\begin{theorem}
\label{theorem:as_convergence_eigenvalues_dyadics}
Let $n\in \N$ and $\beta>0$. 
For any sequence $(y_m)_{m\in \N}$ in $\R^2$. 
\begin{align*}
\lim_{ m \in \N, m\rightarrow \infty} \frac{ \blambda_n(y_m+Q_{2^m},\beta)}{\log 2^m} = \frac{2\beta^2 }{\varrho_1} 
= 4\beta^2 
 \sup_{ \overset{ V\in C_c^\infty(\R^2)}{ \|V\|_{L^2}^2 \le 1} } 
\sup_{ \overset{ \psi \in C_c^\infty(\R^2)  }{ \|\psi \|_{L^2}^2 =1} } 
\int_{\R^2} - | \nabla \psi|^2 + V \psi^2 
 \quad a.s.
\end{align*} \towil{not necessary to repeat the var formula here, possibly suggest by proofs AOP to take out}
\end{theorem}
\begin{proof}
Without loss of generality we may assume $y_m=0$ for all $m\in \N$ and take $\beta = 1$. 

$\bullet$ First we prove the convergence of the principal eigenvalue, i.e., we consider $n=1$. 
Let $p,q \in \R$ be such that $p<\frac{2}{\varrho_1}<q$. We show that 
\begin{align*}
\liminf_{m\rightarrow \infty} \frac{\blambda(Q_{2^m})}{\log 2^m} > p \quad \mbox{a.s.}, \qquad \qquad
\limsup_{m\rightarrow \infty} \frac{\blambda(Q_{2^m})}{\log 2^m} < q \quad \mbox{a.s.}
\end{align*}
By the lemma of Borel-Cantelli it is sufficient to show that 
\begin{align*}
& \sum_{m=1}^\infty  \P \left[ \frac{\blambda(Q_{2^m})}{\log 2^m} < p \right] < \infty, 
& \sum_{m=1}^\infty  \P \left[ \frac{\blambda(Q_{2^m})}{\log 2^m} > q \right] < \infty. 
\end{align*}
By Lemma \ref{lemma:infima_rate_function_translated_by_constant}
\begin{align*}
\lim_{r\rightarrow \infty} \inf I_{r} (1, \infty) 
= \lim_{r\rightarrow \infty} \inf I_{\frac32 r} [1 -\tfrac{16K}{r^2}, \infty) 
= \varrho_1.
\end{align*}
Let $r>0$ be large enough such that 
\begin{align*}
p \inf I_{r}(1,\infty) <2 < q \inf I_{\frac32 r} [1- \tfrac{16K}{r^2},\infty).
\end{align*}
Let $\mu > \inf I_{r}(1,\infty)$ be such that $p \mu <2$ and 
$\kappa < \inf I_{\frac32 r} [1- \frac{16K}{r^2},\infty) $ be such that $q \kappa >2$.
By Theorem \ref{theorem:tail_bounds} for $M \in \N$ large enough \towil{the two $8$'s should be $2$'s}
\begin{align*}
\sum_{m=M}^\infty  \P \left[ \frac{\blambda(Q_{2^m})}{\log 2^m }  < p \right] 
&\le \sum_{m=M}^\infty  2^{- m  \frac{p 2^{(2-p \mu)m } }{2r^2}   } <\infty, 
\end{align*}
which is finite because  $\frac{p 2^{(2-p \mu)m} }{8r^2}>1$ for large $m$, as $2-p\mu >0$. Also
\begin{align*}
\sum_{m=M}^\infty  \P \left[ \frac{\blambda(Q_{2^m})}{\log 2^m }  > q \right]
& \le 
\sum_{m=M}^\infty \frac{   2 m \log 2}{r^2} 2^{(2-\kappa q)m}, 
\end{align*}
which is  finite as $2-\kappa q< 0$ (and because $2^{-\alpha m} m \rightarrow 0$ for $\alpha>0$). 

$\bullet$ Let $n\in\N$. 
Let us first observe that as $\blambda_n(Q_{2^m})\le \blambda(Q_{2^m})$, we have  $\limsup_{ m \rightarrow \infty} \frac{ \blambda_n(Q_{2^m})}{\log 2^m} \le  \frac{2}{\varrho_1} $. 
Let $x_1,\dots,x_n \in Q_{2^n}$ be such that $(x_i + Q_1)_{i=1}^n$ are disjoint. By 
Theorem \ref{theorem:comparing_eigenvalues_different_boxes_summary} we obtain almost surely
\begin{align*}
\liminf_{m\rightarrow \infty} \frac{\blambda_n(Q_{2^{n+m}})}{\log 2^{n+m} } 
\ge 
\min_{i\in \{1,\dots,n\}}
\lim_{m\rightarrow \infty} 
\frac{\blambda(2^m x_i + Q_{2^m})}{\log 2^{n} +\log 2^{m} }  = \frac{2}{\varrho_1}.
\end{align*}
\end{proof}

\begin{proof}[Proof of Theorem \ref{theorem:as_convergence_all_eigenvalues}]
The condition on $y_L$ is assumed in order to have the monotonicity of $L\mapsto \blambda_n(y_L)$ on $\I$. 
Therefore and for convenience, we  assume $y_L=0$ for all $L \in \I$.
Also we take $\beta =1$. 
Write $s= \frac{2}{\varrho_1}$. 
Let $\epsilon \in (0,s)$. 
By Theorem \ref{theorem:as_convergence_eigenvalues_dyadics} there exists an $M$ such that for all $m\ge M$ 
\begin{align*}
(\log  2^m) ( s- \epsilon) 
\le  \blambda_n(Q_{2^m})
\le (\log 2^m) (s+\epsilon) \qquad \mbox{a.s.}
\end{align*}
Let $a\in [1,2]$, then almost surely, as $L \mapsto \blambda_n( Q_L)$ is an increasing function 
\begin{align*}
 (\log a 2^{m-1}) (s- \epsilon) 
& \le \blambda_n(Q_{2^m}) \le \blambda_n ( Q_{a2^m}) \le \blambda_n( Q_{2^{m+1}}) 
 \le (\log a 2^{m+1})( s+ \epsilon), 
\end{align*}
and 
\begin{align*}
\left( 1- \frac{\log 2}{\log (2^m)}\right) (s- \epsilon)
& \le 
\left( 1- \frac{\log 2}{\log (a2^m)}\right) (s- \epsilon) \\
& \le 
\frac{ \blambda_n(Q_{a 2^m})}{ \log (a 2^m)}
\le \left(1+ \frac{\log 2}{\log (2^{m})}\right) (s+\epsilon). 
\end{align*}
From this it follows that almost surely 
\bfold
\lim_{L \in \I , L\rightarrow \infty} \frac{ \blambda_n (Q_{L})}{ \log (L)} =s .
\efold
\end{proof}

\section{Dirichlet and Neumann Besov spaces, para- and resonance  products}
\label{section:dirichlet_besov}

Let $d\in \N$. 
Let $L>0$. 
We will first introduce Dirichlet and Neumann spaces on $Q_L=[0,L]^d$. 
In order to do this we use 3 different bases of $L^2([0,L]^d)$, one standard (the $e_k$'s), one as an underlying basis for Dirichlet spaces (the $\fd_k$'s) and one as an underlying basis for Neumann spaces (the $\fn_k$'s). 
After defining these spaces (in Definition \ref{def:Dir_and_Neu_Besov}) we prove a few results that compare Besov and Sobolev spaces. 
Later, in Definition \ref{def:spaces_on_general_boxes} we show how to generalize this to spaces on general boxes of the form $\prod_{i=1}^d [a_i, b_i]$. 
Then we present bounds on Fourier multipliers (Theorem \ref{theorem:schauder}) and define para- and resonance products (Definition \ref{definition:paraproducts_odd_with_even}) and state their Bony estimates (Theorem \ref{theorem:bony_estimates}). 

In the following we will introduce some notation. 
For $\fq \in \{-1,1\}^d$ and $x\in \R^d$ we use the following short hand notation ($\fq \circ x$ is known as the Hadamard product) 
\begin{align*}
(\prod \fq) = \prod_{i=1}^d \fq_i, \qquad \fq\circ x = (\fq_1 x_1 ,\dots, \fq_d x_d).
\end{align*}
We call a function $f: [-L,L]^d \rightarrow \C$ \emph{odd} if $f(x) = (\prod \fq) f( \fq \circ x) $ for all $\fq\in \{-1,1\}^d$, and similarly we call $f$ \emph{even} if $f(x) = f(\fq \circ x)$ for all $\fq \in \{-1,1\}^d$. 
For any $f: [0,L]^d \rightarrow \C$ 
  we write $\tilde f: [-L,L]^d \rightarrow \C$ for its odd extension 
(the $\sim$ notation is taken as it looks like the graph of an odd function)  
  and $\overline f : [-L,L]^d \rightarrow \C$ for its even extension (similarly, the notation -- is taken as it looks like the graph of an even function), i.e., for the functions that satisfy 
\begin{align*}
\tilde f (\fq \circ x) = (\prod \fq) f(x), \quad \overline f (\fq \circ x ) = f(x) \quad \mbox{ for all } x\in [0,L]^d, \fq\in \{-1,1\}^d. 
\end{align*}


If a function $f: [-L,L]^d \rightarrow \C$ is \emph{periodic}, 
which means that $f(y,L) = f(y,-L)$ and $f(L,y)= f(-L,y)$ for all $y \in [-L,L]$, 
then it can be extended periodically on $\R^d$ (with period $2L$) we will also consider it to be a function on the domain $\T_{2L}^d$. 
Note that if $f$ is periodic and odd, then $f=0$ on $\partial [0,L]^d$. 
\begin{calc}
Indeed, as $f$ is odd we have $f(x_1,x_2) =-  f( -x_1, x_2)$ from which it follows that $f(x_1,x_2)=0$ in case $x_1=0$. 
As $f$ is periodic, we have $f(L,x_2) = f(-L,x_2)$ so that combined with the above rule we see that $f(x_1,x_2)=0$ also in case $x_1=-L$ or $x_1=L$. 
\end{calc}

\bigskip

For $k = (k_1,\dots, k_d) \in \N_0^d$ let $\nu_k = 2^{-\frac12 \# \{ i : k_i =0\} }$ and write $\fd_{k,L}$ and $\fn_{k,L}$ or simply $\fd_k$ and $\fn_k$ for the functions $ [0,L]^d \rightarrow \C$ and
$e_{k,2L}$ or simply $e_k$ for the function $[-L,L]^d \rightarrow \C$ given by
\begin{align}
\label{eqn:formula_d_k}
\fd_{k,L}(x) & = \fd_k(x) = (\tfrac{2}{L})^{\frac{d}{2}} \prod_{i=1}^d \sin( \tfrac{\pi}{L} k_i x_i ), \\
\label{eqn:formula_f_k}
\fn_{k,L}(x) & = \fn_k(x) = \nu_k (\tfrac{2}{L})^{\frac{d}{2}} \prod_{i=1}^d \cos( \tfrac{\pi}{L} k_i x_i ), \\
e_{k,2L}(x) &= e_k(x) = (\tfrac{1}{2L})^{\frac{d}{2}} e^{\frac{\pi \i}{L} \langle k, x \rangle }. 
\end{align}
Note that $\tilde \fd_k(x)$ equals the right-hand side of \eqref{eqn:formula_d_k} and $\overline \fn_k(x)$ equals the right-hand side of \eqref{eqn:formula_f_k} for $x\in [-L,L]^d$, 
so that $\tilde \fd_k$ and $\overline \fn_k$ are elements of $C^\infty(\T_{2L}^d)$. 
We can also write $\tilde \fd_k$ and $\overline \fn_k$ as follows 
\begin{align}
\label{eqn:d_k_in_terms_of_e_k}
\tilde \fd_k(x) &=  (\tfrac{2}{L})^{\frac{d}{2}} \prod_{i=1}^d \frac{e^{ \frac{\pi \i }{L} k_i x_i } - e^{ -  \frac{\pi \i }{L} k_i x_i }}{2\i }
 =  (-\i)^d 
\sum_{\fq\in \{-1,1\}^d}   (\prod \fq) e_{\fq\circ k}(x), \\
\label{eqn:f_k_in_terms_of_e_k}
\overline \fn_k(x) &= \nu_k  (\tfrac{2}{L})^{\frac{d}{2}} \prod_{i=1}^d \frac{e^{ \frac{\pi \i }{L} k_i x_i } + e^{ -  \frac{\pi \i }{L} k_i x_i }}{2}
 = \nu_k
\sum_{\fq\in \{-1,1\}^d}    e_{\fq\circ k}(x).
\end{align}
\begin{calc}
\begin{align*}
(\tfrac{2}{L})^{\frac{d}{2}} \prod_{i=1}^d \frac{e^{ \frac{\pi \i }{L} k_i x_i } - e^{ -  \frac{\pi \i }{L} k_i x_i }}{2\i }
= (\tfrac{1}{2L})^{\frac{d}{2}} \prod_{i=1}^d \sum_{q_i \in \{-1,1\}} \frac{q_i}{\i }  e^{ \frac{\pi \i }{L} q_i k_i x_i } . 
\end{align*}
\end{calc}
For an integrable function $f: \T_{2L}^d \rightarrow \C$ its $k$-th Fourier coefficient is defined by 
\begin{align*}
 \cF f(k) = \langle f, e_k \rangle 
=  \frac{1}{(2L)^\frac{d}{2}} \int_{\T_{2L}^d} f(x) e^{- \frac{\pi \i }{L} \langle  k, x \rangle } \dd x \qquad (k\in\Z^d). 
\end{align*}

\begin{obs}
\label{obs:odd_and_even_functions_and_their_fourier_transform_and_inner_prod}
It is not difficult to see that for $\varphi ,\psi \in L^2([0,L]^d)$, the following equalities hold:
\begin{align}
\label{eqn:hat_of_tilde_is_odd}
& \cF(\tilde \varphi)(k) = (\prod \fq) \cF(\tilde \varphi)( \fq \circ k ) \mbox{ for all } k\in \Z^d, \fq\in \{-1,1\}^d, \\
\label{eqn:hat_of_tilde_on_boundary_zero}
& \cF( {\tilde  \varphi})(k) =0 \mbox{ for all } k \in \Z^d \mbox{ with } k_i=0 \mbox{ for some } i, \\
\label{eqn:hat_of_overline_is_even}
& \cF(\overline \varphi)(k) =  \cF(\overline \varphi)( \fq \circ k ) \mbox{ for all } k\in \Z^d, \fq\in \{-1,1\}^d, \\
&
\label{eqn:inner_product_of_tildes_vs_without_and_with_overline}
\langle \tilde \varphi, \tilde \psi \rangle_{L^2[-L,L]^d} = 2^d \langle \varphi, \psi \rangle_{L^2[0,L]^d} = 
\langle \overline \varphi, \overline \psi \rangle_{L^2[-L,L]^d} 
, \\
& 
\label{eqn:innerprod_d_k_versus_e_k}
\langle \varphi,\fd_k \rangle 
\begin{calc2} 
= \frac{1}{2^d} \langle \tilde \varphi , \tilde \fd_k \rangle  
= 
\Big(\frac{\i}{2} \Big)^{d} 
\sum_{\fq\in \{-1,1\}^d}   (\prod \fq) \cF( \tilde \varphi)(\fq\circ k) 
\end{calc2}
= \i^{d} \cF(\tilde \varphi)(k) \mbox{ for all } k \in \N^d, \\
& 
\label{eqn:innerprod_f_k_versus_e_k}
\langle \varphi, \fn_k \rangle 
\begin{calc2}
=  \frac{1}{2^d}\langle \overline \varphi , \overline \fn_k \rangle  
= 
 \frac{1}{2^d}
\sum_{\fq\in \{-1,1\}^d}  \cF( \overline \varphi)(\fq\circ k) 
\end{calc2}
= \cF(\overline \varphi)(k) \mbox{ for all } k \in \N_0^d.
\end{align}
\end{obs}

\begin{obs}
\label{obs:derivatives_and_fourier_transforms}
By partial integration one obtains that 
\bfold 
\cF( \partial^\alpha f) (k) = ( \tfrac{\pi \i }{L} k )^{\alpha} \cF(f)(k). 
\efold 
So that 
$
\cF( \Delta f) (k) = - |\tfrac{\pi }{L}k |^2 \cF(f) (k). 
$ 
Consequently 
$\langle \Delta f, \fd_k \rangle = - |\tfrac{\pi }{L}k |^2  \langle f, \fd_k \rangle$ and 
$\langle \Delta f, \fn_k \rangle = - |\tfrac{\pi }{L}k |^2  \langle f, \fn_k \rangle$. 
This will be used later to define $(a- \Delta)^{-1}$ for $a\in \R\setminus \{0\}$. 

 \begin{calc}
 Indeed, by partial integration one has 
\begin{align*}
\cF( \partial_i f) (k) 
\cand \begin{calc} =
\frac{1}{(2L)^{\frac{d}{2}}} \int_{\T_{2L}^d} e^{-\frac{\pi \i }{L} \langle k, x \rangle} \partial_i f(x) \dd x
\end{calc}\cnewline
\cand \begin{calc} =
- \frac{1}{(2L)^{\frac{d}{2}}} \int_{\T_{2L}^d}   f(x) \partial_i
e^{-\frac{\pi \i }{L} \langle k, x \rangle}
 \dd x
\end{calc}\cnewline
\cand \begin{calc} =
\frac{\pi \i }{L} k_i \frac{1}{(2L)^{\frac{d}{2}}} \int_{\T_{2L}^d}   f(x) 
e^{-\frac{\pi \i }{L} \langle k, x \rangle}
 \dd x.
\end{calc}
\end{align*}
\end{calc}

Moreover, from this one obtains that the spectrum of $-\Delta$ is given by $\{ \frac{\pi^2}{L^2} |k|^2 : k \in \Z^d\}$ and that every $e_k$ is an eigenvector.

\end{obs}

\begin{lemma}
\label{lemma:basis_d_k_and_f_k}
$\{ \fd_k : k \in \N^d\}$ and $\{ \fn_k : k\in \N_0^d\}$ form orthonormal bases for $L^2([0,L]^d)$. 
\end{lemma}
\begin{proof}
We leave it to the reader to check that those sets are orthonormal. 
Let $\varphi \in L^2([0,L]^d)$. 
By expressing $\tilde \varphi$ and $\overline \varphi$ in terms of the basis $\{e_k: k\in \Z^d\}$ and using \ref{obs:odd_and_even_functions_and_their_fourier_transform_and_inner_prod} one obtains $\tilde \varphi = \sum_{k\in \N^d} \langle  \varphi ,  \fd_k \rangle_{L^2[0,L]^2}  \tilde \fd_k$ and $\overline \varphi = \sum_{k\in \N_0^d} \langle  \varphi,  \fn_k \rangle_{L^2[0,L]^2} \overline \fn_k $. 
\end{proof}

\begin{calc}
Indeed 
\begin{align*}
\sum_{k\in \N^d} \langle  \varphi ,  \fd_k \rangle_{L^2[0,L]^2}  \tilde \fd_k
& = 
\sum_{k\in \N^d} \i^d \cF( \tilde \varphi) (k)  \tilde \fd_k \\
& = 
\sum_{k\in \N^d} \i^d (-\i)^d \sum_{\fq\in \{-1,1\}^d}   (\prod \fq) \cF( \tilde \varphi) (k)   e_{\fq\circ k}(x) \\
& = 
\sum_{k\in \N^d} \sum_{\fq\in \{-1,1\}^d}  \cF( \tilde \varphi) (\fq \circ k)   e_{\fq\circ k}(x) \\
& = 
\sum_{k\in \Z^d}  \cF( \tilde \varphi) (k)   e_{k}(x) 
= \tilde \varphi(x). 
\end{align*}
For the orthonormality, for $k, l \in \N^d$
\begin{align*}
2^d \langle \fd_k , \fd_l \rangle_{L^2([0,L]^d)}
= \langle \tilde \fd_k , \tilde \fd_l \rangle_{L^2(\T_{2L}^d)}
= \sum_{\fq,\fp \in \{-1,1\}^d} (\prod \fq \fp) \langle e_{\fq \circ k} , e_{\fp \circ l } \rangle 
= 
\begin{cases}
0 & k \ne l ,\\
2^d & k = l .
\end{cases}
\end{align*}
For $k,l \in \N_0^d$
\begin{align*}
2^d \langle \fn_k , \fn_l \rangle_{L^2([0,L]^d)}
= \langle \overline \fn_k , \overline \fn_l \rangle_{L^2(\T_{2L}^d)} 
& = \sum_{\fq,\fp \in \{-1,1\}^d} \langle e_{\fq \circ k} , e_{\fp \circ l } \rangle \\
& = 
\begin{cases}
0 & k \ne l, \\
2^d 2^{\# \{ i : k_i =0\}}  & k = l .
\end{cases}
\end{align*}
\end{calc}

\begin{definition}
We define the set of test functions on $[0,L]^d$ that oddly and evenly extend to smooth functions on $\T_{2L}^d$ (here $\cS(\T_{2L}^d) = C^\infty (\T_{2L}^d)$): 
\begin{align*}
\cS_0([0,L]^d)& := \{ \varphi \in C^\infty ([0,L]^d) : \tilde \varphi \in \cS(\T_{2L}^d)\}, \\
\cS_\fn([0,L]^d)& := \{ \varphi \in C^\infty ([0,L]^d) : \overline \varphi \in \cS(\T_{2L}^d)\}. 
\end{align*}
We equip $\cS_0([0,L]^d)$, $\cS_\fn([0,L]^d)$ and $\cS(\T_{2L}^d)$ with the Schwarz--seminorms. 
\begin{calc}
The Schwarz-seminorms $\|\cdot\|_{k,\cS}$ for $k\in \N_0$ are defined by 
\begin{align}
\label{eqn:schwarz_k_norm_without_square}
\|f\|_{k,\cS} := \sum_{\alpha: |\alpha|\le k} \sup_{x\in \R^d} |(1+|x|)^k D^{\alpha} f(x)| <\infty. 
\end{align}
\end{calc}
Note that\footnote{For the notation see Section \ref{subsection:notation}.} $C_c^\infty([0,L]^d)$ is a subset of both  $\cS_0([0,L]^d)$ and $\cS_\fn([0,L]^d)$.
\end{definition}

In the following theorem we state how one can represent elements of $\cS$, $\cS_0$ and $\cS_\fn$ and of $\cS'$, $\cS_0'$ and $\cS_\fn'$ in terms of series in terms of $e_k$, $\fd_k$ and $\fn_k$. 

\begin{theorem}
\label{theorem:representation_schwarz_func_and_distr}
\begin{enumerate}
\item 
\label{item:representation_of_functions_in_schwarz_spaces}
Every $\omega \in \cS(\T_{2L}^d)$, $\varphi \in \cS_0([0,L]^d)$ and $\psi \in \cS_\fn([0,L]^d)$ can be represented by 
\begin{align}
\label{eqn:rep_cS_0_function}
\omega = \sum_{k\in \Z^d} a_k e_k,
\quad 
\varphi = \sum_{k\in \N^d} b_k \fd_k, 
\quad 
\psi = \sum_{k\in \N_0^d} c_k \fn_k, 
\end{align}
where $(a_k)_{k\in \Z^d}$, $(b_k)_{k\in \N^d}$ and $(c_k)_{k\in \N_0^d}$ in $\C$ are such that 
\begin{align}
\label{eqn:growth_coefficients}
\forall n\in \N: \ 
\sup_{k\in\Z^d} (1+|k|)^n |a_k| < \infty, \ 
\sup_{k\in\N^d} (1+|k|)^n |b_k| < \infty, \ 
\sup_{k\in\N_0^d} (1+|k|)^n |c_k| < \infty , 
\end{align}
and $a_k = \langle \omega, e_k \rangle$, $b_k = \langle \varphi, \fd_k \rangle$ and $c_k = \langle \psi, \fn_k \rangle$. 

Conversely, if $(a_k)_{k\in \Z^d}$,  $(b_k)_{k\in \N^d}$ and $(c_k)_{k\in \N_0^d}$ satisfy \eqref{eqn:growth_coefficients} then $\sum_{k\in\Z^d} a_k e_k$, 
 $\sum_{k\in\N^d} b_k \fd_k$ and $\sum_{k\in\N_0^d} c_k \fn_k$ converge in $ \cS(\T_{2L}^d)$, $\cS_0([0,L]^d)$ and $\cS_\fn([0,L]^d)$, respectively. 

\item 
\label{item:representation_of_distributions_on_schwarz_spaces}
Every $w\in \cS'(\T_{2L}^d)$, $u\in \cS_0'([0,L]^d)$ and $v\in \cS_\fn'([0,L]^d)$ can be represented by 
\begin{align}
\label{eqn:rep_cS_0_distribution}
w = \sum_{k\in \Z^d} a_k e_k ,
\quad 
u = \sum_{k\in \N^d} b_k \fd_k, 
\quad 
v = \sum_{k\in \N_0^d} c_k \fn_k, 
\end{align}
where $(a_k)_{k\in \Z^d}$, $(b_k)_{k\in \N^d}$ and $(c_k)_{k\in \N_0^d}$ in $\C$ are such that 
\begin{align}
\label{eqn:growth_coefficients_distri}
\exists n\in \N:  \
\sup_{k\in\Z^d} \frac{ |a_k|}{ (1+|k|)^n}< \infty, \ 
\sup_{k\in\N^d}  \frac{|b_k|}{ (1+|k|)^n} < \infty, \ 
\sup_{k\in\N_0^d} \frac{|c_k|}{ (1+|k|)^n} < \infty , 
\end{align}
and $a_k = \langle w, e_k \rangle$, $b_k = \langle u, \fd_k \rangle$ and $c_k = \langle v, \fn_k \rangle$. 

Conversely, if $(a_k)_{k\in \Z^d}$, $(b_k)_{k\in \N^d}$ and $(c_k)_{k\in \N_0^d}$ satisfy \eqref{eqn:growth_coefficients_distri} then $\sum_{k\in\Z^d} a_k e_k$, $\sum_{k\in\N^d} b_k \fd_k$ and $\sum_{k\in\N_0^d} c_k \fn_k$ converge in $ \cS'(\T_{2L}^d)$, $\cS_0'([0,L]^d)$ and $\cS_\fn'([0,L]^d)$, respectively.

\end{enumerate}
\end{theorem}
\begin{proof}
Let $\omega \in \cS(\T_{2L}^d)$. 
As one has the relation $ \cF ( \Delta^n   \omega)(k) =  ( - \frac{\pi^2  }{L^2} |k|^2)^{n}  \cF(  \omega)(k)$ for all $n\in\N_0$,
we have \eqref{eqn:growth_coefficients} and 
$\sum_{k\in \Z^d: |k|\le N} \cF( \omega)(k) e_k \xrightarrow{N\rightarrow \infty}  \omega $ in $\cS(\T_{2L}^d)$, see also
\cite[Corollary 2.2.4]{StSh03}. 

Let $\varphi \in \cS_0([0,L]^d)$. 
Using the shown convergence above for $\omega = \tilde \varphi$, by
\eqref{eqn:d_k_in_terms_of_e_k}, \eqref{eqn:hat_of_tilde_is_odd},   \eqref{eqn:hat_of_tilde_on_boundary_zero} and \eqref{eqn:innerprod_d_k_versus_e_k}
\begin{align*}
\sum_{\substack{ k\in\Z^d \\ |k| \le N}} \cF(\tilde \varphi)(k) e_k
& = \sum_{\substack{ k\in\N^d \\ |k| \le N}} \sum_{\fq \in \{-1,1\}^d} \cF(\tilde \varphi)( \fq \circ k ) e_{\fq\circ k } 
= \sum_{\substack{ k\in\N^d \\ |k| \le N}} \langle \varphi, \fd_k \rangle \tilde \fd_k.
\end{align*}
Hence 
 $\sum_{k\in\N^d: |k| \le N} \langle \varphi, \fd_k \rangle \fd_k $ converges to $\varphi$ in $\cS_0([0,L]^d)$. 

Let $\psi \in \cS_\fn([0,L]^d)$. 
Using the shown convergence above for $\overline \psi$, by  \eqref{eqn:f_k_in_terms_of_e_k}, \eqref{eqn:hat_of_overline_is_even} and \eqref{eqn:innerprod_f_k_versus_e_k}
\begin{align*}
\sum_{\substack{ k\in\Z^d \\ |k| \le N}} \cF(\overline \psi )(k) e_k
& = \sum_{{\substack{ k\in\N_0^d \\ |k| \le N}}} 2^{-\# \{ i : k_i = 0\}} \sum_{\fq \in \{-1,1\}^d} \cF(\overline \psi)( \fq \circ k ) e_{\fq\circ k } 
= \sum_{\substack{ k\in\N_0^d \\ |k| \le N}} c_k \overline \fn_k.
\end{align*}
Hence 
 $\sum_{k\in\N^d: |k| \le N} \langle \psi, \fn_k \rangle \fn_k $ converges to $\psi$ in $\cS_\fn([0,L]^d)$. 

\ref{item:representation_of_distributions_on_schwarz_spaces} follows from \ref{item:representation_of_functions_in_schwarz_spaces}. 

\begin{calc}
Let $w \in \cS'(\T_{2L}^d)$. 
Then there exists an $n\in \N$ and a $C>0$ such that $|w(\varphi)| \le C \|\varphi\|_{m} = C \sum_{\alpha: |\alpha|\le n} \| D^\alpha \varphi \|_\infty$. 
As $D^\alpha e_k = (- \frac{\pi \i }{L})^\alpha  k^\alpha e_k$, \eqref{eqn:growth_coefficients_distri} follows. 
\end{calc}
\end{proof}

For $\varphi \in \cS_0([0,L]^d)$, note that $\tilde \varphi = \sum_{k \in \N^d} \langle \varphi, \fd_k \rangle \tilde \fd_k$.
Moreover, note that $\omega \in \cS(\T_{2L}^d)$ is odd if and only if $\langle \omega, e_{\fq \circ k} \rangle = (\prod \fq) \langle \omega, e_k \rangle$ for all $k \in \Z^d$ and $\fq \in \{-1,1\}^d$. This motivates the following definition. 

\begin{definition}
For $u\in \cS_0'([0,L]^d)$ we write $\tilde u$ for the distribution in $\cS'(\T_{2L}^d)$ given by $\tilde u = \sum_{k\in\N^d} \langle u, \fd_k \rangle \tilde \fd_k$.
For $v \in \cS_\fn'([0,L]^d)$ we write $\overline v$ for the distribution in $\cS'(\T_{2L}^d)$ given by $\overline v = \sum_{k\in\N_0^d}   \langle u,\fn_k \rangle \overline \fn_k$. 
A $w \in \cS'(\T_{2L}^d)$ is called \emph{odd} if $\langle w, e_{\fq \circ k} \rangle =  (\prod \fq) \langle w, e_k \rangle$ for all $k \in \Z^d$ and $\fq \in \{-1,1\}^d$. If instead $\langle w, e_{\fq \circ k } \rangle =  \langle w, e_k \rangle$ for all $k \in \Z^d$ and $\fq \in \{-1,1\}^d$, then $w$ is called \emph{even}. 

Note that $\tilde u$ is odd and $\overline v$ is even. 
\end{definition}

By \eqref{eqn:inner_product_of_tildes_vs_without_and_with_overline} and Theorem \ref{theorem:representation_schwarz_func_and_distr}, for $u \in \cS_0'([0,L]^d)$, $\varphi \in \cS_0([0,L]^d)$ and  $v \in \cS_\fn'([0,L]^d)$, $\psi  \in \cS_\fn([0,L]^d)$
\begin{align}
\label{eqn:pairing_on_box_compared_with_extensions_on_torus}
\langle u, \varphi \rangle  = 2^{-d} \langle \tilde u , \tilde \varphi \rangle, \qquad 
\langle v, \psi \rangle  = 2^{-d} \langle \overline v , \overline \psi  \rangle. 
\end{align}

\begin{theorem}
\label{theorem:closedness_and_completeness_of_S_spaces}
\begin{enumerate}
\item 
\label{item:tilde_S_spaces_closed}
We have 
\begin{align*}
\tilde \cS_0(\T_{2L}^d) & := \{ \tilde \varphi : \varphi \in \cS_0 ([0,L]^d) \} 
= \{ \psi \in \cS (\T_{2L}^d) : \psi \mbox{ is odd}\}, \\
\overline \cS_\fn(\T_{2L}^d) &  := \{ \overline \varphi : \varphi \in \cS_\fn ([0,L]^d) \} 
= \{ \psi \in \cS (\T_{2L}^d) : \psi \mbox{ is even}\}, 
\end{align*}
and $\tilde \cS_0(\T_{2L}^d) $ and $\overline \cS_\fn(\T_{2L}^d)$ are closed in $\cS(\T_{2L})$. 
\item 
\label{item:S_spaces_complete}
$\cS(\T_{2L}^d)$, $\cS_0([0,L]^d)$ and $\cS_\fn([0,L]^d)$ are complete. 
\item 
\label{item:tilde_dual_S_spaces_closed}
We have 
\begin{align*}
\tilde \cS_0'(\T_{2L}^d) & := \{ \tilde u : u \in \cS_0' ([0,L]^d) \} = \{ w \in \cS'(\T_{2L}^d) : w \mbox{ is odd}\}, \\
\overline \cS_\fn'(\T_{2L}^d) & := \{ \overline v : v \in \cS_\fn' ([0,L]^d) \} = \{ w \in \cS'(\T_{2L}^d) : w \mbox{ is even}\}, 
\end{align*}
and $\tilde \cS_0'(\T_{2L}^d)$ and $\overline \cS_\fn'(\T_{2L}^d)$ are  closed in $\cS'(\T_{2L}^d)$. 
\item 
\label{item:S_dual_spaces_complete}
$\cS'(\T_{2L}^d)$, 
$\cS_0'([0,L]^d)$ and $\cS_\fn'([0,L]^d)$ 
are (weak$^*$) sequentially complete.
\end{enumerate}
\end{theorem}
\begin{proof}
\ref{item:tilde_S_spaces_closed}
 follows as convergence in $\cS$ implies pointwise convergence and therefore the limit of odd and even functions is again odd and even, respectively.  
  \ref{item:S_spaces_complete} follows from \ref{item:tilde_S_spaces_closed} as $\cS(\T_{2L}^d)$ is complete (see  \cite[Page 134]{Do69}). 
\ref{item:tilde_dual_S_spaces_closed} 
If a net $(w_\iota)_{\iota \in \I} $ in $\tilde \cS_0'$ converges in $\cS'$ to some $w$, then $ \langle w_\iota, e_k \rangle \rightarrow \langle w ,e_k \rangle$ for all $k$, so that $w$ is odd. 
\ref{item:S_dual_spaces_complete} follows from \ref{item:tilde_dual_S_spaces_closed} as $\cS'(\T_{2L}^d)$ is weak$^*$ sequentially complete (see \cite[Page 137]{Do69}). 
\end{proof}

As we index the basis $e_k$, $\fd_k$ and $\fn_k$ by elements $k$ in $\Z^d$ and not in $\frac1L \Z^d$, in the next definition of a Fourier multiplier we have an additional $\frac1L$ factor in the argument of the functions $\tau$ and $\sigma$.

\begin{definition}
\label{def:fourier_multiplier}
Let $\tau : \R^d \rightarrow \R$, $\sigma : [0,\infty)^d \rightarrow \R$, 
 $w\in \cS'(\T_{2L}^d)$, $u\in \cS_0'([0,L]^d)$ and $v\in \cS_\fn'([0,L]^d)$. We define (at least formally) the so-called \emph{Fourier multipliers}  by
\begin{align}
\notag 
\tau(\rD) w & = \sum_{k\in \Z^d} \tau(\tfrac{k}{L}) \langle w, e_k \rangle  e_k, \\
\label{eqn:definition_sigma_D}
\sigma(\rD) u & = \sum_{k\in \N^d} \sigma(\tfrac{k}{L}) \langle u,\fd_k \rangle  \fd_k, 
\quad 
\sigma(\rD) v = \sum_{k\in \N_0^d} \sigma(\tfrac{k}{L}) \langle v, \fn_k \rangle  \fn_k. 
\end{align}
\end{definition}

Let $(\rho_j)_{j\in \N_{-1}}$ form a \emph{dyadic partition of unity}, i.e.,  $\rho_{-1}$ and $\rho_0$ are $C^\infty$ radial functions on $\R^d$, where $\rho_{-1}$ is supported in a ball and $\rho_0$ is supported in an annulus, $\rho_j = \rho(2^{-j} \cdot)$ for $j \in \N_0$, and 
\begin{align}
\label{eqn:partition_prop_sum_to_1_and_squares_inbetween_half_and_1}
& \sum_{j \in \N_{-1}} \rho_{j}(y) =1, 
\qquad 
\frac12 \le \sum_{j \in \N_{-1}} \rho_j(  y)^2 \le 1 \qquad (y \in \R^d), \\
\label{eqn:partition_prop_disjointness_far_away}
& |j - k| \ge 2 \Longrightarrow \supp \rho_j \cap \supp \rho_k = \emptyset \qquad (j,k\in \N_0). 
\end{align}

Let $w\in \cS'(\T_{2L}^d)$, $u\in \cS_0'([0,L]^d)$ and $v\in \cS_\fn'([0,L]^d)$. 
We define the  Littlewood-Paley blocks $\Delta_j w$, $\Delta_j u$ and $\Delta_j v$ for $j\in \N_{-1}$ by
$\Delta_j w = \rho_j( \rD) w$, $\Delta_j u = \rho_j( \rD) u$, $\Delta_j v = \rho_j( \rD) v$, i.e.,   
\begin{align*}
 \Delta_j  w = \sum_{k\in\Z^d} \langle w, e_k \rangle \rho_j(\tfrac{k}{ L }) e_k, \quad 
 \Delta_j  u = \sum_{k\in\N^d} \langle u, \fd_k \rangle \rho_j(\tfrac{k}{ L }) \fd_k, \quad 
 \Delta_j  v = \sum_{k\in\N_0^d} \langle v, \fn_k \rangle \rho_j(\tfrac{k}{ L }) \fn_k. 
\end{align*}
Let $\overline \sigma : \R^d \rightarrow \R$ be the even extension of $\sigma$, i.e., $\overline \sigma( \fq \circ x) = \sigma(x)$ for all $x\in [0,\infty)^d$ and $\fq \in \{-1,1\}^d$. 
As $ \sigma(\rD) \fd_k = \sigma(\frac{k}{L}) \fd_k$ and $\overline \sigma(\rD) \tilde \fd_k = \sigma(\frac{k}{L}) \tilde \fd_k$, by Theorem \ref{theorem:representation_schwarz_func_and_distr} we obtain that for all $u\in \cS_0'([0,L]^d)$ and $v\in \cS_\fn'([0,L]^d)$, 
\begin{align}
\label{eqn:DDelta_and_Delta_of_tilde}
\widetilde{\sigma(\rD) u} = \overline \sigma(\rD) \tilde u, \qquad \overline{\sigma(\rD) v} = \overline \sigma(\rD) \overline v. 
\end{align}
Moreover, with $a_{d,p} = 2^{-\frac{d}{p}}$ for $p <\infty$ and $a_{d,\infty} = 1$  we have for all $p \in [1,\infty]$
\begin{align*}
\| \sigma(\rD) u \|_{L^p([0,L]^d)} 
&= a_{d,p} \| \widetilde{ \sigma(\rD) u} \|_{L^p(\T_{2L}^d)}
= a_{d,p} \| \overline \sigma(\rD) \tilde u \|_{L^p(\T_{2L}^d)},  \\
\| \sigma(\rD) v \|_{L^p([0,L]^d)} 
&= a_{d,p} \| \overline{ \sigma(\rD) v} \|_{L^p(\T_{2L}^d)}
= a_{d,p} \| \overline \sigma(\rD) \overline v \|_{L^p(\T_{2L}^d)}. 
\end{align*}
Therefore, by applying the above to $\sigma = \rho_j$, with $\|\cdot\|_{B_{p,q}^\alpha}$ the standard Besov norm, 
\begin{align*}
a_{d,p} \|\tilde u \|_{B^{\alpha}_{p,q}} =   \| ( 2^{i\alpha} \| \Delta_i u \|_{L^p})_{i \in \N_{-1}}\|_{\ell^q}, \qquad 
a_{d,p} \|\overline v\|_{B^{\alpha}_{p,q}} =   \| ( 2^{i\alpha} \| \Delta_i v \|_{L^p})_{i \in \N_{-1}}\|_{\ell^q}. 
\end{align*}
This motivates the following definition. 

\begin{definition}
\label{def:Dir_and_Neu_Besov}
Let $\alpha \in \R$, $p,q\in [1,\infty]$. 
We define the 
\emph{Dirichlet Besov space} 
$B^{\Dir,\alpha}_{p,q}([0,L]^d)$ to be the space of 
$u \in \cS_0'([0,L]^d)$ for which $\|u\|_{B^{\Dir,\alpha}_{p,q}}:= a_{d,p} \| \tilde u\|_{B^{\alpha}_{p,q}}<\infty$. 
Similarly, we define the 
\emph{Neumann Besov space}
$B^{\Even,\alpha}_{p,q}([0,L]^d)$ as the space of 
$v \in \cS_\fn'([0,L]^d)$ for which $\|v\|_{B^{\Even,\alpha}_{p,q}}:= a_{d,p} \| \overline v\|_{B^{\alpha}_{p,q}}<\infty$. 

We will abbreviate $\cC^\alpha_\fn = B^{\Even,\alpha}_{\infty,\infty}$, $H^\alpha_\fn = B^{\Even,\alpha}_{2,2}$. 
In Theorem \ref{theorem:equivalences_H_0_spaces} we show $H^\alpha_0 = B^{\Dir,\alpha}_{2,2}$.
\end{definition}

\begin{calc}
\begin{obs}
Let us see if the definition is such that we get ``what we want''. 
First of all let us note that the Littlewood-Paley block $\Delta_i u$ needs to be defined as $\rho_i(\rD) u$ which has the factor $\frac{1}{L}$ in front of the Fourier coefficient $k$, so that we have equivalence with the $H^\alpha$ spaces for $p=q=2$. 
This so that the regularity of distributions does not change when considering a bigger space. 
Let us demonstrate this for $\delta_0$ in $\cS_\fn'([0,L]^d)$. 
We have 
\begin{align*}
 \Delta_i \delta_0 = \sum_{k\in \N_0^d} \rho_i(\tfrac{k}{L}) \fn_0. 
\end{align*}
As $\|\fn_0\|_{L^\infty}$ is bounded by $L^{-\frac{d}{2}}$ and equal to it at $0$ we have 
\begin{align*}
\| \sum_{k\in \N_0^d} \rho_i(\tfrac{k}{L}) \fn_0(0) \fn_0\|_{L^\infty} 
= L^{-d} \sum_{k\in \N_0^d} \rho_i(\tfrac{k}{L}). 
\end{align*}
The latter sum is about equal to $2^{id} L^d$ so that $\delta_0 \in B_{\infty,\infty}^{-d}([0,L]^d)$ and has the about the same norm for all $L$ (at least for $L=\lambda 2^m$ with $\lambda>0$ and $m\in \N_0$ the norms are the same). 

Now let us see what happens for $\alpha =0$ and constant functions. 
We have 
\begin{align*}
\langle \1 , \fn_{0,L} \rangle = \int_{[0,L]^d} L^{-\frac{d}{2}} \dd x = L^{\frac{d}{2}}, 
\end{align*}
and thus 
\begin{align*}
\Delta_{-1} \1 = \rho_{-1}(0) \langle \1, \fn_{0,L} \rangle \fn_{0,L}   =  1 \qquad \Delta_i \1 =0 \mbox{ for } i \ge 0. 
\end{align*}
Hence $\| \Delta_{-1} \1\|_{L^\infty} =1 $ and thus $\| \1\|_{B_{\infty,\infty}^{\fn,0}([0,L]^d)} =1 = \| \1 \|_{L^\infty}$. 
On the other hand $\| \Delta_{-1} \1\|_{L^p} = L^{\frac{1}{p}}$ and thus $\| \1\|_{B_{2,2}^{\fn,0}([0,L]^d)} = L^{\frac{1}{p}} = \| \1 \|_{L^2}$.  
\end{obs}
\end{calc}

As $B_{p,q}^\alpha(\T_{2L}^d)$ is a Banach space, 
 $\|\cdot\|_{B^{\Dir,\alpha}_{p,q}}$ is a norm on 
$B^{\Dir,\alpha}_{p,q}([0,L]^d)$ under which it is a Banach space. 
Similarly, $\|\cdot \|_{B^{\Even,\alpha}_{p,q}}$  is a norm on 
$B^{\Even,\alpha}_{p,q}([0,L]^d)$ under which it is a Banach space.

\begin{theorem}
\label{theorem:denseness_smooth_compact_in_dirchlet_besov}
$C_c^\infty([0,L]^d )$ 
is dense in $B_{p,q}^{\Dir, \alpha}([0,L]^d)$ for all $\alpha \in \R$, $p, q \in [1,\infty)$. 
\end{theorem}
\begin{proof}
The proof follows the same strategy as the proof of \cite[Proposition 2.74]{BaChDa11}. 
\end{proof}

\begin{theorem}
\label{theorem:equivalent_spaces_to_H_alpha_R_d_and_torus}
For $\alpha>0$, 
$H^\alpha(\R^d) = B_{2,2}^\alpha(\R^d) = \Lambda_{2,2}^\alpha(\R^d)$  and their norms are equivalent (for the definitions see \cite[p. 36]{Tr83}). 

\end{theorem}
\begin{proof}
For $H^\alpha(\R^d) = F_{2,2}^\alpha(\R^d)$ see \cite[p.88]{Tr83}, for $F_{2,2}^\alpha(\R^d) = B_{2,2}^\alpha(\R^d)$ see \cite[p.47] {Tr83} and for $B_{2,2}^\alpha(\R^d) = \Lambda_{2,2}^\alpha(\R^d)$
  see \cite[p.90]{Tr83}. 
\end{proof}

\begin{lemma}
For $\alpha \in \R$ the spaces $B_{2,2}^\alpha(\T_{2L}^d)$ and $H^\alpha(\T_{2L}^d)$ (see \cite[p. 168]{ScTr87}) are equal with equivalent norms. 
Here $H^\alpha(\T_{2L}^d)$ is the space of distributions in $\cS'(\T_{2L}^d)$ for which $\|u\|_{H^\alpha}<\infty$, where
\begin{align*}
\|u\|_{H^{\alpha}} = \sqrt{ \sum_{k\in \N_0^2} (1+|\tfrac{k}{L}|^2)^\alpha \langle u, e_k\rangle^2 }. 
\end{align*}
\end{lemma}
\begin{proof}
Observe that by the properties of the dyadic partition:
for all $\alpha \in \R $ there exist $c_\alpha, C_\alpha>0$ such that 
\begin{align}
\label{eqn:sum_of_power_2_rho_j_and_one_plus_mod_sq}
c_\alpha \left(1+ |\tfrac{k}{ L }|^2\right)^\alpha 
\le \sum_{j \in \N_{-1}} 2^{2\alpha j } \rho_j(\tfrac{k}{ L })^2
\le C_\alpha \left(1+ |\tfrac{k}{ L }|^2\right)^\alpha. 
\end{align}
Therefore the equivalence of the norms follows by Plancherel's formula. 
\end{proof}

The following is a consequence of the fact that the norms of $H^\alpha (\T_{2L}^d)$ (see \cite[p. 168]{ScTr87}) and $B_{2,2}^\alpha(\T_{2L}^d)$ are equivalent.

\begin{theorem}
\label{theorem:equivalent_norms_H_spaces}
For all $\alpha \in \R$ we have 
for $u \in \cS_\fn'([0,L]^d)$ and $v\in \cS_0'([0,L]^d)$
\begin{align*}
\|u\|_{B_{2,2}^{\fn,\alpha}} 
\eqsim \sqrt{ \sum_{k\in \N_0^2} (1+|\tfrac{k}{L}|^2)^\alpha \langle u, \fn_k\rangle^2 }, \quad 
\|v\|_{B_{2,2}^{\fd,\alpha}} 
\eqsim \sqrt{ \sum_{k\in \N_0^2} (1+|\tfrac{k}{L}|^2)^\alpha \langle v, \fd_k\rangle^2 }.
\end{align*}
\end{theorem}

\begin{theorem}
\label{theorem:equivalences_H_0_spaces}
For $\alpha >0$ 
the spaces $B_{2,2}^{\Dir, \alpha}([0,L]^d)$ and $H^\alpha_0([0,L]^d)$ are equal with equivalent norms, where $H^\alpha_0([0,L]^d)$ is the closure of $C_c^\infty([0,L]^d)$ in $H^\alpha(\R^d)$. 
\end{theorem}
\begin{proof}
As $C_c^\infty([0,L]^d)$ is dense in $B_{2,2}^{\Dir, \alpha}([0,L]^d)$ (Theorem \ref{theorem:denseness_smooth_compact_in_dirchlet_besov}) it is sufficient to prove the equivalence of  the norms on $C_c^\infty([0,L]^d)$. 
Let $f \in C_c^\infty([0,L]^d)$. 
By definition of the $\Lambda_{2,2}^\alpha$ norm,  $ \|  f\|_{\Lambda_{2,2}^\alpha(\T_L^d)} =  \| f\|_{\Lambda_{2,2}^\alpha(\R^d)}$. 
As $D^\beta \tilde f = \widetilde{ D^\beta f} $ we have 
$\| \tilde f\|_{\Lambda_{2,2}^\alpha(\T_{2L}^d)} = 2^{\frac{d}{2}} \| f\|_{\Lambda_{2,2}^\alpha(\T_L^d)}$. 
Because $\| \tilde f \|_{B_{2,2}^{\alpha}(\T_{2L}^d)}= 2^{\frac{d}{2}} \| f \|_{B_{2,2}^{\Dir,\alpha}([0,L]^d)} $ (by definition), the proof follows by Theorem \ref{theorem:equivalent_spaces_to_H_alpha_R_d_and_torus}. 
\begin{calc}
Explaining the ``sufficient'' now. 
Let $f \in B_{2,2}^{\Dir,\alpha}([0,L]^d)$ and $g_n \in C_c^\infty([0,L]^d)$ be such that $g_n \rightarrow f$ in $B_{2,2}^{\Dir,\alpha}$. Then by the proved equivalences $(g_n)_n$ is Cauchy in $H_0^\alpha$, hence is  converging to $f$. This extends the equivalence of norms on $C_c^\infty([0,L]^d)$ to the whole space. 
\end{calc}
\end{proof}

\begin{theorem}
\label{theorem:compact_embeddings_besov_spaces}
Let $p,q \in [1,\infty]$ and $\beta,\gamma \in \R$, $\gamma <\beta$. 
Then $B_{p,q}^\beta(\T_{2L}^d)$ is compactly embedded in $B_{p,q}^\gamma(\T_{2L}^d)$, i.e., every bounded set in $B_{p,q}^\beta(\T_{2L}^d)$ is compact in $B_{p,q}^\gamma(\T_{2L}^d)$. 
The analogous statement holds for $B_{p,q}^{\Dir,\beta}([0,L]^d)$ and $B_{p,q}^{\Even,\beta}([0,L]^d)$ spaces. 
In particular, the injection $j : H_0^\beta ([0,L]^d) \rightarrow H_0^\gamma ([0,L]^d)$ is a compact operator.
\end{theorem}
\begin{proof}
We consider the underlying space to be $\T_{2L}^d$, i.e., periodic boundary conditions; the other cases follow by Theorem \ref{theorem:closedness_and_completeness_of_S_spaces}.
Suppose that $u_n \in B_{p,q}^\beta$ and $\|u_n\|_{B_{p,q}^\beta} \le 1$ for all $n\in\N$. We prove that there is a subsequence of $(u_n)_{n\in\N}$ that converges in $B_{p,q}^\gamma$. 
By \cite[Theorem 2.72]{BaChDa11} there exists a subsequence of $(u_n)_{n\in\N}$, which we assume to be the sequence itself, such that $u_n \rightarrow u$ in $\cS'$ and $\|u\|_{B_{p,q}^\beta} \le 1 $. As $\langle u_n, e_k \rangle \rightarrow \langle u, e_k \rangle$ for all $k\in \Z^d$, we have $\| \Delta_j (u_n -u )\|_{L^p} \rightarrow 0$ for all $j\in \N_{-1}$. 
Let $\epsilon>0$. 
Choose $J \in \N$ large enough such that $2^{(\gamma - \beta) J} < \epsilon$, so that for all $n\in\N$
\begin{align*}
\| ( 2^{\gamma j} \|\Delta_j (u_n -u) \|_{L^p})_{j=J+1}^\infty \|_{\ell^q}  
&  \le  2^{(\gamma - \beta) J} \| ( 2^{\beta j} \|\Delta_j (u_n -u) \|_{L^p})_{j=J+1}^\infty \|_{\ell^q} \\
& \le  2^{(\gamma - \beta) J} (\|u_n\|_{B_{p,q}^\beta} + \|u\|_{B_{p,q}^\beta}) 
< 2\epsilon. 
\end{align*}
Then, by choosing $N\in \N$ large enough such that $\| ( 2^{\gamma j} \|\Delta_j (u_n -u) \|_{L^p})_{j=-1}^J \|_{\ell^q} < \epsilon$  for all $n\ge N$, one has with the above bound that $\| u_n - u\|_{B_{p,q}^\gamma} < 3 \epsilon$ for all $n\ge N$. 
\end{proof}

%

\begin{obs}
\label{obs:H_0_equals_L_2}
Observe that by Lemma \ref{lemma:basis_d_k_and_f_k} $H_0^0([0,L]^d) = H_\fn^0([0,L]^d) = L^2([0,L]^d)$ and $\|\cdot\|_{H_0^0} \eqsim \|\cdot\|_{H_\fn^0} \eqsim \|\cdot\|_{L^2}$. 
\end{obs}

\begin{obs}
\label{obs:inverse_laplacian_plus_one}
By \ref{obs:derivatives_and_fourier_transforms} we have $(a-\Delta)^{-1} f = \sigma(\rD) f$ for $\sigma(x) = (a+\pi^2 |x|^2)^{-1}$.
\end{obs}

\begin{obs}
\label{obs:scaling_fourier_multipliers}
For any function $\varphi$ and $\lambda \in \R$ we write $l_\lambda \varphi$ for the function $x\mapsto \varphi(\lambda x)$. 
For a distribution $u$ we write $l_\lambda u$ for the distribution given by $\langle l_\lambda u , \varphi \rangle = \lambda^{-d} \langle  u , l_{\frac{1}{\lambda}} \varphi \rangle  $. 
As $l_\lambda e_{k,2L} = \lambda^{-\frac{d}{2}} e_{k, \frac{2L}{\lambda}}$, 
and 
$\langle l_\lambda u , e_{k,\frac{2L}{\lambda}} \rangle = \lambda^{-\frac{d}{2}} \langle u, e_{k,2L} \rangle$, 
we have for $u \in \cS'(\T_{2L}^d)$
\begin{align}
\label{eqn:scaling_and_fourier_multipliers}
l_\lambda [ \sigma(\lambda \rD) u ]  = \sigma(\rD) [l_\lambda u]. 
\end{align}
Similarly, \eqref{eqn:scaling_and_fourier_multipliers} holds for 
$u\in \cS'_0([0,L]^d)$ and $u\in  \cS'_\fn([0,L]^d)$ (use e.g. \ref{obs:odd_and_even_functions_and_their_fourier_transform_and_inner_prod}).
\begin{calc}
If $\varphi  \in \cS([0,L]^d)$, then $l_\lambda \varphi \in \cS( [0, \frac{L}{\lambda}]^d)$. 
If $u \in \cS'([0,L]^d)$, then $l_\lambda u \in \cS'([0,\frac{L}{\lambda}]^d)$. 
\begin{align*}
l_\lambda [ \sigma(\rD) w ] = 
l_\lambda [ 
 \sum_{k\in\Z^d} \langle w, e_k \rangle \sigma(\tfrac{k}{L}) e_k
]
= \lambda^{-\frac{d}{2}}  \sum_{k\in\Z^d} \langle w, e_k \rangle \sigma(\tfrac{k}{L}) e_{k, \frac{2L}{\lambda}}
\\
\sigma(\rD) [l_\lambda w]
=  \sum_{k\in\Z^d} \langle l_\lambda w, e_{k,\frac{2L}{\lambda}} \rangle \sigma(\tfrac{\lambda k}{L}) e_{k,\frac{2L}{\lambda}}
\end{align*}
\end{calc}
\end{obs}

\begin{definition}
\label{def:spaces_on_general_boxes}

Let $y \in \R^d$, $ s \in (0,\infty)^d$ and 
\bfold 
\Gamma = y + \prod_{i=1}^d [0,s_i]. 
\efold 
Let $l : \prod_{i=1}^d [0 , s_i ] \rightarrow [0,1]^d$ be given by $l(x) = ( \frac{x_1 }{s_1}, \dots, \frac{x_d }{s_d} ) $. 
For a function $\varphi$ we define new functions $l\varphi$ and $\cT_y \varphi$ by $l \varphi(x)  =\varphi \circ l(x)$ and $\cT_y \varphi(x)= \varphi(x-y)$ and for a distribution $u$ we define the distributions $lu$ and $\cT_y u$ by $\langle lu, \varphi \rangle = |\det l |^{-1} \langle u, l^{-1} \varphi \rangle$ and $\langle \cT_y u, \varphi \rangle = \langle u, \cT_y^{-1} \varphi \rangle$. 
We define
\begin{align}
\notag \cS_0(\Gamma) & := \cT_y l [\cS_0([0,1]^d)], 
\qquad  \cS_0'( \Gamma )  := \cT_y l (\cS_0'([0,1]^d)), \\
\label{eqn:definition_sigma_D_on_rectangle}
\sigma(\rD) u & := \cT_y l [ (l\sigma)(\rD) ( (\cT_y l)^{-1} u) ] \qquad \mbox{ for } u \in \cS_0'(\Gamma). 
\end{align} 
Note that the definition of $\sigma(\rD)u$ is consistent with \eqref{eqn:definition_sigma_D} by \ref{obs:scaling_fourier_multipliers}.
\begin{calc}
Indeed, for $u\in \cS_0'([0,L]^d)$ we have $l_L u \in \cS_0'([0,1]^d)$, $ | \det l_L | = L^d$  and 
\begin{align*}
 \| \sigma (\rD) l_L u \|_{L^p([0,1]^d)} 
= \| l_L[ \sigma (L \rD) u] \|_{L^p([0,1]^d)} 
= L^{-\frac{d}{p}}  \| \sigma (L \rD) u \|_{L^p([0,L]^d)} . 
\end{align*}
\begin{hide}
This is not equal to $L^{-\frac{d}{p}}  \| \Delta_i f \|_{L^p([0,L]^d)} $ because of the factor $L$ in $\rho_i$. But this does not matter for the equivalence with the $H^\alpha$ spaces as we see later. 
\end{hide}
\end{calc}
Moreover, we define $\Delta_i = \rho_i(\rD)$ (as in \eqref{eqn:definition_sigma_D_on_rectangle}) and 
\begin{align*}
\|u\|_{B_{p,q}^{\Dir,\alpha}}(\Gamma) 
:=  \| ( 2^{i\alpha} \| \Delta_i u \|_{L^p})_{i \in \N_{-1}}\|_{\ell^q}.  
\end{align*}
Similarly, we define $\cS_\fn(\Gamma),\cS_\fn'(\Gamma),B_{p,q}^{\fn,\alpha}(\Gamma)$ and $\|\cdot\|_{B_{p,q}^{\fn,\alpha}(\Gamma)}$.

\begin{hide}
OLD (when $\Delta_i = \rho_i(L\rD)$):
\begin{align*}
\|u\|_{B_{p,q}^{\Dir,\alpha}}(\Gamma) 
:= |\det l|^{-\frac{1}{p}} \left\| (\cT_y l)^{-1} u \right\|_{B_{p,q}^{\Dir,\alpha}([0,1]^d)} ,  
\end{align*}
where for $p = \infty$ we make the convention that $|\det l|^{-\frac{1}{p}}=1$. 
Observe that this agrees with our definition of the $B_{p,q}^{\Dir,\alpha}$ on $[0,L]^d$ for all $L>0$ (see also \ref{obs:scaling_fourier_multipliers}). 
Indeed, for $f\in B_{p,q}^{\Dir,\alpha}([0,L]^d)$ we have $l_L f \in B_{p,q}^{\Dir,\alpha}([0,1]^d)$, $ | \det l_L | = L^d$  and 
\begin{align*}
\| \Delta_i l_L f \|_{L^p([0,1]^d)} 
& = \| \rho_i (\rD) l_L f \|_{L^p([0,1]^d)} 
= \| l_L[\rho_i (L \rD) f] \|_{L^p([0,1]^d)} \\
& = L^{-\frac{d}{p}}  \| \rho_i (L \rD) f \|_{L^p([0,L]^d)} 
= L^{-\frac{d}{p}}  \| \Delta_i f \|_{L^p([0,L]^d)}. 
\end{align*}
Note that for $s = (L,L,\dots,L)$ the corresponding $l$ equals $l_L^{-1}$. So $|\det l_L^{-1}|^{-\frac{1}{p}} = L^{\frac{d}{p}}$ and indeed
\begin{align*}
\|f\|_{B_{p,q}^{\Dir,\alpha}}([0,L]^d)
= L^{\frac{d}{p}} \| l_L f\|_{B_{p,q}^{\Dir,\alpha}}([0,1]^d). 
\end{align*}
\end{hide}

\end{definition}

The following theorem gives a bound on Fourier multipliers, similar as in \cite[Theorem 2.78]{BaChDa11}. However, considering the particular choice $H^\gamma(\T_{2L}^d)= B_{2,2}^\gamma(\T_{2L}^d)$ allows us to reduce condition to control all derivatives of $\sigma$ to a condition that only controls the growth of $\sigma$ itself. 

\begin{theorem} 
\label{theorem:schauder} 
Let $\gamma,m \in \R$ and $M>0$. There exists a $C>0$ such that the following statements hold. 
\begin{enumerate}
\item 
For all bounded $\sigma : \R^d  \rightarrow \R$ such that 
$| \sigma(x) | \le M ( 1+ |x|)^{-m}$ for all $x\in \R^d$, 
\begin{align}
\label{eqn:fourier_multiplier_bound_on_H}
\|\sigma(\rD)w  \|_{H^{\gamma +m }} & \le C  \|w \|_{H^\gamma} \qquad (w\in \cS'(\T_{2L}^d)).
\end{align}
By \eqref{eqn:DDelta_and_Delta_of_tilde}, one may replace ``$H$'' and ``$\cS'(\T_{2L}^d)$'' by ``$H_0$'' and ``$\cS'_0([0,L]^d)$'' or ``$H_\fn$'' and ``$\cS'_\fn([0,L]^d)$'' in \eqref{eqn:fourier_multiplier_bound_on_H}. 
\item 
For all $\sigma : \R^d  \rightarrow \R$ which are $C^\infty$ on $\R^d \setminus \{0\}$, such that  $|\partial^\alpha \sigma(x)| \le M |x|^{-m-|\alpha|}$ for all $x \in \R^d\setminus \{0\}$ and $\alpha \in \N_0^d$ with $|\alpha| \le 2 \lfloor 1+ \frac{d}{2} \rfloor$, 
\begin{align}
\label{eqn:fourier_multiplier_bound_on_C}
\|\sigma(\rD)w  \|_{\cC^{\gamma +m }} & \le C  \|w \|_{\cC^\gamma}  \qquad (w\in \cS'(\T_{2L}^d)).
\end{align}
By \eqref{eqn:DDelta_and_Delta_of_tilde}, one may replace ``$\cC$'' and ``$\cS'(\T_{2L}^d)$'' by ``$\cC_\fn$'' and ``$\cC'_\fn([0,L]^d)$'' in \eqref{eqn:fourier_multiplier_bound_on_C}. 
\end{enumerate}
\end{theorem}
\begin{proof}
Let $a>0$ be such that $\rho(k) = 0$ if $|k|<a$. 
Then for $j \ge 0$ one has  $|\rho_j(k) \sigma(k)|\le M (1+\frac{a 2^j}{L})^{-m} \rho_j(k)  \le M L^m a^{-m} 2^{-j m} \rho_j(k) $ for all $k\in \Z^d$. As $\sigma$ is bounded on the support of $\rho_{-1}$, there exists a $C>0$ such that for all $j \in \N_{-1}$
\begin{align*}
\| \sigma(\rD) \Delta_j w \|_{L^2} 
&= \sqrt{ \sum_{k\in \Z^d} |w(e_k)|^2 |\sigma(\tfrac{k}{L})|^2 |\rho_j(k)|^2 }
\le C 2^{-jm} \|  \Delta_j w \|_{L^2} .
\end{align*}
\eqref{eqn:fourier_multiplier_bound_on_C} follows from 
\cite[Lemma 2.2]{BaChDa11}.
\end{proof}

\begin{obs}
\label{obs:sigma_one_plus_x_sq_inverse_satisfies_conditions}
Using the multivariate chain rule (Fa\`a di Bruno's formula) one can prove that $\sigma(x) = (1+\pi^2|x|^2)^{-1}$ satisfies the conditions in Theorem \ref{theorem:schauder} (those needed for \eqref{eqn:fourier_multiplier_bound_on_C}). 
\end{obs}

One other bound that we will refer to is a special case of \cite[Proposition 2.71]{BaChDa11}:

\begin{theorem}
\label{theorem:inclusion_of_H_into_C}
For all $\alpha \in \R$ there exists a $C>0$ such that 
\bfold 
 \|w\|_{\cC_\fn^{\alpha}} \le C \|w\|_{H_\fn^{\alpha+\frac{d}{2}}} 
\efold 
for all $w\in \cS_\fn'([0,L]^d)$.
\end{theorem}


Now 
we consider (para- and resonance-) products between elements of $\cS_0'([0,L]^d)$ and $\cS_\fn'([0,L]^d)$, and between elements of $\cS_\fn'([0,L]^d)$. 

\begin{obs}
\label{obs:products_of_cS_distributions}
Let $w_1,w_2 \in \cS'(\T_{2L}^d)$ be represented by $w_1 = \sum_{k \in \Z^d} a_k e_k$ and $w_2 = \sum_{l \in \Z^d} b_l e_l$. 
Then formally $w_1 w_2 = \sum_{m\in \Z^d} c_m e_m$, with $c_m = \sum_{k,l\in \Z^d, k+l=m} a_k b_l$. 

\begin{calc}
Let us first observe the following. 
\begin{itemize}
\item If $(a_k)_{k\in\Z^d}$ and $(b_k)_{k\in \Z^d}$ satisfy \eqref{eqn:growth_coefficients}, then so does $(c_m)_{m\in\Z^d}$.
Indeed, by assuming that $a_0=b_0=0$ and using that $|kl| \ge \frac{|k+l|}{2}$ for $k,l \ne 0$ we have for all large enough $n\in\N$
\begin{align*}
|c_m| \lesssim \sum_{\substack{k,l \in \Z^d \setminus \{0\} \\ k+l = m} } |kl |^{-n} 
\lesssim \sum_{\substack{k,l \in \Z^d \setminus \{0\}\\ k+l = m} } |kl |^{-\frac{n}{2}} |k+l|^{-\frac{n}{2}} 
\lesssim |m|^{-\frac{n}{2}}. 
\end{align*}
\item If $(a_k)_{k\in\Z^d}$ and $(b_k)_{k\in \Z^d}$ satisfy \eqref{eqn:growth_coefficients_distri}, then $(c_m)_{m\in\Z^d}$ as above might not (take $a_k = b_k = |k|^n$ for some $n\in\N$). 
\end{itemize}
\end{calc}
Of course this series is not always convergent (e.g. take $a_k = b_k = |k|^n$ for some $n\in\N$ and see \eqref{eqn:growth_coefficients_distri}).
But if it does, then due to the identities 
\begin{align}
\label{eqn:basis_prod_odd_even}
(2L)^{\frac{d}{2}} \tilde \fd_k \overline \fn_l 
& 
\begin{calc2} 
= (-\i)^d \nu_l \sum_{\fq, \fp \in \{-1,1\}^d} (\prod \fq) e_{\fq \circ (k + \fp \circ l)} 
\end{calc2}
= \nu_l \sum_{\fp \in \{-1,1\}^d} \tilde \fd_{k+\fp \circ l}, \\
\label{eqn:basis_prod_odd_odd}
(2L)^{\frac{d}{2}} \tilde \fd_k \tilde \fd_l 
&
\begin{calc2}
 = (-1)^d \sum_{\fq,\fp \in \{-1,1\}^d} (\prod \fp) e_{\fq \circ (k + \fp \circ l)} 
 \end{calc2}
= (-1)^d \sum_{\fp \in \{-1,1\}^d} \nu_{k+\fp \circ l}^{-1} (\prod \fp) \overline \fn_{ k + \fp \circ l},  \\
\label{eqn:basis_prod_even_even}
(2L)^{\frac{d}{2}} \overline \fn_k \overline \fn_l 
& 
\begin{calc2}
= \nu_k \nu_l \sum_{\fq,\fp \in \{-1,1\}^d}  e_{\fq \circ (k + \fp \circ l)} 
\end{calc2}
= \sum_{\fp \in \{-1,1\}^d} \frac{\nu_k \nu_l}{\nu_{k+\fp \circ l}} \overline \fn_{ k + \fp \circ l}, 
\end{align}
the product obeys the following rules 
\begin{align*}
\mbox{even} \times \mbox{even} = \mbox{even}, 
\quad \mbox{odd} \times \mbox{even} = \mbox{odd},
\quad \mbox{odd} \times \mbox{odd} = \mbox{even}.
\end{align*}
For example, if $u \in \cS_0'$ and $v\in \cS_\fn'$ and $uv$ exists in a proper sense, then $uv \in \cS_0'$.
\end{obs}

\begin{definition}
\label{definition:paraproducts_odd_with_even}
For $u\in \cS_0'([0,L]^d) \cup  \cS_\fn'([0,L]^d)$ and $v\in  \cS_\fn'([0,L]^d)$ we write (at least formally)
\begin{align}
\label{eqn:para_oe_and_reso}
u \para v  = v \arap u = \sum_{{\substack{ i,j\in \N_{-1} \\ i \le j-1}}} \Delta_i  u \Delta_j v , \qquad 
u \reso v  = \sum_{{\substack{ i,j\in \N_{-1} \\ |i-j|\le 1 }}} \Delta_i  u \Delta_j v .
\end{align}
\end{definition}

\begin{obs}
\label{obs:para_and_resonance_products_extended_to_torus}
As $\widetilde{ \fd_k \fn_m} = \tilde \fd_k \overline \fn_m$ and $\overline{\fn_k \fn_m} = \overline{\fn}_k \overline{\fn}_m$, we have (at least formally) 
\begin{align}
\label{eqn:tilde_of_products}
\widetilde{u \para v} = \tilde u \para \overline{v}, \qquad
\widetilde{u \arap v} = \tilde u \arap \overline{v}, \qquad 
\widetilde{u \reso v} = \tilde u \reso \overline{v}, \\
\overline{u \para v} = \overline{u} \para \overline{v}, \qquad
\overline{u \arap v} = \overline{u} \arap \overline{v}, \qquad 
\overline{u \reso v} = \overline{u} \reso \overline{v}. 
\end{align}
With this one can extend the Bony estimates on the (para-/resonance) products on the torus to Bony estimates between elements of $B^{\Dir,\alpha}_{p,q}([0,L]^d)$ and $B^{\Even,\beta}_{p,q}([0,L]^d)$ and between elements of $B^{\Even,\beta}_{p,q}([0,L]^d)$. We list some Bony estimates in Theorem \ref{theorem:bony_estimates}.
\end{obs}

\begin{theorem}(Bony estimates)
\label{theorem:bony_estimates}
\begin{enumerate}
\item 
\label{item:bony_para}
For all $\alpha<0$, $\gamma \in \R$, there exists a $C>0$ such that for all $L>0$ 
\begin{align*}
\| f \arap \xi \|_{H_0^{\alpha+\gamma}} \le C \|f\|_{H_0^\gamma} \|\xi \|_{\cC_\fn^\alpha} \qquad (f \in \cS_0'([0,L]^d),\xi \in \cS_\fn'([0,L]^d)).
\end{align*}
\item 
\label{item:bony_for_L_functions_with_minimal_loss_of_reg}
For all $\delta>0$, $\gamma \ge -\delta$ and $\beta \in \R$ there exists a $C>0$ such that for all $L>0$ 
\begin{align*}
\| f\para \xi \|_{H_0^{\beta-\delta}} \le C \|f\|_{H_0^\gamma} \|\xi \|_{\cC_\fn^\beta} \qquad (f \in \cS_0'([0,L]^d),\xi \in \cS_\fn'([0,L]^d)).
\end{align*}
\item 
\label{item:bony_reso}
For all $\alpha,\gamma \in \R$ with $\alpha + \gamma>0$, there exists a $C>0$ such that for all $L>0$
\begin{align*}
 \| f \reso \xi \|_{H_0^{\alpha +\gamma  }} 
& \le C \|f\|_{H_0^\gamma} \|\xi\|_{\cC_\fn^\alpha} \qquad (f \in \cS_0'([0,L]^d),\xi \in \cS_\fn'([0,L]^d)), \\ 
\| f \reso \xi \|_{\cC_\fn^{\alpha +\gamma  }} 
& \le C \| f \|_{\cC_\fn^\gamma} \|\xi\|_{\cC_\fn^\alpha}
\qquad (f ,\xi \in \cS_\fn'([0,L]^d)). 
\end{align*}
\item 
\label{item:bony_product_of_distributions}
For all $\alpha,\gamma \in \R$ with $\alpha + \gamma>0$ and $\delta>0$ there exists a $C>0$ such that for all $L>0$ 
\begin{align*}
\|f \xi \|_{H_0^{\alpha \wedge \gamma - \delta}} \le C \|f\|_{H_0^\gamma} \|\xi \|_{\cC_\fn^\alpha}  \qquad (f \in \cS_0'([0,L]^d),\xi \in \cS_\fn'([0,L]^d)).
\end{align*}
\end{enumerate}
The above statements also hold by simultaneously replacing ``$H_0$'' and ``$\cS'_0$'' by ``$H_\fn$'' and ``$\cS'_\fn$''. 
\end{theorem}
\begin{proof}
By \ref{obs:para_and_resonance_products_extended_to_torus} it is sufficient to consider the analogue statements with periodic boundary conditions, that is, considering the underlying space $\T_{2L}^d$. 
For \ref{item:bony_para} and \ref{item:bony_for_L_functions_with_minimal_loss_of_reg} see \cite[Lemma 2.1]{PrTr16} and \cite[Proposition 2.82]{BaChDa11} where the underlying space is $\R^d$ rather than the torus. 
For \ref{item:bony_reso} see \cite[Proposition 2.85]{BaChDa11}. 
\ref{item:bony_product_of_distributions} follows from the rest. 
\end{proof}

The last observation we make is that one can also define Besov spaces with mixed boundary conditions, to which we refer in Definition~\ref{def:dirichlet_domain_etc}. 

\begin{obs}[Besov spaces with mixed boundary conditions]
\label{obs:besov_spaces_with_mixed_bc}
Beside the Dirichlet and Neumann Besov spaces one can define Besov spaces with mixed boundary conditions as follows. 
First observe that for $k\in \N_0^d$, the function $\fd_{k,L}$ is the product of the one dimensional functions $\fd_{k_i,L}$, in the sense that $\fd_{k,L}(x) = \prod_{i=1}^d \fd_{k_i,L}(x_i)$. Similarly, $\fn_{k,L}(x) = \prod_{i=1}^d \fn_{k_i,L}(x_i)$. 
One could interpret this as taking Dirichlet (or Neumann) boundary conditions in every direction. 
Instead one could for example for $d=2$ take the function $f_{k,L}(x) = \fd_{k_1,L}(x_1) \fn_{k_2,L}(x_2)$ and analogously to Definition~\ref{def:Dir_and_Neu_Besov} define a Besov space with mixed boundary conditions. 
Moreover, analogous to Definition~\ref{definition:paraproducts_odd_with_even} one can define the para- and resonance products as in \eqref{eqn:para_oe_and_reso} and obtain the Bony estimates as in Theorem~\ref{theorem:bony_estimates} for elements with ``opposite boundary conditions''.
\end{obs}

\section{The operator \texorpdfstring{$\Delta + \xi$}{} with Dirichlet boundary conditions}
\label{section:operator_def}

We define the Anderson Hamiltonian with Dirichlet boundary conditions and study its spectral properties that will be used in the rest of the paper. 
\textbf{In this section we assume $d=2$, $y \in \R^2$ and $s \in (0,\infty)^2$ and write 
$\Gamma = y + \prod_{i=1}^2 [0,s_i]$.
Moreover, we let $\alpha \in (- \frac43, -1)$ and $\xi \in \cC^\alpha_\fn(\Gamma)$. We abbreviate $\cC^\alpha_\fn(\Gamma)$ by $\cC^\alpha_\fn$, $H^\gamma_0(\Gamma)$ by $H^\gamma_0$, etc. 
We write $\sigma : \R^2 \rightarrow (0,\infty)$ for the function given by 
\begin{align*}
\sigma(x) = \frac{1}{1+\pi^2 |x|^2}. 
\end{align*}
Additional assumptions are given in \ref{obs:assumption_box}.
}
Remember, see \ref{obs:inverse_laplacian_plus_one}, that $\sigma(\rD) = (1-\Delta)^{-1}$. 

\begin{definition}
\label{def:rough_distributions_even}
For $\beta \in \R$, we define the 
\emph{space of enhanced Neumann distributions}, written $\fX_\fn^\beta$, to be the closure in $\cC_\fn^\beta \times \cC_\fn^{2\beta+2}$ of the set 
\begin{align*}
\{ (\zeta, \zeta \reso \sigma(\rD) \zeta - c) : \zeta \in \cS_\fn, c\in \R\}. 
\end{align*}
We equip $\fX_\fn^\beta$ with the relative topology with respect to $\cC_\fn^\beta \times \cC_\fn^{2\beta +2}$. 
\end{definition}

We will now define the Dirichlet domain of the Anderson Hamiltonian analogously to \cite{AlCh15} did on the torus.

\begin{definition}
\label{def:dirichlet_domain_etc}
Let $\bxi = (\xi, \Xi) \in \fX_\fn^\alpha$. 
For $\gamma \in (0,\alpha + 2)$ we define 
$
\cD_\xi^{\Dir,\gamma} = \{ f\in H^\gamma_0 : f^{\sharp \xi} \in H^{2\gamma}_0 \},
$ 
where 
$
f^{\sharp \xi} := f - f \para \sigma(\rD) \xi .
$ 
Moreover, we define an inner product on $\cD_\xi^{\Dir,\gamma}$, written $\langle \cdot, \cdot \rangle_{\cD_{\xi}^{\Dir,\gamma}}$, by 
$
\langle f, g \rangle_{\cD_{\xi}^{\Dir,\gamma}} = \langle f, g \rangle_{H_0^\gamma} + \langle f^{\sharp \xi}, g^{\sharp \xi} \rangle_{H_0^{2\gamma}}. 
$ 

For $\gamma \in (-\frac{\alpha}{2} , \alpha +2)$ 
we define the space of 
\emph{strongly paracontrolled distributions} by 
$
\fD_{\bxi}^{\Dir,\gamma} = \{ f\in H_0^\gamma: f^{\flat \bxi} \in H_0^2\}, 
$ 
where  
$ 
f^{\flat \bxi} : = f^{\sharp \xi} - B(f,\bxi) 
$ 
and 
$B(f,\bxi)  = 
\sigma(\rD) ( 
 f \Xi 
+ f \arap \xi  
- ( (\Delta -1) f )\para \sigma(\rD) \xi 
- 2 
 \sum_{i=1}^d \partial_{x_i} f \para \partial_{x_i} \sigma(\rD) \xi
)$
(for the paraproducts under the sum, see \ref{obs:besov_spaces_with_mixed_bc}).
We define an inner product on $\fD_\bxi^{\Dir,\gamma}$, written $\langle \cdot, \cdot \rangle_{\fD_\bxi^{\Dir,\gamma}}$, by 
$
\langle f, g \rangle_{\fD_\bxi^{\Dir,\gamma}} = \langle f, g \rangle_{H_0^\gamma} + \langle f^{\flat \bxi}, g^{\flat \bxi} \rangle_{H_0^{2}}. 
$ 
As in the periodic setting, one has $\fD_\bxi^{\Dir,\gamma} \subset H_0^{\alpha+2-}$ for all $\gamma \in (-\frac{\alpha}{2} , \alpha+2)$. We write  $\fD_\bxi^\Dir= \{ f\in H_0^{\alpha+2-} : f^{\flat \bxi} \in H_0^2\}$. 
\end{definition}

We will define the Anderson Hamiltonian on the Dirichlet domain in a similar sense as is done on the periodic domain, however we choose to change the sign in front of the Laplacian as this is more common in literature on the parabolic Anderson model.

\begin{definition}
\label{def:anderson_operator}
Let $\gamma \in (- \frac{\alpha}{2}, \alpha +2)$, $\bxi \in \fX_\fn^\alpha$. 
We define\footnote{The definition needs of course justification to show $H_0^{\gamma-2}$ is really the codomain, this is shown in Theorem \ref{theorem:dirichlet_summary}.} the operator $\sH_{\bxi} : \cD_\xi^{\Dir,\gamma} \rightarrow H_0^{\gamma-2}$ by 
\begin{align*}
\sH_{\bxi} f =  \Delta f + f \diamond \bxi, 
\end{align*}
where $f \diamond \bxi = f \para \xi +  f^{\sharp \xi} \reso \xi  + \cR(f,\sigma(\rD)\xi,\xi) + f \Xi + f \arap \xi$ and 
$
\cR(f,g,h) := (f \para g) \reso h - f (g \reso h) . 
$ 
\end{definition}

We state the main results about the spectrum of the Anderson Hamiltonian, on its Dirichlet domain. These results are analogous to the Anderson Hamiltonian on the torus \cite{AlCh15} (one can just read the theorem below without the Dirichlet and Neumann notations, i.e., the sub- or superscripts ``$0$,$\fd$,$\fn$'', and with the spaces interpreted to be defined on a torus).
Moreover, they are similar to the results of \cite{La19}, which proof is based on the theory of regularity structures. 

\begin{theorem}
\label{theorem:dirichlet_summary}
For $\gamma \in (- \frac{\alpha}{2}, \alpha +2)$ there exists a $C>0$ such that 
\begin{align}
\label{eqn:bound_on_dirichlet_norm_of_H}
\| \sH_\bxi f \|_{H_0^{\gamma -2 }} \le C \|f\|_{\cD_\xi^{\Dir,\gamma}} (1+\|\bxi\|_{\fX_\fn^\alpha})^2 \qquad (f \in \cD_\xi^{\Dir,\gamma}, \bxi \in \fX_\fn^\alpha). 
\end{align}
$\sH_\bxi( \fD_\bxi^\Dir) \subset L^2 $ and $\sH_\bxi : \fD_\bxi^\Dir  \rightarrow L^2$ is closed and self-adjoint as an operator on $L^2$, and $\fD_\bxi^\Dir$ is dense in $L^2$. 
There exist 
$\lambda_1(\Gamma,\bxi) > \lambda_2(\Gamma,\bxi) \ge \lambda_3(\Gamma,\bxi) \ge \cdots $ 
 such that $\limn \lambda_n(\Gamma,\bxi) =-\infty$, $\sigma(\sH_\bxi) = \sigma_p(\sH_\bxi) = \{ \lambda_n(\Gamma,\bxi) : n\in\N\}$ and 
 $\#\{n\in\N: \lambda_n(\Gamma,\bxi) = \lambda\} = \dim \ker (\lambda - \sH_\bxi)<\infty$ for all $\lambda \in \sigma(\sH_\bxi)$. 
  One has
\begin{align*}
\fD_\bxi^\fd = \bigoplus_{\lambda \in \sigma(\sH_\bxi)} \ker ( \lambda - \sH_\bxi). 
\end{align*} 
There exists an $M>0$ such that for all $n\in\N$ and $\bxi, \btheta \in \fX_\fn^\alpha$ 
\begin{align}
\label{eqn:locally_lipschitz_eigenvalue_of_dirichlet}
|\lambda_n(\Gamma,\bxi)  - \lambda_n(\Gamma,\btheta) |
\le M 
 \|\bxi - \btheta \|_{\fX_\fn^\alpha} 
\left( 1 + \|\bxi  \|_{\fX_\fn^\alpha} + \| \btheta \|_{\fX_\fn^\alpha} \right)^M . 
\end{align}
With the notation $\sqsubset $ for ``is a linear subspace of'', 
\begin{align}
\label{eqn:min-max_dir}
\lambda_n(\Gamma,\bxi) = 
\sup_{\substack{F\sqsubset \fD_\bxi^\Dir \\ \dim F = n}}
\inf_{\substack{\psi \in F \\ \|\psi\|_{L^2}=1}} \langle \sH_\bxi \psi, \psi\rangle_{L^2}
\end{align}
In particular, 
\bfold 
\lambda_1(\Gamma,\bxi)
= \sup_{\psi\in \fD_\bxi^\Dir : \|\psi\|_{L^2} =1} \langle \sH_\bxi \psi, \psi\rangle_{L^2} 
\efold. 
\end{theorem}

\begin{remark}
\label{remark:neumann_statements}
Let us mention that in an analogous way one can state (and prove) the same statement for the operator with Neumann boundary conditions by replacing ``$\Dir$'' by ``$\Even$'' and ``$H_0$'' by ``$H_\fn$''.  
\end{remark}

\begin{remark}
In \cite{AlCh15} it is pointed out  that in \eqref{eqn:min-max_dir} one may replace $\fD_\bxi^\fd$ by $\cD_\xi^\gamma$ for $\gamma \in (\frac23, \alpha+2)$, and 
$\langle \sH_\bxi \psi, \psi\rangle_{L^2}$ by ${}_{H_0^{-\gamma}}\langle \sH_\bxi \psi , \psi \rangle_{H_0^\gamma}$, where ${}_{H_0^{-\gamma}} \langle \cdot, \cdot \rangle_{H_0^\gamma} : H_0^{-\gamma} \times H_0^\gamma \rightarrow \R $ is the continuous bilinear map (see \cite[Theorem 2.76]{BaChDa11}) given by 
\begin{align*}
{}_{H_0^{-\gamma}} \langle f, g \rangle_{H_0^\gamma} = \sum_{ \substack{ i,j \in \N_{-1} \\ |i-j|\le 1}} \langle \Delta_i  f, \Delta_j g \rangle_{L^2}. 
\end{align*} 
This is done for the periodic setting, but the arguments can easily be adapted to our setting. 
\begin{calc}
Indeed, first one shows that $\fD_\bxi^\Dir$ is dense in $\cD_\xi^{\Dir,\gamma}$: 
$\cS_0$  and thus $L^2$ is dense in $H_0^{\gamma-2}$ (see \cite[Theorem 2.74]{BaChDa11} and Theorem \ref{theorem:equivalences_H_0_spaces}), therefore for $a \notin \sigma(\sH_\bxi)$ and $\cG_a = (a- \sH_\bxi)^{-1}$,  $\fD_\bxi^{\Dir} = \cG_a L^2 $ is dense in $\cD_\xi^{\Dir,\gamma} = \cG_a H_0^{\gamma -2}$.
This proves \ref{item:strong_domain_dense_in_domain_dir_short}. 

With this it is sufficient to use the continuity of the map
\begin{align*}
\cD_\xi^{\Dir,\gamma} \times \cD_\xi^{\Dir,\gamma}  \rightarrow \R, \qquad (\psi, \varphi) \mapsto {}_{H_0^{-\gamma}} \langle \sH_\bxi \psi , \varphi \rangle_{H_0^\gamma}, 
\end{align*}
which follows from the following bound 
(observe that $\gamma - 2 < - \gamma$ and use \eqref{eqn:bound_on_dirichlet_norm_of_H})
\begin{align*}
| {}_{H_0^{-\gamma}} \langle \sH_\bxi \psi , \varphi \rangle_{H_0^\gamma} |
\le \| \sH_\bxi \psi \|_{H_0^{\gamma-2}} \| \varphi \|_{H_0^\gamma}
\lesssim (1+\|\bxi\|_{\fX^\alpha})^2 \| \psi\|_{\cD_\xi^{\Dir,\gamma}} \|\varphi\|_{\cD_\xi^{\Dir,\gamma}}. 
\end{align*}
\end{calc}
\end{remark}

\begin{obs}
\label{obs:rough_distributions_for_H_elements}
Let $\eta \in L^2$ (which equals $H_\fn^0$, see \ref{obs:H_0_equals_L_2}).
By Theorem \ref{theorem:schauder} $\sigma(\rD)\eta \in H_\fn^2$, which is included in $\cC^1_\fn$ by Theorem \ref{theorem:inclusion_of_H_into_C}. 
Then by Theorem \ref{theorem:bony_estimates}, $\eta \reso \sigma(\rD) \eta \in H_\fn^1$.
Moreover, if $\eta_\epsilon \rightarrow \eta$ in $L^2$, then $ \eta_\epsilon \reso \sigma(\rD) \eta_\epsilon \rightarrow \eta \reso \sigma(\rD) \eta$ in $H_\fn^1$ (by the same theorems). Hence, by Theorem \ref{theorem:inclusion_of_H_into_C} we obtain the following convergence in $\fX_\fn^\alpha$ for all $\alpha \le -1$
\begin{align*}
(\eta_\epsilon , \eta_\epsilon \reso \sigma(\rD) \eta_\epsilon) \rightarrow 
(\eta , \eta \reso \sigma(\rD) \eta).
\end{align*}
\begin{calc}
Indeed
\begin{align*}
\| \eta \reso \sigma(\rD) \eta - \eta_\epsilon \reso \sigma(\rD) \eta_\epsilon \|_{H_\fn^2}
\lesssim \|\eta - \eta_\epsilon\|_{H^0_\fn} \|\eta\|_{H^0_\fn} + \|\eta - \eta_\epsilon\|_{H^0_\fn} \|\eta_\epsilon \|_{H^0_\fn}. 
\end{align*} 
\end{calc}
We write $\lambda_n(\Gamma, \eta)  = \lambda_n(\Gamma, (\eta, \eta \reso \sigma(\rD) \eta ))$. 
\end{obs}

By \ref{obs:rough_distributions_for_H_elements} and the continuity of $\bxi \mapsto \lambda_n(\Gamma,\bxi)$, see \eqref{eqn:locally_lipschitz_eigenvalue_of_dirichlet} in Theorem \ref{theorem:dirichlet_summary}, we obtain the following lemma. 

\begin{lemma}
\label{lemma:continuity_of_eigenvalues_on_H_00}
The map 
$ 
L^2(\Gamma) \rightarrow \R,  \eta \mapsto \lambda_n(\Gamma,\eta) 
$ 
is continuous. 
\end{lemma}

\begin{obs}
\label{obs:smooth_potentials}
Let $\zeta \in \cS_\fn^\infty$. 
Then $\bzeta := (\zeta, \zeta \reso \sigma(\rD) \zeta )\in \fX_\fn^\beta$, 
$f \para \sigma(\rD) \zeta \in H_0^\beta$ for all $\beta \in \R$ and $B(f,\bzeta) \in H_0^2$ and $f\in H_0^\gamma$ with $\gamma\in(0,1)$ (use Theorems \ref{theorem:schauder}, \ref{obs:sigma_one_plus_x_sq_inverse_satisfies_conditions} and  \ref{theorem:bony_estimates}). 
\begin{calc}
Clearly $\zeta \reso \sigma(\rD) \zeta \in \cC_\fn^{2\beta -2}$ for all $\beta \in \R$ as $\zeta \in \cC_\fn^\alpha$ for all $\alpha \in \R$. 
We have 
\begin{align*}
\| f \para \sigma(\rD) \zeta \|_{H_0^\beta} 
\lesssim \|f\|_{H_0^\gamma} \|\sigma(\rD) \zeta\|_{\cC_\fn^{\beta-\delta}}
\lesssim \|f\|_{H_0^\gamma} \| \zeta\|_{\cC_\fn^{\beta-\delta-2}}.
\end{align*}
Let us check each of the individual terms in $B$: 
\begin{align*}
B(f,\bzeta)  = 
\sigma(\rD) ( 
 f \zeta \reso \sigma(\rD) \zeta
+ f \arap \zeta  
- ( (\Delta -1) f )\para \sigma(\rD) \zeta 
- 2 
 \sum_{i=1}^d \partial_{x_i} f \para \partial_{x_i} \sigma(\rD) \zeta
).
\end{align*}
$f (\zeta \reso \sigma(\rD) \zeta) \in H_0^{\gamma -\delta}$ for all $\delta>0$, 
$f \arap \zeta \in H_0^\beta$ for all $\beta \in \R$,
$(\Delta-1) f\in H_0^{\gamma-2}$, $\sigma(\rD) \zeta \in \cC_\fn^\beta$ for all $\beta \in \R$, therefore 
$( (\Delta -1) f )\para \sigma(\rD) \zeta \in H_0^\beta$ for all $\beta \in \R$ and similarly $\partial_{x_i} f \para \partial_{x_i} \sigma(\rD) \zeta \in H_0^\beta$ for all $\beta \in \R$ (observe that $\partial_{x_i} f$ and $\partial_{x_i} \sigma(\rD) \zeta$ are in Besov spaces with mixed boundary conditions). 
Therefore $B(f,\bzeta) \in H_0^{2+\gamma-\delta}$, and so if we choose $\delta$ small enough we obtain that $B(f,\bzeta)$ is an element of $H_0^2$. 
\end{calc}
Therefore, for all $\gamma \in (0,1)$, $\cD_\zeta^{\Dir, \gamma} = H_0^{2\gamma}$ and $ \fD_\bzeta^{\Dir,\gamma}= H_0^2$ and for $f\in H_0^\gamma$,  
$
f \reso \zeta =  f^{\sharp \zeta} \reso \zeta  + \cR(f,\sigma(\rD)\zeta,\zeta) + f (\zeta \reso \sigma(\rD) \zeta)$, so that 
\begin{align}
\label{eqn:operator_zeta}
\sH_\zeta f := \Delta f + f \zeta= \sH_\bzeta f . 
\end{align}
Now suppose $\zeta \in L^\infty \subset \cC_\fn^\infty$. Then $\bzeta := (\zeta, \zeta \reso \sigma(\rD) \zeta )\in \fX_\fn^0$, but the Bony estimates give $f \para \sigma(\rD) \zeta \in H^{2-}_0$ (and not $\in H^2_0$). 
Nevertheless, by the Kato-Rellich theorem \cite[Theorem X.12]{ReSi75}
on the domain $H_0^2$ the operator $\sH_\zeta$ defined as in \eqref{eqn:operator_zeta} is self-adjoint. 
As the injection map $H_0^2 \rightarrow L^2$ is compact (see Theorem \ref{theorem:compact_embeddings_besov_spaces}), every resolvent is compact. 
Hence by 
the Riesz-Schauder theorem \cite[Theorem VI.15]{ReSi75} 
and the Hilbert-Schmidt theorem \cite[Theorem VI.16]{ReSi75}
 there exist $\lambda_1(\Gamma,\zeta) \ge \lambda_2(\Gamma,\zeta) \ge \cdots $ 
 such that $\sigma(\sH_\zeta) = \sigma_p(\sH_\zeta) = \{ \lambda_n(\Gamma,\zeta) : n\in\N\}$ and 
 $\#\{n\in\N: \lambda_n(\Gamma,\zeta) = \lambda\} = \dim \ker (\lambda - \sH_\zeta)<\infty$ for all $\lambda \in \sigma(\sH_\zeta)$. 
Moreover, by Fischer's principle
\cite[Section 28, Theorem 4, p. 318]{La02}\footnote{In this reference the operator is actually assumed to be compact and symmetric, whereas we apply it to $\sH_\bxi$. But the compactness is only assumed to guarantee that the spectrum is countable and ordered, so that the arguments still hold.}
 and Lemma \ref{lemma:min_max_over_smooth}
\begin{align}
\notag 
\lambda_n(\Gamma,\zeta) 
& = 
\sup_{\substack{F\sqsubset H_0^2 \\ \dim F = n}}
\inf_{\substack{\psi \in F \\ \|\psi\|_{L^2}=1}} \langle \sH_\zeta \psi, \psi\rangle_{L^2} \\
\label{eqn:min_max_for_smooth}
& = 
\sup_{\substack{F\sqsubset C_c^\infty \\ \dim F = n}}
\inf_{\substack{\psi \in F \\ \|\psi\|_{L^2}=1}} \int -| \nabla \psi |^2 + \zeta \psi^2 . 
\end{align}
\end{obs}

The proof of Theorem \ref{theorem:dirichlet_summary} follows from the results of the Anderson Hamiltonian on the torus with the help of Lemma \ref{lemma:H_and_tilde_H_and_their_domains}. The proof is written below Lemma \ref{lemma:H_and_tilde_H_and_their_domains}. We may restrict ourselves to the case $\Gamma= Q_L$. 

\begin{obs}
\label{obs:assumption_box}
\textbf{For the rest of this section $y=0$ and $b_i=L$ for all $i$, i.e., $\Gamma = Q_L = [0,L]^2$.}
\end{obs}

\begin{obs}
\label{obs:odd_and_even_for_distributions}
For $\fq \in \{-1,1\}^d$ and $w\in \cS'$ we write $l_\fq w$ for the element in $\cS'$ given by $\langle l_\fq w, \varphi \rangle = \langle w, \varphi (\fq \circ \cdot) \rangle$ for $\varphi \in \cS$. 
Then $w$ is odd if and only if $w = (\prod \fq ) l_\fq w$ for all $\fq \in \{-1,1\}^d$ and $w$ is even if and only if $w = l_\fq w$ for all $\fq \in \{-1,1\}^d$. 
\end{obs}

\begin{lemma}
\label{lemma:H_and_tilde_H_and_their_domains}
Let $\bxi \in \fX_\fn^\alpha$. 
Let $\frac23 < \gamma < \alpha+2$. 
Write $\overline \bxi = (\overline \xi, \overline \Xi)$, 
$\cD_{\overline \xi}^\gamma= \cD_{\overline \xi}^\gamma (\T_{2L}^d)$, 
$\fD_{\overline \bxi}^\gamma= \fD_{\overline \bxi}^\gamma (\T_{2L}^d)$. 
\begin{enumerate}
\item 
\label{item:H_and_tilde_H_and_their_domains}
$\widetilde{\cD_\xi^{\Dir, \gamma}}  = 
\{ w\in \cD_{\overline \xi}^\gamma : w \mbox{ is odd}\}$,
$ \widetilde{\fD_\bxi^{\Dir, \gamma}}  = \{ w\in \fD_{\overline \bxi}^\gamma : w \mbox{ is odd}\}$,
$\widetilde{\sH_\bxi f} = \sH_{\overline \bxi} \tilde f \mbox{ and } \| f\|_{\cD_\xi^{\Dir,\gamma} } \eqsim \|\tilde f\|_{\cD_{\overline \xi}^\gamma } \mbox{ uniformly for all }f\in \cD_\xi^{\Dir,\gamma}$ and
$\| f\|_{\fD_{ \bxi}^{\Dir,\gamma} } \eqsim \|\tilde f\|_{\fD_{\overline \bxi}^\gamma } \mbox{ uniformly for all }f\in \fD_\bxi^{\Dir,\gamma}$. 
\item
\label{item:H_on_domain_contained_in}
$\sH_\bxi( \cD_\xi^{\Dir,\gamma}) \subset H_0^{\gamma-2} $, $\sH_\bxi( \fD_\bxi^{\Dir,\gamma}) \subset L^2$. 

\item 
\label{item:operator_commutes_with_l_fq}
$\sH_{\overline \bxi} ( l_\fq f) = l_\fq \sH_{\overline \bxi} f$ for all $f\in \cD_{\overline \xi}^\gamma$ and $\fq \in \{-1,1\}^2$. 

\item 
\label{item:resolvents_compared}
$\sigma(\sH_\bxi) \subset \sigma(\sH_{\overline \bxi})$ (for the operators either on the $\cD$ or $\fD$ domains) and for all $a\in \C \setminus \sigma(\sH_{\overline \bxi})$ the inverse of $a- \sH_\bxi : \fD_\bxi^\fd \rightarrow L^2$ is self-adjoint and compact. 
\item 
\label{item:strong_domain_dense_in_domain_dir_short}
$\fD^{\Dir}_\bxi$ is dense in $\cD_\xi^{\Dir,\gamma}$ and $\cD_\xi^{\Dir,\gamma}$ is dense in $L^2$. 
\begin{hide}
\item 
\label{item:min-max_with_cD_xi_dir}
\begin{align}
\label{eqn:min-max_with_cD_xi_dir}
\sup_{\substack{F\sqsubset \fD_\bxi^\Dir \\ \dim F = n}}
\inf_{\substack{\psi \in F \\ \|\psi\|_{L^2}=1}} \langle \sH_\bxi \psi, \psi\rangle_{L^2}
=
\sup_{ \substack{F\sqsubset \cD_\xi^{\Dir,\gamma} \\ \dim F = n}}
\inf_{\substack{\psi \in F \\ \|\psi\|_{L^2}=1}}  {}_{H_0^{-\gamma}}\langle \sH_\bxi \psi , \psi \rangle_{H_0^\gamma}. 
\end{align}
\end{hide}
\end{enumerate}
\end{lemma}
\begin{proof}
\ref{item:H_and_tilde_H_and_their_domains} follows from the identities \eqref{eqn:tilde_of_products}, $\widetilde{ f^{\sharp \xi}} = {\tilde f}^{\sharp \overline{\xi}}$, 
$\widetilde{B(f,\bxi)} = B(\tilde f , \overline \bxi)$, $\widetilde{f^{\flat \bxi}} = {\tilde f}^{\flat \overline \bxi}$ and because $\|\tilde g\|_{H^\gamma} \eqsim \|g\|_{H_0^\gamma} $ for all $\gamma \in \R$ and $g \in H_0^\gamma([0,L]^d)$  (indeed, $\|g\|_{B_{2,2}^{\Dir,\gamma}} =\|\tilde g\|_{B_{2,2}^{\gamma}}$ by definition and $\|\cdot\|_{H_0^\gamma} \eqsim \|\cdot \|_{B_{2,2}^{\Dir,\gamma}} $ and $\| \cdot \|_{B_{2,2}^{\gamma}} \eqsim \| \cdot\|_{H^\gamma}$ by Theorems \ref{theorem:equivalent_spaces_to_H_alpha_R_d_and_torus} and \ref{theorem:equivalences_H_0_spaces}). 

\ref{item:H_on_domain_contained_in} follows from \ref{item:H_and_tilde_H_and_their_domains} as $\sH_{\overline \bxi} (\cD_{\overline \xi}^\gamma) \subset H^{\gamma-2}$ and $\sH_{\overline \bxi} (\fD_{\overline \bxi}) \subset H^0$ (see \cite{AlCh15}). 

\ref{item:operator_commutes_with_l_fq} follows by a straightforward calculation; 
use that $\cF(l_\fq f) = l_\fq \cF(f)$, $l_\fq \rho_i = \rho_i$, $l_\fq \overline \xi = \overline \xi$ and $l_\fq \overline \Xi = \overline \Xi$ for $\fq \in \{-1,1\}^2$. 

\ref{item:resolvents_compared}
Let $a \in \C$ be such that $a- \sH_{\overline \bxi}$ has a bounded inverse $\cR_a$. 
By \ref{item:operator_commutes_with_l_fq} $(a- \sH_{\overline \bxi}) f$ is odd if and only if $f$ is odd, indeed, if $(a- \sH_{\overline \bxi}) f$ is odd, then $(a-\sH_{\overline \bxi}) [f - (\prod \fq) l_\fq f]=0$ (see  \ref{obs:odd_and_even_for_distributions}) and thus $f= (\prod \fq) l_\fq f$ . 
Hence $a- \sH_\bxi$ has a bounded inverse $\cR_a^\Dir$ such that $\widetilde{\cR_a^\Dir h} = \cR_a \tilde h$. 
 From the fact that $\cR_a$ is self-adjoint and compact it follows that $\cR_a^\Dir$ is too. 
 \begin{calc}
Because when $A$ is closed/open, then $\tilde A$ is closed/open. 
 \end{calc}

$\cS_0$  and thus $L^2$ is dense in $H_0^{\gamma-2}$ (see \cite[Theorem 2.74]{BaChDa11} and Theorem \ref{theorem:equivalences_H_0_spaces}), therefore for $a \notin \sigma(\sH_\bxi)$ and $\cG_a = (a- \sH_\bxi)^{-1}$,  $\fD_\bxi^{\Dir} = \cG_a L^2 $ is dense in $\cD_\xi^{\Dir,\gamma} = \cG_a H_0^{\gamma -2}$.
That $\cD_\xi^{\Dir,\gamma}$ is dense in $L^2$ follows from the periodic counterpart, which is proven in \cite[Lemma 4.12]{AlCh15}.
This proves \ref{item:strong_domain_dense_in_domain_dir_short}. 
%
\begin{hide} 
As $\fD_\bxi^\Dir \subset \cD_\xi^{\Dir,\gamma}$, we have that $\lambda_n^\bxi$ is $\ge$ to the right-hand side of \eqref{eqn:min-max_with_cD_xi_dir}. 
As $\fD_\bxi^\Dir$ is dense in $\cD_\xi^{\Dir,\gamma}$, we obtain the equality \eqref{eqn:min-max_with_cD_xi_dir} by the continuity of 
\begin{align*}
\cD_\xi^{\Dir,\gamma} \times \cD_\xi^{\Dir,\gamma}  \rightarrow \R, \qquad (\psi, \varphi) \mapsto {}_{H_0^{-\gamma}} \langle \sH_\bxi \psi , \varphi \rangle_{H_0^\gamma}, 
\end{align*}
which follows from the following bound 
(observe that $\gamma - 2 < - \gamma$ and use \eqref{eqn:bound_on_dirichlet_norm_of_H})
\begin{align*}
| {}_{H_0^{-\gamma}} \langle \sH_\bxi \psi , \varphi \rangle_{H_0^\gamma} |
\le \| \sH_\bxi \psi \|_{H_0^{\gamma-2}} \| \varphi \|_{H_0^\gamma}
\lesssim (1+\|\bxi\|_{\fX^\alpha})^2 \| \psi\|_{\cD_\xi^{\Dir,\gamma}} \|\varphi\|_{\cD_\xi^{\Dir,\gamma}}. 
\end{align*}
This proves \ref{item:min-max_with_cD_xi_dir}. 
\end{hide}
\end{proof}

\begin{proof}[Proof of Theorem \ref{theorem:dirichlet_summary}]
By Lemma \ref{lemma:H_and_tilde_H_and_their_domains} it follows that $\sH_\bxi$ is a closed densely defined symmetric operator and that $\sigma(\sH_\bxi) \subset \sigma(\sH_{\overline \bxi})$ so that $\sH_\bxi$ is indeed self-adjoint (see \cite[Theorem X.2.9]{Co07}). 
As the resolvents are compact, the statements in Theorem \ref{theorem:dirichlet_summary} up to \eqref{eqn:locally_lipschitz_eigenvalue_of_dirichlet} follow by the 
Riesz-Schauder theorem \cite[Theorem VI.15]{ReSi75} 
and the Hilbert-Schmidt theorem \cite[Theorem VI.16]{ReSi75}
because of the following identity, where $\cR_\mu = (\mu- \sH_\bxi)^{-1}$, 
\begin{align*}
\sigma(\sH_\bxi) = \sigma_p(\sH_\bxi) = \{   \mu -\tfrac{1}{\lambda} : \lambda \in \sigma_p( \cR_\mu ) \setminus \{0\} \},
\end{align*}
this means that $\lambda - \cR_\mu$ is boundedly invertible (or injective)  if and only if $ \mu - \tfrac{1}{\lambda}  - \sH_\bxi$ is, and in turn follows from the identity
\begin{align*}
\lambda( \mu - \tfrac{1}{\lambda}  - \sH_\bxi) 
&= \lambda (\mu - \sH_\bxi) -1 = (\lambda - \cR_\mu) (\mu - \sH_\bxi) \\
&=   (\mu - \sH_\bxi)\lambda -1  = (\mu - \sH_\bxi) (\lambda - \cR_\mu).
\end{align*}
As every eigenvalue of $\sH_\bxi$ is an eigenvalue of $\sH_{\overline \bxi}$ which is locally lipschitz in the analogues sense of \eqref{eqn:locally_lipschitz_eigenvalue_of_dirichlet}, also \eqref{eqn:locally_lipschitz_eigenvalue_of_dirichlet} holds by the equivalences of norms in Lemma \ref{lemma:H_and_tilde_H_and_their_domains}\ref{item:H_and_tilde_H_and_their_domains}. 
\eqref{eqn:min-max_dir} follows from 
\begin{hide}
Lemma \ref{lemma:H_and_tilde_H_and_their_domains}\ref{item:min-max_with_cD_xi_dir} and \end{hide}
 Fischer's principle \cite[Section 28, Theorem 4, p. 318]{La02}. 
\begin{calc}
See also Theorem \ref{theorem:spectrum_under_compact_resolvent}.
\end{calc}
That $\lambda_1 > \lambda_2$, or in other words, that the first eigenvalue is simple, follows from \cite[Theorem XIII.44]{ReSi78}. 
The only condition to prove for that theorem is that the semigroup $e^{t\sH_\bxi}$ is  positivity improving, or differently called the strong maximum principle for $e^{t\sH_\bxi}$. The strategy to obtain this we borrow from \cite[Theorem 5.1]{CaFrGa17}. With $u_t := e^{t \sH_\bxi} u_0$, the map $(t,x) \mapsto u_t(x)$ is the solution to the parabolic Anderson model $\partial_t u = \Delta u + u \diamond \bxi$, hence satisfies $\sup_{s\in [0,t]} \| u_s \|_{B_{\infty,\infty}^{\fd,1-\epsilon}} <\infty$ for all $\epsilon>0$ (see \cite{GuImPe15}, the extension to Dirichlet boundary conditions follows similar as the extension of the operator) and $u_t = P_t u_0 + \int_0^t P_{t-s} (u_s \diamond \bxi) \dd s$, where $P_t u_0 (x) = p_t * u_0 (x)$ and $p_t$ the standard heat kernel $p_t(x) = (2\pi t)^{-\frac{d}{2}} e^{-\frac{|x|^2}{2t}}$. The next step is to prove that $P_t u_0$ is larger than the supremum norm of $\int_0^t P_{t-s} (u_s \diamond \bxi) \dd s$. 
In \cite{CaFrGa17} it is shown that for all $\rho>0$ there exists a $t_\rho$ such that  $P_t \1_{B(x,\delta)}\ge \frac14 \1_{B(x, \delta +\rho t)}$ for $t\in (0,t_\rho]$. 
\begin{calc}
First observe that $P_t \1_{B(x,\delta)}(z) = \int_{B(x,\delta)} p_t(z-y) \dd y = \int_{B(z-x,\delta)} p_t(y) \dd y$, and thus (with $e_1$ the unit vector with $1$ at the first coordinate)
\begin{align*}
\inf_{z\in B(x,\delta+t)} P_t ( \1_{B(x,\delta)})(z) 
& = \inf_{z\in B(0,\delta+t)} P_t ( \1_{B(0,\delta)})(z)  \\
& = P_t ( \1_{B(0,\delta)}) \big( (\delta+t) e_1\big) \\
& = \int_{B \big( \frac{\delta + t}{  \sqrt{t}} e_1, \frac{\delta}{\sqrt{t}} \big) } p_1(y) \dd y \xrightarrow{t \downarrow 0} 
\int_{(0,\infty) \times \R } p_1(y) \dd y
= \frac12 . 
\end{align*}
because $B \big( \frac{\delta + t}{  \sqrt{t}} e, \frac{\delta}{\sqrt{t}} \big) \uparrow (0,\infty) \times \R$ as $t \downarrow 0$. 
\end{calc}
On the other hand, one can prove that for $\epsilon \in (0,1)$ there exists a $C>0$ such that $\| \int_0^t P_{t-s} ( u_s \diamond \bxi)  \dd s \|_{B_{\infty,\infty}^{\fd,\epsilon}} \le C t^{1-\epsilon}$. 
\begin{calc}
Indeed, with \cite[Lemma A.7]{GuImPe15}
\begin{align*}
\|\int_0^t P_{t-s} ( u_s \diamond \bxi) \dd s \|_\infty 
& \le \int_0^t \| P_{t-s} ( u_s \diamond \bxi) \|_{B_{\infty,\infty}^{\fd,\epsilon}} \dd s \\
& \lesssim \int_0^t (t-s)^{-\epsilon} \| u_s \diamond \bxi \|_{B_{\infty,\infty}^{\fd,-1-\epsilon}} \dd s \\
& \lesssim t^{1-\epsilon} \int_0^1 (1-s)^{-\epsilon} \dd s 
\sup_{s\in [0,t]} \| u_s \|_{B_{\infty,\infty}^{\fd,1-\epsilon}} 
\| \bxi \|_{\fX_\fn^{-1-\epsilon}} .
\end{align*}
Now observe that for $\epsilon \in (0,1)$ the integral $\int_0^1 (1-s)^{-\epsilon} \dd s$ equals $(1-\epsilon)^{-1}$. 
\end{calc}
Hence we can choose $t_0 \in (0,t_\rho)$ such that $\begin{calc} \|\int_0^{t_\rho} P_{t_1-s} ( u_s \diamond \bxi) \dd s \|_\infty \le \end{calc} C t_0^{1-\epsilon} \le \frac18$. 
This implies that $u_t \ge \frac18$ on $B(x,\delta + \rho t_0 )$. 
Let $T,\rho>0$, by choosing $n$ such that $\frac{T}{n} \le t_0$, by repeating the argument we have $u_T \ge (\frac18)^n$ on $B(x, \delta + \rho T)$. As this holds for arbitrary $\rho>0$, this implies that $U_T$ is strictly positive everywhere. 
\end{proof}


\section{Enhanced white noise}
\label{section:white_noise}

In this section we prove Theorem \ref{theorem:convergence_in_enhanced_space_to_white_noise}; we first recall a definition and introduce notation. 


\begin{definition}
A \emph{white noise} on $\R^d$ is a random variable $W : \Omega \rightarrow \cS'(\R^d,\R)$ such that  for all $f\in \cS(\R^d,\R)$ the random variable $\langle W,f \rangle$ is a centered Gaussian random variable. 
\end{definition}

\begin{obs}
\label{obs:bounded_lin_op_corresponding_to_white_noise}
Because $\| \langle W, f \rangle \|_{L^2(\Omega,\P)} = \|f\|_{L^2(\R^d)}$, the function $ f\mapsto \langle W, f \rangle$ extends to a bounded linear operator $\cW : L^2(\R^d) \rightarrow L^2(\Omega,\P)$  such that for all $f\in L^2(\R^d)$, $\cW f$ is a complex Gaussian random variable,  $\cW \overline f = \overline{ \cW f}$ and $\E [ \cW f \overline{\cW g} ] = \langle f, g \rangle_{L^2}$ for all $f, g\in L^2(\R^d)$. 
\end{obs}

\begin{obs}
\label{obs:overview_white_noise_section}
Let $W$ be a white noise on $\R^2$ and $\cW$ be as in \ref{obs:bounded_lin_op_corresponding_to_white_noise}. 
For the rest of this section we fix $L>0$. 
Unless mentioned otherwise $\tau \in C_c^\infty(\R^d, [0,1])$ is an even function that is equal to $1$ on a neighbourhood of $0$. 
Define $\xi_{L,\epsilon} \in \cS_\fn([0,L]^d)$ by
(for $\langle \cW, \fn_{k,L} \rangle$, we interpret $\fn_{k,L}$ to be the function in $L^2(\R^d)$ being equal to $\fn_{k,L}$ on $[0,L]^d$ and equal to $0$ elsewhere)
\begin{align}
\label{eqn:xi_epsilon_dir}
\xi_{L,\epsilon}  = 
 \sum_{k\in \N_0^d} \tau( \tfrac{\epsilon}{L} k) \langle \cW, \fn_{k,L} \rangle  \fn_{k,L} . 
\end{align}
For $k\in\N_0^d$ define $Z_k := \langle \cW, \fn_{k,L} \rangle$. Then $Z_k$ is a (real) normal random variable with 
\begin{align}
\label{eqn:first_definition_Z_k}
\E[Z_k]=0,  \qquad \E[Z_k Z_l]= \delta_{k,l}. 
\end{align}

Before we state the convergence to the enhanced white noise, let us discuss our choice of regularization \eqref{eqn:xi_epsilon_dir}.
We use the regularisation by means of a Fourier multiplier, as in \cite{AlCh15}. 
This basically means we `project' the white noise on the Neumann space on the box and then take the regularisation corresponding to a Fourier multiplier. 
Another option is to consider mollified white noise on the full space by convolution and then project the white noise on the Neumann space. 
In a future work by K\"onig, Perkowski and van Zuijlen, it will be shown that both choices lead to the same limiting object (up to a constant, by using techniques from Section~\ref{section:proof_white_noise_conv}). 
This also confirms that our construction of the Anderson Hamiltonian with enhanced white noise agrees with the construction of the Anderson Hamiltonian in \cite{La19}, where the Anderson Hamiltonian is considered a limit of the operators with mollified white noise as potentials. 
\end{obs}

\begin{theorem}
\label{theorem:convergence_in_enhanced_space_to_white_noise}
Let $d=2$. For all $\alpha <-1$
there exists a $\bxi_L \in \fX_\fn^\alpha$ such that  the following convergence holds almost surely in $\fX_\fn^\alpha$, i.e., on a measurable set $\Omega_L$ with $\P(\Omega_L)=1$ 
\begin{align}
\label{eqn:limit_of_mollifiers}
\lim_{\epsilon \downarrow 0, \epsilon \in \mathbb{Q}\cap (0,\infty)} (\xi_{L,\epsilon}, \xi_{L,\epsilon} \reso \sigma(\rD) \xi_{L,\epsilon} - c_\epsilon  ) = \bxi_L, 
\end{align}
where $c_\epsilon= \tfrac{1}{2\pi} \log(\tfrac{1}{\epsilon}) + c_\tau \in \R$ and $c_\tau$ only depends on $\tau$. $\bxi_L$ does not depend on the choice of $\tau$. 
$\xi_L$ is a white noise in the sense that for $\varphi, \psi \in \cS_\fn (Q_L)$, $\xi_L(\varphi)$ and $\xi_L(\psi)$ are Gaussian random variables with 
\begin{align}
\label{eqn:white_noise_characterising_expectation}
\E [ \xi_L(\varphi) ] =0, \qquad 
\E [ \xi_L(\varphi) \xi_L(\psi)] = \langle \varphi, \psi \rangle_{L^2([0,L]^d)}. 
\end{align}
 Moreover, for $\varphi \in C_c^\infty(Q_L)$ one has almost surely (i.e., on $\Omega_L$)
\begin{align*}
 \langle \xi_{L} , \varphi \rangle
 = \lim_{\epsilon\downarrow 0} \langle \xi_{L,\epsilon} , \varphi \rangle 
 = \sum_{k\in \N_0^d}  \langle \cW, \fn_{k,L} \rangle  \langle \fn_{k,L}, \varphi \rangle 
 = \langle W, \varphi \rangle. 
\end{align*}
Hence, for every $L>0$ the $W$ viewed as an element of $\cD'(Q_L)$ extends almost surely uniquely to a $\xi_L$ in $\cC_\fn^\alpha$. 
\end{theorem}

Instead of taking $Q_L$ as an underlying space, we can also take a shift of the box, i.e., $y+Q_L$:

\begin{obs}
\label{obs:shift_enhanced_white_noise}
For $y \in \R^d$ we define 
\begin{align*}
\xi_{L,\epsilon}^y = \cT_y \Big[ 
 \sum_{k\in \N_0^d} \tau( \tfrac{\epsilon}{L} k) \langle \cT_{y}^{-1} \cW, \fn_{k,L} \rangle  \fn_{k,L} \Big].
\end{align*}
If $d=2$, by Theorem \ref{theorem:convergence_in_enhanced_space_to_white_noise} there exists a $\bxi_L^y = (\xi_L^y, \Xi_L^y) \in \fX_\fn^\alpha(y+Q_L)$ such that almost surely
\begin{align}
\label{eqn:limit_of_mollifiers_shift}
\lim_{\epsilon \downarrow 0, \epsilon \in \mathbb{Q}\cap (0,\infty)} (\xi^y_{L,\epsilon}, \xi^y_{L,\epsilon} \reso \sigma(\rD) \xi^y_{L,\epsilon} - \tfrac{1}{2\pi} \log(\tfrac{1}{\epsilon}) ) = \bxi_L^y,
\end{align}
and such that $ \xi_L^y$ is a white noise in the sense described in Theorem \ref{theorem:convergence_in_enhanced_space_to_white_noise} (i.e. $\cT_{-y} \xi_L^y$ satisfies \eqref{eqn:white_noise_characterising_expectation}).
\end{obs}

\textbf{For the rest of this section we fix $L>0$ and drop the subindex $L$; we write $\xi_{\epsilon}= \xi_{L,\epsilon}$ 
and $\fn_k = \fn_{k,L}$.}

\begin{definition}
\label{def:Xi_epsilon}
Define $\Xi_\epsilon \in \cS_\fn(Q_L)$ by 
\begin{align}
\label{eqn:Xi}
\Xi_\epsilon(x) 
 = \xi_\epsilon \reso \sigma(\rD) \xi_\epsilon(x) -  \E [\xi_\epsilon \reso \sigma(\rD) \xi_\epsilon (x) ] .
\end{align}
\end{definition}

The strategy of the proof of the following theorem is rather similar to the proof on the torus in \cite{AlCh15}, but due to the differences of the Dirichlet setting and for the sake of self-containedness we provide the proof.  

\begin{theorem}
\label{theorem:convergence_to_dirichlet_white_noise}
For all $\alpha < - \frac{d}{2}$, $\xi_\epsilon$ converges almost surely as $\epsilon \downarrow 0$ in $\cC_\fn^\alpha$, to the white noise $\xi_L$ (as in Theorem \ref{theorem:convergence_in_enhanced_space_to_white_noise}). 
Moreover, for $d=2$ and all $\alpha < -1$, $\Xi_\epsilon$ converges almost surely as $\epsilon \downarrow 0$ in $\cC_\fn^{2\alpha + 2}$; the limit is independent of the choice of $\tau$. 
\end{theorem}
\begin{proof}
The proof relies on the Kolmogorov-Chentsov theorem (Theorem \ref{theorem:kolmogorov_chentsov}).
Lemma \ref{lemma:gaussian_besov_bound}\ref{item:zeta_bound_hypercontractivity}
shows that the required bound for this theorem can be reduced to bounds on the second moments of $\Delta_i (\xi_\epsilon - \xi_{\delta})(x)$ and $\Delta_i  (\Xi_\epsilon - \Xi_{\delta})(x)$, given in \ref{obs:second_moment_bounds} (the proofs of these bounds are lengthy and therefore postponed to Section \ref{section:proof_white_noise_conv}).
\eqref{eqn:white_noise_characterising_expectation} follows from 
\begin{align*}
\E[\langle \xi_\epsilon, \varphi \rangle \langle \xi_\epsilon, \psi \rangle] 
=  \sum_{k \in \N_0^d} \tau( \epsilon k)^2 \langle \varphi, \fn_k \rangle \langle \psi, \fn_k \rangle
\xrightarrow{\epsilon \downarrow 0} 
 \sum_{k \in \N_0^d}  \langle \varphi, \fn_k \rangle \langle \psi, \fn_k \rangle =  \langle \varphi, \psi \rangle.
\end{align*}
That the limit of $\Xi_\epsilon$ is independent of the choice of $\tau$, follows from Theorem~\ref{theorem:convergence_enhanced_pair_theta}~\ref{item:convergence_Theta_min_Xi_eps}. 
\end{proof}

\begin{theorem}[Kolmogorov-Chentsov theorem]
\label{theorem:kolmogorov_chentsov}
Let $\zeta_\epsilon $ be a random variable with values in a Banach space $\fX$ for all $\epsilon>0$. 
Suppose there exist $a,b,C>0$ such that for all $\epsilon, \delta>0$,  
\begin{align*}
\E \left[ \| \zeta_\epsilon - \zeta_\delta\|_{\fX}^a \right] \le C |\epsilon - \delta|^{1+b}. 
\end{align*}
Then there exists a random variable $\zeta$ with values in $\fX$ such that in $L^a(\Omega, \fX)$ and almost surely 
\bfold 
\lim_{\epsilon \downarrow 0, \epsilon \in \mathbb{Q} \cap (0,\infty)} \zeta_\epsilon = \zeta. 
\efold 
\end{theorem}
\begin{proof}
This follows from the proof of \cite[Theorem 2.23]{Ka97}.  
\begin{calc}
First note that $\zeta_\epsilon$ is Cauchy in $L^a(\Omega, \fX)$, so that $\lim_{\epsilon\downarrow0} \zeta_\epsilon =\zeta_0$ exists as a limit in this space. Then $\zeta_t$ for $t \in [0,1]$ is as in the Kolmogorov-Chentsov theorem. Therefore it has a continuous modification $\tilde \zeta_t$. 
So for countably many $\epsilon \in (0,1]$ we have $\zeta_\epsilon = \tilde \zeta_\epsilon$ on a full probability set and thus $\zeta_\epsilon \rightarrow \tilde \zeta_0$ almost surely. 
\end{calc}
\end{proof}

In Lemma \ref{lemma:gaussian_besov_bound}\ref{item:zeta_bound_hypercontractivity} we show how we obtain $L^p$ bounds on the $\cC_\fn$ norm from bounds on squares of the Littlewood-Paley blocks.
Lemma \ref{lemma:gaussian_besov_bound}\ref{item:zeta_convergence_hypercontractivity} follows from \ref{item:zeta_bound_hypercontractivity} and will be used in Section \ref{section:eigenvalues_on_boxes} to prove Theorem \ref{theorem:theta_epsilon_converges_to_xi_r}. 

To prove Lemma \ref{lemma:gaussian_besov_bound} we use the following auxiliary lemma. 
It is generally known that the $p$-th moment of a centered Gaussian random variable $Z$ can be bounded by its second moment, as $\E[ | Z|^p ] = (p-1)!! \E[ |Z|^2]^{\frac{p}{2}}$ (see \cite[p.110]{Pa65}). 
We will use the generalisation of this bound, which is a consequence of the so-called hypercontractivity.

\begin{lemma}
\label{lemma:p_th_moment_of_Gaussian}
\cite[Theorem 1.4.1 and equation (1.71)]{Nu09}
Suppose that $Z_n$ for $n\in\N$ are independent standard Gaussian random variables. 
If $Z$ is a random variable in the first or second Wiener chaos, which means it is of the form 
$\sum_{n\in\N} a_n Z_n$ or $\sum_{n,m\in \N} a_{n,m} (Z_n Z_m - \E[ Z_n Z_m])$ 
with $a_n, a_{n,m} \in \C$, then for $p>1$ 
\begin{align*}
\E[ | Z|^p ] \le p^p \E[ |Z|^2]^{\frac{p}{2}}. 
\end{align*}
\end{lemma}

\begin{lemma}
\label{lemma:gaussian_besov_bound} 
Let $A>0$ and $a\in \R$. 
\begin{enumerate}
\item 
\label{item:zeta_bound_hypercontractivity}
Suppose $\zeta$ is a random variable with values in $\cS'_\fn( [0,L]^d)$ such that $\Delta_i \zeta (x)$ is a random variable of the form as $Z$ is, as in Lemma \ref{lemma:p_th_moment_of_Gaussian} for all $i\in \N_{-1}$ and $x\in [0,L]^d$. Suppose that for all $i \in \N_{-1}$, $x\in [0,L]^d$  
\begin{align}
\label{eqn:bound_Delta_i_zeta}
\E[ | \Delta_i \zeta (x) |^2]
\le A 2^{a i }.
\end{align}
Then for all  $\kappa >0$ 
there exists a $C>0$ independent of $\zeta$ such that for all $p\ge 1$
\begin{align}
\label{eqn:zeta_bound_hypercontractivity}
\E[ \| \zeta  \|_{\cC_\fn^{-\frac{a}{2}-\kappa -\frac{2}{p}}}^p ] 
\le C  p^p
L^{d} A^\frac{p}{2}.
\end{align}
\item 
\label{item:zeta_convergence_hypercontractivity}
Suppose that $(\zeta_\epsilon)_{\epsilon>0}$ is a family of such random variables for which \eqref{eqn:bound_Delta_i_zeta} holds for all $i\in \N_{-1}$ and $x\in [0,L]^d$, and that for all $k\in \N_0^d$ 
\begin{align}
\label{eqn:zeta_epsilon_paired_with_fn_k_conv_0}
\E[ |\langle \zeta_\epsilon, \fn_{k,L} \rangle|^2]  \rightarrow 0.
\end{align}
Then for all $\kappa>0$ and $p>1$
\begin{align*}
\E[  \| \zeta_\epsilon  \|_{\cC_\fn^{-\frac{a}{2}-\kappa -\frac{2}{p}}}^p ]  \rightarrow 0. 
\end{align*}
Consequently, we have $\zeta_\epsilon \xrightarrow{\P} 0$ (convergence in probability) in $\cC_\fn^{-\frac{a}{2}-\kappa -\frac{2}{p}}([0,L]^d)$.
\end{enumerate}
\end{lemma}
\begin{proof}
\ref{item:zeta_bound_hypercontractivity}
For $\kappa>0$, by Lemma \ref{lemma:p_th_moment_of_Gaussian}, with $C_\kappa = \sum_{i=-1}^\infty  2^{ - \kappa i}$, 
\begin{align*}
\E[ \| \zeta \|_{B_{p,p}^{\Even, -\frac{a}{2}-\kappa}}^p ]
& =   \sum_{i=-1}^\infty 2^{(-\frac{a}{2}-\kappa) p i} \E \left[ \| \Delta_i \zeta \|_{L^p}^p \right]  
 \le  p^p L^{d} \Big(\sum_{i=-1}^\infty  2^{ -p \kappa i} \Big) A^{ \frac{p}{2}} 
\le C_\kappa p^p L^{d} A^{ \frac{p}{2}}. 
\end{align*}
Using the embedding property of Besov spaces \cite[Proposition 2.71]{BaChDa11}, 
which implies the existence of a $C>0$ such that 
$
 \| \cdot \|_{\cC_\fn^{-\frac{a}{2}-\kappa -\frac{2}{p}}}
\le C 
 \| \cdot \|_{B_{p,p}^{\Even,-\frac{a}{2}-\kappa}}
$, one obtains \eqref{eqn:zeta_bound_hypercontractivity}. \\
\ref{item:zeta_convergence_hypercontractivity}
By Lemma \ref{lemma:p_th_moment_of_Gaussian} (and Fubini) 
\begin{align*}
\E[ \| \Delta_i \zeta_\epsilon \|_{L^p}^p ] 
\cand \begin{calc}= \int \E[ | \Delta_i \zeta_\epsilon(x) |^p ] \dd x \end{calc} \cnewline
& \le p^p \int \E[ | \Delta_i \zeta_\epsilon(x) |^2 ]^{\frac{p}{2}} \dd x 
 \lesssim 
 p^p L^d \Big(\sum_{k\in\N_0^d} \rho_i( \tfrac{k}{L} )^2 \E[ |\langle \zeta_\epsilon, \fn_k \rangle|^2 ] \Big)^{\frac{p}{2}}.
\end{align*}
and so 
\begin{align*}
\E[ \| \zeta_\epsilon  \|_{B_{p,p}^{\fn,-\frac{a}{2}-\kappa }}^p ]
\cand 
\begin{calc}
= \sum_{i=-1}^\infty 2^{(-\frac{a}{2}-\kappa) p i} \E \left[ \| \Delta_i \zeta_\epsilon \|_{L^p}^p \right]  
\end{calc}
\cnewline
& \le p^p L^d \left( 
\sum_{i = -1}^I 2^{(-\frac{a}{2}-\kappa) p i} 
\left(\sum_{k\in\N_0^d} \rho_i( \tfrac{k}{L} )^2 \E[ |\langle \zeta_\epsilon, \fn_k \rangle|^2 ] \right)^{\frac{p}{2}}
 + A^{\frac{p}{2}} \sum_{i \ge I+1} 2^{-\kappa i }
\right) .
\end{align*}
The latter becomes arbitrarily small by choosing $I$ large and subsequently $\epsilon$ small. 
\end{proof}

\begin{obs}
\label{obs:second_moment_bounds}

The following two statements are proved in Section \ref{section:proof_white_noise_conv}. 
\begin{enumerate}
\item 
(Lemma \ref{lemma:bound_on_expect_mod_Delta_i_difference_xi_in_x_squared})
For all $\gamma\in (0,1)$ there exists a $C>0$ such that for all $i \in \N_{-1}$, $\epsilon,\delta>0$, $x\in [0,L]^d$
\begin{align*}
\E[ | \Delta_i (\xi_\epsilon - \xi_\delta) (x)|^2]
& \le C
2^{(d+2\gamma )i} |\epsilon - \delta|^{\gamma}. 
\end{align*}
\item 
(Lemma \ref{lemma:bound_on_expect_mod_Delta_i_difference_Xi_in_x_squared})
Let $d=2$. 
For all $\gamma \in (0,1)$ there exists a $C>0$ such that for all  $i\in \N_{-1}$, $\epsilon,\delta>0$, $x\in Q_L$
\begin{align*}
\E [ |\Delta_i ( \Xi_\epsilon - \Xi_\delta)(x)|^2] 
\le C  2^{2\gamma i } |\epsilon - \delta|^{\gamma}  . 
\end{align*}
\end{enumerate}
\end{obs}

\begin{definition}
Define $c_{\epsilon,L} \in \R$ by 
\begin{align}
\label{eqn:c_eps_L}
c_{\epsilon,L} = \frac{1}{4 L^2}
\sum_{k\in \Z^2} 
\frac{ \tau(\tfrac{\epsilon}{L}k)^2}{1+\frac{\pi^2}{L^2} |k|^2}. 
\end{align}
\end{definition}

In the periodic setting one has that with  $\xi_\epsilon$ defined as in \cite{AlCh15}, 
$ 
\E[\xi_\epsilon \reso \sigma(\rD) \xi_\epsilon(x)]= c_{\epsilon,L}.
$ 
Observe that it is independent of $x$. 
In our setting, the Dirichlet setting, we have (remember \eqref{eqn:first_definition_Z_k} and use that 
$\sum_{i,j \in \N_{-1},|i-j|\le1} \rho_i( \frac{k}{L} ) \rho_j( \frac{k}{L} \ew) = 1$)
\begin{align}
\label{eqn:expression_expectation_xi_reso_xi}
\E [\xi_\epsilon \reso \sigma(\rD) \xi_\epsilon (x) ]
\begin{calc}
=  
\sum_{\substack{i,j \in \N_{-1}\\|i-j|\le1} }
\sum_{k\in \N_0^2} 
\rho_i( \tfrac{k}{L} )\rho_j( \tfrac{k}{L} ) \frac{  \tau(\tfrac{\epsilon}{L} k)^2}{1+\frac{\pi^2}{L^2} |k|^2} \fn_k(x)^2
\end{calc}
=  
\sum_{k\in \N_0^2} 
\frac{ \tau(\tfrac{\epsilon}{L} k)^2}{1+\frac{\pi^2}{L^2} |k|^2} \fn_k(x)^2 .
\end{align}
By \eqref{eqn:basis_prod_even_even}, as $\fn_0(x) = \frac{2}{L} \nu_0= \frac{1}{L} $ and $\nu_{2k}=\nu_k$, 
\begin{align}
\label{eqn:fn_k_squared}
 \fn_k(x)^2 
\cand \begin{calc} \notag = \frac{1}{2L}
\sum_{\fp  \in \{-1,1\}^2} \frac{\nu_k^2}{\nu_{k+\fp \circ k }} \fn_{k+ \fp \circ k} (x) \end{calc} \cnewline
\cand \begin{calc} \notag = 
\frac{1}{2L}  \frac{\nu_k^2}{\nu_{2k}} \fn_{k+ \fp \circ k} (x)
+
\frac{1}{2L}  \frac{\nu_k^2}{\nu_{(k_1,0)}} \fn_{(2k_1,0)} (x)
+
\frac{1}{2L}  \frac{\nu_k^2}{\nu_{(0,k_2)}} \fn_{(0,2k_2)} (x)
+
\frac{1}{2L}  \frac{\nu_k^2}{\nu_{0 }} \fn_{0} (x) \end{calc} \cnewline
& = 
\frac{1}{2L} \nu_k \fn_{2k }(x) 
+ \frac{1}{2L} \frac{\nu_k^2}{\nu_{(k_1,0)}}  \fn_{(2k_1,0) }(x) 
+ \frac{1}{2L} \frac{\nu_k^2}{\nu_{(0,k_2)}}  \fn_{(0,2k_2) }(x) + \frac{\nu_k^2}{L^2} .
\end{align}
Note that 
\begin{calc}
as $\fn_k(0) = 2\frac{\nu_k}{L} $
\end{calc}
\begin{align}
\label{eqn:c_epsilon_L_in_terms_of_f_k_zero}
c_{\epsilon,L} = 
\sum_{k\in \N_0^2} 
\frac{\tau(\tfrac{\epsilon}{L} k)^2}{1+\frac{\pi^2}{L^2} |k|^2} \frac{ \nu_k^2}{L^2} = \frac{1}{4} \E[\xi_\epsilon \reso \sigma(\rD) \xi_\epsilon(0)] . 
\end{align}
\begin{calc}
\begin{align*}
4\nu_k^2 = 2^{2-\# \{i : k_i =0\}} = 
\begin{cases}
1 & k = 0, \\
2 & k_1 = 0, k_2 \ne 0 \mbox{ or } k_1 \ne 0, k_2 =0, \\
4 & k \in \N^2. 
\end{cases}
\end{align*}
\end{calc}

Lemma \ref{lemma:expectation_reso_xi_eps_C_0_plus_reso_constant} deals with this $x$ dependence of $\E [\xi_\epsilon \reso \sigma(\rD) \xi_\epsilon (x) ]$.

The following observations will be used multiple times. 

\begin{obs}
As $0 \le \rho_i \le 1$ and there is a $b \ge 1$ such that $\rho_i$ is supported in a ball of radius $2^i b$ for all $i \in \N_{-1}$, one has for all $i \in \N_{-1}$, $x\in \R^d$ and $\gamma>0$ 
\begin{align}
\label{eqn:polynomial_bound_on_rho_i}
\rho_i(x)  \le \left(2b \frac{2^i}{1+|x|} \right)^{\gamma}. 
\end{align}
\end{obs}

\begin{theorem}
\label{theorem:fourier_mollifiers_converge_in_B_22_space}
Let $\tau : \R^2 \rightarrow [0,1]$ be a compactly supported even function that equals $1$ on a neighbourhood of $0$.
There exists a $C>0$ such that for all $\gamma \in \R$, $L>0$ and $h\in H_\fn^\gamma(Q_L)$ we have $\| h - \tau(\epsilon \rD) h\|_{H_\fn^\gamma} \rightarrow 0$ and for $\beta < \gamma$ 
\begin{align*}
\| h - \tau(\epsilon \rD) h\|_{H_\fn^\beta} \le C \epsilon^{\gamma-\beta} \|h\|_{H_\fn^\gamma}. 
\end{align*}
\end{theorem}
\begin{proof}
By assumption on $\tau$ there exists an $a>0$ such that $\tau =1$ on $B(0,a)$. Then 
\begin{align*}
\begin{cases}
1- \tau(\tfrac{\epsilon}{L} k) =0 & |k| < \tfrac{La}{\epsilon}, \\
(1+|\frac{k}{L}|^2)^{\beta-\gamma}
\lesssim \epsilon^{2(\gamma-\beta)} & |k| \ge \frac{La}{\epsilon}. 
\end{cases}
\end{align*}
\begin{calc}
If $|k| \ge \frac{La}{\epsilon}$, then $|\frac{k}{L}| \ge \frac{a}{\epsilon}$ and thus 
\begin{align*}
\left(1+|\frac{k}{L}|^2 \right)^{-1}
\le \left(1+ (\frac{a}{\epsilon})^2 \right)^{-1} \le \frac{\epsilon^2}{a^2}. 
\end{align*}
\end{calc}
By the following bounds the theorem is proved; by Theorem \ref{theorem:equivalent_norms_H_spaces}
\begin{align*}
\| h - \tau(\epsilon \rD) h\|_{H_\fn^\beta} 
\lesssim 
\sqrt{ \sum_{k\in \N_0^d} (1+|\tfrac{k}{L}|^2)^{\beta} (1- \tau(\tfrac{\epsilon}{L} k) )^2 \langle h, \fn_k \rangle^2 }
\lesssim \epsilon^{\gamma-\beta} \|h\|_{H_\fn^\gamma}.
\end{align*}
\end{proof}

\begin{lemma}
\label{lemma:expectation_reso_xi_eps_C_0_plus_reso_constant}
Let $\tau : \R^2 \rightarrow [0,1]$ be a compactly supported even function that equals $1$ on a neighbourhood of $0$. Then 
$x\mapsto \E [\xi_\epsilon \reso \sigma(\rD) \xi_\epsilon (x) ] -  c_{\epsilon,L} $ converges in $\cC_\fn^{-\gamma}$  to a limit that is independent of $\tau$ as $\epsilon \downarrow 0$ for all $\gamma>0$. 
\end{lemma}
\begin{proof}
Let $\gamma>0$. 
As there are only finitely many $k\in \N_0^2$ for which $\tau(\tfrac{\epsilon}{L} k) \ne 0$, $x\mapsto \E [\xi_\epsilon \reso \sigma(\rD) \xi_\epsilon (x) ] -  c_{\epsilon,L} $ is smooth. 
We can rewrite \eqref{eqn:fn_k_squared} and find uniformly bounded $a_k,b_k$ such that $\fn_k(x)^2 - \frac{\nu_k^2}{L^2}= \frac{1}{2L} [\fn_{k} + a_k \fn_{(k_1,0)} + b_k \fn_{(0,k_2)}](2x)$. By \eqref{eqn:c_epsilon_L_in_terms_of_f_k_zero} this means that $\E [\xi_\epsilon \reso \sigma(\rD) \xi_\epsilon (x) ] $ 
(see \eqref{eqn:expression_expectation_xi_reso_xi})
 can be decomposed into three sums. 

For the first sum (by taking the part with ``$\fn_k$''), as $\delta_0 \in H_\fn^{-1}$ and $\langle \delta_0, \fn_k \rangle = \frac2L$ for all $k \in \N_0^2$
\begin{align*}
\frac{1}{2L}
\sum_{k\in \N_0^2} 
\frac{ \tau(\tfrac{\epsilon}{L} k)^2}{1+ \frac{\pi^2}{L^2} |k|^2} 
 \fn_{k } (2x) 
= \tfrac{1}{4} [ \tau (\epsilon  \rD)^2 \sigma(\rD) \delta_0] (2x).
\end{align*}
By Theorem \ref{theorem:schauder} 
$\sigma(\rD) \delta_0 \in H_\fn^1$, so that by Theorem \ref{theorem:fourier_mollifiers_converge_in_B_22_space} 
$\tau (\epsilon  \rD)^2 \sigma(\rD) \delta_0 \rightarrow  \sigma(\rD) \delta_0$ in $H_\fn^{1-\gamma}$ and thus in $\cC_\fn^{-\gamma}$ (by \cite[Theorem 2.71]{BaChDa11}). 
This convergence is `stable' under `multiplying the argument by $2$' (see also \ref{obs:scaling_fourier_multipliers}).  

Now let us show the convergence of the other sums. 
We only consider the sum with ``$a_k \fn_{(k_1,0)}$'' in it, as the sum with ``$b_k \fn_{(0,k_2)}$'' follows similarly. 
Let us write $h_\epsilon$ for 
\begin{align*}
h_\epsilon (x)
& 
\begin{calc2}
= \sum_{k\in \N_0^2} 
\frac{ \tau(\tfrac{\epsilon}{L} k)^2}{1+ \frac{\pi^2}{L^2} |k|^2} a_k
 \fn_{(k_1,0) } (x) 
 \end{calc2}
 = \sum_{l ,m\in \N_0} 
\frac{ \tau(\tfrac{\epsilon}{L}(l,m))^2}{1+ \frac{\pi^2}{L^2} (l^2+m^2)} a_{(l,m)}
 \fn_{(l, 0) } (x). 
\end{align*}
\begin{calc2}
Then 
\begin{align*}
\Delta_i h_\epsilon = \sum_{l\in \N_0} \rho_i (  \tfrac{l}{L}  ,0) \sum_{m\in \N_0} \frac{ \tau(\tfrac{\epsilon}{L}(l,m))^2}{1+ \frac{\pi^2}{L^2} (l^2+m^2)} 
a_{(l,m)}
 \fn_{(l, 0) } . 
\end{align*}
\end{calc2}
With \eqref{eqn:polynomial_bound_on_rho_i} 
$\|\Delta_i \fn_{(l,0)}\|_{L^\infty} \lesssim |\rho_i( \tfrac{l}{L} ,0) | \lesssim 2^{\gamma i } (1+  \tfrac{l^2}{L^2} )^{-\gamma}$.  
Hence
\begin{align*}
\sup_{i \in \N_{-1}} 2^{-\gamma i } \| \Delta_i (h_\epsilon - h_0) \|_{L^\infty} 
\lesssim \sum_{l,m\in \N_0 } (1+  \tfrac{l^2}{L^2}  )^{-\gamma}  \frac{ \left| \tau(\tfrac{\epsilon}{L}(l,m))^2-1 \right|}{1+ \frac{\pi^2}{L^2} (l^2+m^2)} . 
\end{align*}
By Lebesgue's dominated convergence theorem and the next bound it follows that $h_0 \in \cC_\fn^{-\gamma}$ and $h_\epsilon \rightarrow h_0$ in $\cC_\fn^{-\gamma}$. By using that $1+l^2+m^2 \ge (1+l)^{1-\frac{\gamma}{2}} (1+m)^{1+\frac{\gamma}{2}}$, 
\begin{align*}
\sum_{l,m\in \N_0 }  \frac{ (1+  \tfrac{l^2}{L^2})^{-\gamma} }{1+ \frac{\pi^2}{L^2} (l^2+m^2)} 
& \lesssim 
\sum_{l,m\in \N_0 }  \frac{ 1 }{(1+l)^{1+\frac{\gamma}{2}} (1+m)^{1+\frac{\gamma}{2}}} 
 < \infty. 
\end{align*}
By these convergences and by plugging in the factor $2$ also here the convergence is proved. 
\end{proof}


Before we give the proof of Theorem \ref{theorem:convergence_in_enhanced_space_to_white_noise}, we study the behaviour of $c_{\epsilon,L}$.

\begin{lemma}
\label{lemma:renormalisation_constant}
Let $\tau : \R^2 \rightarrow [0,1]$ be almost everywhere continuous, be equal to $1$ on $B(0,a)$ and zero outside $B(0,b)$ for some $a,b$ with $0<a<b$. 
There exist a $c_\tau \in \R$ that only depends on $\tau$, and $(C_L)_{L\ge 1}$ in $\R$ that do not depend on $\tau$ with $C_L \xrightarrow{L\rightarrow \infty} 0$ such that 
\bfold 
c_{\epsilon,L} 
- \tfrac{1}{2\pi} \log \tfrac{1}{\epsilon} -c_\tau  \xrightarrow{\epsilon \downarrow 0} C_L 
\efold 
for all $L \ge 1$. 
\end{lemma}
\begin{proof}
We define $\lfloor y \rfloor = (\lfloor y_1 \rfloor , \lfloor y_2 \rfloor )$
and $h_L(y) = ( L^2 + \pi^2 |y|^2)^{-1}$
 for $y\in \R^2$. Then
\bfold 
4 c_{\epsilon,L} 
 \begin{calc}
 =   \sum_{k\in  \Z^2} \frac{\tau(\frac{\epsilon}{L} k)^2}{ L^2 + \pi^2 |k|^2}
 \end{calc} 
 =   \int_{\R^2} \tau( \tfrac{\epsilon}{L}  \lfloor  y \rfloor )^2 h_L ( \lfloor  y \rfloor) \dd y.
\efold 
We first show that 
\bfold 
4 c_{\epsilon,L} -   \int_{\R^2} \tau( \tfrac{\epsilon}{L}   y  )^2 h_L (   y ) \dd y \rightarrow 0. 
\efold 
Write $A(s,t)$ for the annulus $\{ y \in \R^2 : s \le |y| \le t\}$. 
To shorten notation, we write $\delta = \frac{\epsilon}{L}$. 
As $| \lfloor y \rfloor - y | \le \sqrt{2}$
\begin{align*}
4 c_{\epsilon,L} - \int_{\R^2} \tau( \tfrac{\epsilon}{L}   y )^2 h_L (   y ) \dd y
& =  \int_{B(0,\frac{a}{\delta} - \sqrt{2})} h_L(\lfloor y \rfloor) - h_L(y) \dd y \\
& \qquad + \int_{A( \frac{a}{\delta} - \sqrt{2}, \frac{b}{\delta} + \sqrt{2})} 
\tau( \delta  \lfloor  y \rfloor )^2 h_L ( \lfloor  y \rfloor) 
- \tau( \delta    y  )^2 h_L (   y ) \dd y. 
\end{align*}
As $h_L(\lfloor y \rfloor) - h_L(y) = h_L(\lfloor y \rfloor) h_L(y)( |y|^2 - |\lfloor y \rfloor|^2)$, $h_L(\lfloor y \rfloor) \lesssim h_L(y)$ and $( |y|^2 - |\lfloor y \rfloor|^2) \lesssim 1+ |y|$, we have 
$h_L(\lfloor y \rfloor) - h_L(y) \lesssim (1+|y|) h_L(y)^2$. As the latter function is integrable over $\R^2$, it follows by Lebesgue's dominated convergence theorem that $\int_{B(0,\frac{a}{\delta} - \sqrt{2})} h_L(\lfloor y \rfloor) - h_L(y) \dd y$ converges in $\R$ to a $C_L$ for which $C_L \xrightarrow{L\rightarrow \infty}0$. 
On the other hand, the integral over the annulus can be written as 
\begin{align}
\label{eqn:difference_integral_with_integer_parts_and_without}
 \int_{A( a - \sqrt{2}\delta, b + \sqrt{2} \delta)} 
\frac{\tau( \delta  \lfloor \frac{x}{\delta} \rfloor )^2}{ \delta^2 L^2 + \pi^2 \delta^2 | \lfloor \frac{x}{\delta}\rfloor |^2} 
- \frac{\tau(  x )^2}{ \delta^2 L^2 + \pi^2 | x |^2}  \dd x.
\end{align}
Again by a domination argument (note that $\frac{1}{|x|^2}$ is integrable over annuli), using that 
$ | \frac{x}{\delta} |^2 
\begin{calc}
\le  ( | \lfloor \frac{x}{\delta} \rfloor |+ \sqrt{2}) ^2 
\end{calc}
\le
 4 + 2 | \lfloor \frac{x}{\delta} \rfloor |^2 \le 4 (L^2 + | \lfloor \frac{x}{\delta} \rfloor |^2 )$,
\begin{calc}
so that 
\begin{align*}
\left| 
\frac{\tau( \delta  \lfloor \frac{x}{\delta} \rfloor )^2}{ \delta^2 L^2 + \pi^2 \delta^2 | \lfloor \frac{x}{\delta}\rfloor |^2} 
\right| 
\le \frac{4}{\pi^2 |x|^2}. 
\end{align*}
\end{calc} 
 we conclude that \eqref{eqn:difference_integral_with_integer_parts_and_without} converges to $0$. 
Observe that 
 \begin{align*}
  \int_{A( a - \sqrt{2}\delta, b + \sqrt{2} \delta)} 
 \frac{\tau(  x )^2}{ \delta^2 L^2 + \pi^2 | x |^2}  \dd x \xrightarrow{\delta \downarrow 0}
   \int_{A( a , b )} 
 \frac{\tau(  x )^2}{\pi^2 | x |^2}  \dd x. 
 \end{align*}
By some substitutions  (remember $\delta = \frac{\epsilon}{L}$), for $\epsilon <a$
\begin{align*}
\frac{1}{2\pi}
\int_{B(0,\frac{a}{\delta} - \sqrt{2})}  h_L(y) \dd y
\cand \begin{calc}
= \int_0^{\frac{a}{\delta} - \sqrt{2}} \frac{x}{L^2 + \pi^2 x^2} \dd x
= \int_0^{\frac{a}{\epsilon} - \frac{\sqrt{2}}{L} } \frac{y}{1 + \pi^2 y^2} \dd y
\end{calc} \cnewline
\cand \begin{calc}
= \int_0^{1 } \frac{y}{1 + \pi^2 y^2} \dd y
+ 
\int_0^{\frac{a}{\epsilon} }\frac{y}{1 + \pi^2 y^2} \dd y
+ 
\int_{\frac{a}{\epsilon} - \frac{\sqrt{2}}{L} }^{\frac{a}{\epsilon}} \frac{y}{1 + \pi^2 y^2} \dd y
\end{calc} \cnewline
& = \int_0^{1} \frac{s}{1 + \pi^2 s^2} \dd s
+ \int_1^{\frac{a}{\epsilon}} \frac{s}{1 + \pi^2 s^2} \dd s
- \int_{a - \frac{\sqrt{2}\epsilon}{L} }^{a} \frac{s}{\epsilon^2 + \pi^2 s^2} \dd s.
\end{align*}
The last integral converges as $\epsilon \downarrow 0$ to zero. 
For the second integral we consider 
\begin{align*}
\int_1^{\frac{a}{\epsilon}} \frac{s}{1 + \pi^2 s^2} - \frac{1}{\pi^2 s} \dd s
= 
\int_1^{\frac{a}{\epsilon}}
 \frac{ -  1 }{\pi^2 s(  1+ \pi^2 s^2) } \dd s, \qquad 
 \int_1^{\frac{a}{\epsilon}}  \frac{1}{\pi^2 s} \dd s = \frac{1}{\pi^2} \log(\frac{a}{\epsilon}). 
\end{align*}
Observe that if $a\le 1$ then  $\int_{A(a,1)} \frac{1}{\pi^2 |x|^2} \dd x =- \frac{2}{\pi} \log a$ and if $a\ge 1$ then $ \int_{A( 1, a)} \frac{1}{\pi^2 |x|^2} \dd x =\frac{2}{\pi} \log a $. Therefore, with 
\begin{align*}
c_\tau =    \int_{A(a \wedge 1,b)} 
 \frac{\tau(  x )^2}{\pi^2 | x |^2}  \dd x 
 - \int_{A(a \wedge 1,1)} 
 \frac{1}{ \pi^2 | x |^2}  \dd x 
 +  \int_0^{1} \frac{2\pi s}{1 + \pi^2 s^2} \dd s
 -  \int_1^{\infty}
 \frac{  2 }{\pi s(  1+ \pi^2 s^2) } \dd s.
\end{align*}
we obtain that $c_{\epsilon,L} - \tfrac{1}{2\pi} \log \tfrac{1}{\epsilon} - C_L - c_\tau \xrightarrow{\epsilon \downarrow 0 } 0$.
Observe that $c_\tau$ does not depend on the choice of $a,b$ (such that $\tau=1$ on $B(0,a)$ and $\tau =0$ outside $B(0,b)$).
\end{proof}

\begin{proof}[Proof of Theorem \ref{theorem:convergence_in_enhanced_space_to_white_noise}]
This is a consequence of 
Theorem \ref{theorem:convergence_to_dirichlet_white_noise} and
Lemmas \ref{lemma:expectation_reso_xi_eps_C_0_plus_reso_constant} and \ref{lemma:renormalisation_constant}. 
\begin{calc}
Indeed, we can decompose $\xi_\epsilon \reso \sigma(\rD) \xi_\epsilon -c_\epsilon$ into the sum of $\Xi_\epsilon$ (which converges by Theorem~\ref{theorem:convergence_to_dirichlet_white_noise}), $\E[\xi_\epsilon \reso \sigma(\rD) \xi_\epsilon] - c_{\epsilon,L}$ (which converges by Lemma~\ref{lemma:expectation_reso_xi_eps_C_0_plus_reso_constant}) and $ c_{\epsilon,L} -c_\epsilon$ (which converges by Lemma~\ref{lemma:renormalisation_constant}). So 
\begin{align*}
\lim_{\epsilon \downarrow 0} 
\xi_\epsilon \reso \sigma(\rD) \xi_\epsilon -c_\epsilon
& = \lim_{\epsilon \downarrow 0}  \Xi_\epsilon  +
\lim_{\epsilon \downarrow 0} \Big( \E[\xi_\epsilon \reso \sigma(\rD) \xi_\epsilon] - c_{\epsilon,L} \Big) 
+ C_L. 
\end{align*}
\end{calc}
\end{proof}

\section{Scaling and translation}
\label{section:scaling_and_translation}

In this section we prove the scaling properties of the eigenvalues, by scaling the size of the box and the noise. 
\textbf{In this section we fix $L>0$ and $n\in\N$.}

\begin{lemma}
\label{lemma:scaling_smooth_potentials_dir}
Suppose that $V\in L^\infty([0,L]^d)$. 
 For all $\beta >0$ 
\begin{align*}
\lambda_n([0,L]^d, V)
= \tfrac{1}{\beta^2} \lambda_n([0,\tfrac{L}{\beta}]^d, \beta^2 V(\beta\cdot)).
\end{align*}
\end{lemma}
\begin{proof}
Fix $n\in \N$ and write $\lambda = \lambda_n([0,L]^d,V)$. 
Suppose that $g\in H_0^2$ (see \ref{obs:smooth_potentials}) is an eigenfunction for $\lambda$ of $\Delta + V$. 
With $g_\beta(x) := g(\beta x)$ we have for almost all $x$
\begin{align*}
\Delta g_\beta (x) + \beta^2 V(\beta x) 
= \beta^2 (\Delta g)(\beta x) + \beta^2 V(\beta x) = \beta^2 \lambda g_\beta(x). 
\end{align*}
So that $\beta^2 \lambda$ is an eigenvalue of $\Delta + \beta^2 V(\beta\cdot)$ on $[0,\frac{L}{\beta}]^d$. 
As the multiplicities of the eigenvalues on $[0,L]^d$ and $[0,\tfrac{L}{\beta}]^d$ are the same, $\beta^2 \lambda = \lambda_n( [0,\tfrac{L}{\beta}]^d, \beta^2 V(\beta \cdot))$. 
\end{proof}

\begin{obs}
For $y\in \R^2$, $L>0$ and $\beta \in \R$ we write 
\begin{align*}
\blambda_n(y+Q_L, \beta) = \lambda_n(y+Q_L, (\beta \xi_L^y, \beta^2 \Xi_L^y) ),
\quad 
\blambda_n(y+Q_L) = \blambda_n(y+Q_L,1), 
\end{align*}
where $\bxi_L^y= (\xi_L^y, \Xi_L^y)$ is as in \ref{obs:shift_enhanced_white_noise}. 
\end{obs}

\begin{lemma}
\label{lemma:scaling_eigenvalues_with_white_noise}
For $\alpha,\beta>0$
\begin{align*}
\blambda_n 
\left( Q_L, \beta \right)
\overset{d}{=}
 \tfrac{1}{\alpha^2} \blambda_n 
(
Q_{\frac{L}{\alpha}}, \alpha \beta
) 
+ \tfrac{1}{2\pi} \log \alpha.
\end{align*}
\end{lemma}
\begin{proof}
For simplicity we take $\beta =1$. 
 $ \alpha l_\alpha \xi_L$ is a white noise  on $Q_{\frac{L}{\alpha}}$, 
 so that $\langle \alpha l_\alpha \xi_L, \fn_k \rangle \overset{d}{=} \langle  \xi_{\frac{L}{\alpha}}, \fn_k \rangle$ for all $k\in \N_0^2$ and thus $\frac1\alpha  \xi_{\frac{L}{\alpha}} \overset{d}{=}  l_\alpha \xi_L$. 
By \ref{obs:scaling_fourier_multipliers} 
$ 
l_\alpha \xi_{L,\epsilon} = \tau(\tfrac{\epsilon}{\alpha}\rD) [l_\alpha \xi_L]
 \overset{d}{=} 
 \tfrac{1}{\alpha}   \xi_{\frac{L}{\alpha},\frac{\epsilon}{\alpha} }. 
$ 
So that by Lemma \ref{lemma:scaling_smooth_potentials_dir} 
\begin{align*}
& \lambda_n  \left(Q_L, (\xi_{L,\epsilon}, \xi_{L,\epsilon} \reso \sigma(\rD) \xi_{L,\epsilon} - \tfrac{1}{2\pi} \log (\tfrac{1}{\epsilon}) )  \right)
 = 
\lambda_n \left(Q_{L}, \xi_{L,\epsilon} \right)  - \tfrac{1}{2\pi} \log (\tfrac{1}{\epsilon}) \\
& \qquad \overset{d}{=} \tfrac{1}{\alpha^2}
\lambda_n \left(Q_{\frac{L}{\alpha}},  \alpha  \xi_{\frac{L}{\alpha},\frac{\epsilon}{\alpha}}\right)  - \tfrac{1}{2\pi} \log (\tfrac{1}{\epsilon}) \\
& \qquad  \overset{d}{=} \tfrac{1}{\alpha^2}
\lambda_n \left(Q_{\frac{L}{\alpha}}, ( \alpha  \xi_{\frac{L}{\alpha},\frac{\epsilon}{\alpha}}, 
\alpha^2 \left[  \xi_{\frac{L}{\alpha},\frac{\epsilon}{\alpha}} \reso \sigma(\rD)  \xi_{\frac{L}{\alpha},\frac{\epsilon}{\alpha}} -  \tfrac{1}{2\pi} \log (\tfrac{\alpha}{\epsilon}) \right] )  \right) + \tfrac{1}{2\pi} \log \alpha. 
\end{align*}
Now we can subtract $c_\tau$ from both sides and take the limit $\epsilon\downarrow 0$.
\end{proof}
%

\begin{lemma}
\label{lemma:translation_of_eigenvalues}
For $y\in \R^2$ and $\beta>0$ 
\begin{align*}
\blambda_n(Q_L, \beta  )
 \overset{d}{=} \blambda_n(y+Q_L, \beta ). 
\end{align*}
Moreover, if $y +Q_L^\circ \cap Q_L^\circ = \emptyset$, then
$\blambda_n(Q_L, \beta  )$ and $\blambda_n(y+Q_L, \beta )$
 are independent. 
\end{lemma}
\begin{proof}
As (see also Definition \ref{def:spaces_on_general_boxes}, in particular \eqref{eqn:definition_sigma_D_on_rectangle}) 
$ 
\sH_{\bxi_L^y} f = \cT_y ( \sH_{\cT_{-y} \bxi_L^y} (\cT_{-y} f) ),
$ 
it is sufficient to show $\bxi_L \overset{d}{=} \cT_{-y} \bxi_L^y$. 
As $\cT_{-y} \cW \overset{d}{=} \cW$, we have $\cT_{-y} \xi_{L,\epsilon}^y \overset{d}{=} \xi_{L,\epsilon}$ and hence obtain $\bxi_L \overset{d}{=} \cT_{-y} \bxi_L^y$ by \eqref{eqn:limit_of_mollifiers} and \eqref{eqn:limit_of_mollifiers_shift}. 

For the ``moreover''; note that $( \langle \cT_y^{-1} \cW, \fn_{k,L} \rangle )_{k\in\N_0^2}$ and $( \langle  \cW, \fn_{k,L} \rangle )_{k\in\N_0^2}$ are independent when $y+Q_L^\circ \cap Q_L^\circ = \emptyset$ (as $\E[ \langle \cT_y^{-1}  \cW,  \fn_{k,L} \rangle \langle  \cW, \fn_{m,L} \rangle ] = \langle \cT_y \fn_{k,L} , \fn_{k,L} \rangle =0$). 
\end{proof}

\section{Comparing eigenvalues on boxes of different size}
\label{section:eigenvalues_on_boxes}

\subsection{Bounded potentials}

In this section we prove the bounds comparing eigenvalues on large boxes with eigenvalues on smaller boxes for bounded potentials, see Lemma \ref{lemma:monotonicity_of_eigenvalues_in_radius}, Theorem \ref{theorem:eigenvalues_boxes} and Theorem \ref{theorem:n_eigenvalue_bound_from_below_of_principals}. 
In Section \ref{subsection:white_noise_bounds}, Theorem \ref{theorem:lower_and_upper_bounds_white_noise}, we extend this for white noise potentials. 
We fix $d\in \N$ and use the notation $|k|_\infty = \max_{i \in \{1,\dots, d\}} |k_i|$.

\begin{lemma}
\label{lemma:monotonicity_of_eigenvalues_in_radius}
Let $L>r>0$ and $\zeta \in L^\infty([0,L]^d)$. 
For all $y\in \R^2$ such that $y+[0,r]^d  \subset [0,L]^d$, we have 
\begin{align*}
 \lambda_n(y+[0,r]^d, \zeta) \le \lambda_n([0,L]^d, \zeta) . 
\end{align*}
\end{lemma}
\begin{proof}
This follows from \eqref{eqn:min_max_for_smooth} 
as one can identify a finite dimensional
$F\sqsubset H_0^2(y+[0,r]^d)$ with a linear subspace of $H_0^2([0,L]^d)$ with the same dimension. 
\end{proof}


We will now prove an upper bound for $\lambda_n(Q_L,\zeta)$ in terms of a maximum over smaller boxes. 
For this we cover $Q_L$ by smaller boxes that overlap and correct the potential with a function that takes into account the overlaps. We use the following lemma. 

\begin{lemma}
\label{lemma:partition_with_a_and_r}
Let $r>a>0$. 
There exists a smooth function $\eta : \R^d \rightarrow [0,1]$ with $\eta = 1$ on $[0,r-a]^d$ and $\supp  \eta \subset [-a,r]^d$ such that $\| \nabla \eta\|_\infty \le \frac{K}{a}$  for some $K>0$ that does not depend on $r$ and $a$, and 
\begin{align}
\label{eqn:sum_squares_eta}
\sum_{k\in\Z^d} \eta(x-rk)^2 =1 \qquad (x\in \R^d). 
\end{align}
\end{lemma}
\begin{proof}
We adapt the proof of  \cite[Proposition 1]{GaKo00} 
and 
\cite[Lemma 4.6]{BiKo01LTT}. 
Let $\varphi : \R \rightarrow [0,1]$ be smooth,
$\varphi=0$ on $(-\infty,-1]$ and $\varphi=1$ on $[1,\infty)$  for all $x\in \R$. 
\begin{calc}
One has 
\begin{align*}
\varphi( \tfrac{2x}{a} +1) = 
\begin{cases}
1 & x\in [0,\infty), \\
0 & x\in (-\infty,-a],
\end{cases} \quad
\varphi( \tfrac{2(x-r)}{a} +1) = 
\begin{cases}
1 & x\in [r,\infty), \\
0 & x\in (-\infty,r-a],
\end{cases}
\end{align*}
\end{calc}
Let 
\begin{align*}
 \zeta(x) 
& =  \sqrt{ \varphi(\tfrac{2x}{a}+1) (1- \varphi(\tfrac{2(x-r)}{a}+1))} 
\end{align*}
Then $ \zeta =0$ outside $[-a,r]$, $ \zeta = 1$ on $[0,r-a]$ and 
$
 \sum_{k\in \Z}  \zeta(x-rk)^2 
 =1. 
$ 
\begin{calc}
Indeed, we have for $x\in [r-a,r]$ that 
\begin{align*}
\varphi(\tfrac{2x}{a}+1) (1- \varphi(\tfrac{2(x-r)}{a}+1))
+ \varphi(\tfrac{2(x-r)}{a}+1) (1- \varphi(\tfrac{2(x-2r)}{a}+1)) =1, 
\end{align*}
as $\varphi(\tfrac{2x}{a}+1) = 1$ and $(1- \varphi(\tfrac{2(x-2r)}{a}+1))=1$ (because $x-2r <-a$).
Furthermore, note that 
\begin{align*}
 \zeta'(x) 
& = \frac{2}{a} 
\sqrt{ (1- \varphi(\tfrac{2(x-r)}{a}+1))}   
\frac{\dd}{\dd u} \sqrt{\varphi(u)} \Big|_{u= \frac{2x}{a} +1} \\
& \qquad + \frac{2}{a}
\sqrt{ \varphi(\tfrac{2x}{a}+1)} 
\frac{\dd}{\dd u} \sqrt{1-\varphi(u)} \Big|_{u= \frac{2(x-r)}{a} +1}
\end{align*}
\end{calc}
Moreover,
$
\|  \zeta' \|_\infty \le \frac{2}{a} [ \|\sqrt{\varphi}'\|_\infty +  \|\sqrt{1-\varphi}'\|_\infty].
$ 
Hence with $ \eta: \R^d \rightarrow [0,1]$ defined by $  \eta (x) = \prod_{i=1}^d  \zeta(x_i)$ we have \eqref{eqn:sum_squares_eta} and $\| \nabla  \eta \|_\infty \le \frac{C}{a}$ for some $C>0$. 
\end{proof}

\begin{obs}[IMS formula]
\label{obs:IMS}
Write $\eta_k(x) = \eta(x-rk)$. Then 
\begin{align*}
& \eta_k^2 \Delta \psi + \Delta (\eta_k^2 \psi) - 2 \eta_k \Delta ( \eta_k \psi) 
\cnewline
\cand \begin{calc}
= 2 \eta_k^2 \Delta \psi + \psi \Delta (\eta_k^2) + \nabla \psi \cdot \nabla (\eta_k^2) - 2 \eta_k ( \eta_k \Delta \psi + \nabla \eta_k \cdot \nabla \psi + \psi \Delta \eta_k ) 
\end{calc} \cnewline
\cand \begin{calc}
= 2 \eta_k^2 \Delta \psi + 2\psi \eta_k \Delta \eta_k + \psi |\nabla \eta_k|^2+ 2\eta_k \nabla \psi \cdot \nabla \eta_k - 2 \eta_k ( \eta_k \Delta \psi + \nabla \eta_k \cdot \nabla \psi + \psi \Delta \eta_k ) 
\end{calc} \cnewline
\cand  = \psi | \nabla \eta_k|^2.
\end{align*}
Consequently,
\begin{calc}
\begin{align*}
\Delta \psi & = \frac12 \sum_{k\in \Z^d} \eta_k^2 \Delta \psi + \Delta ( \eta_k^2 \psi) 
 = \sum_{k\in \Z^d} \eta_k \Delta( \eta_k \psi) + \Phi \psi, 
\end{align*}
and thus 
\end{calc}
with $\sH_k \psi = \eta_k \sH(\eta_k \psi)$ (where $\sH=\sH_\zeta$) and $\Phi = \sum_{k \in \Z^d} | \nabla \eta_k|^2$ \towil{$\frac12$ missing in definition $\Phi$}
\begin{align}
\label{eqn:IMS_formula}
\sH- \Phi = \sum_{k\in \Z^d}  \sH_k.
\end{align}
\eqref{eqn:IMS_formula} is also called the IMS-formula, see also \cite[Lemma 3.1]{Si83} with references to first works in which it appears. 
The technique to prove \cite[Proposition 1]{GaKo00}, which we slightly generalize, is basically the IMS-formula. 
\end{obs}

\begin{theorem}
\label{theorem:eigenvalues_boxes}
For all $r>a>0$ there is a smooth function $\Phi_{a,r} : \R^d \rightarrow [0,\infty)$ whose support is contained in the $a$-neighbourhood of the grid $r \Z^d + \partial [0,r]^d$, is periodic in each coordinate with period $r$, with 
$\|\Phi_{a,r}\|_\infty \le \frac{K}{a}$ 
for some $K>0$ that does not depend on $a$ and $r$,
 such that  $\zeta \in L^\infty(\R^d)$ and $L>r$ 
 \begin{align}
\label{eqn:eigenvalue_overlapping_box_estimate}
\lambda([0,L]^d,\zeta) - \tfrac{K}{a} 
\le 
\lambda([0,L]^d,\zeta- \Phi_{a,r})  
\le 
\max_{k\in \N_0^d , |k|_\infty <  \frac{L}{r}+1}
\lambda(rk + [-a,r]^d,\zeta).
\end{align}
\end{theorem}
\begin{proof}
Let $\eta$ be as in Lemma \ref{lemma:partition_with_a_and_r}, $\eta_k(x) = \eta(x-rk)$ and $\Phi_{a,r} = \Phi = \sum_{k\in \Z^d} |\nabla \eta_k|^2$. \towil{factor $\frac12$}
By Lemma \ref{lemma:partition_with_a_and_r} it follows that $\|\Phi\|_\infty \le \frac{K}{a}$ for some $K>0$ that does not depend on $a$ and $r$. 
Observe that $\sum_{k \in \N_0^d: |k|_\infty < \frac{L}{r}+1} \eta_k^2$ equals $1$ on $[0,L]^d$. 
With $\sH_k$ as in \ref{obs:IMS}, $\sH_k$ is self-adjoint and $\sH_k \le \lambda(rk+[-a,r]^d)\eta_k^2$ for all $k\in \Z^d$. 
Hence we have by the IMS-formula \eqref{eqn:IMS_formula} on $H_0^2([0,L]^d)$
\begin{align*}
 \sH- \Phi \le \sum_{k \in \N_0^d, |k|_\infty < \frac{L}{r}+1} \lambda(rk + [-a,r]^d) \eta_k^2
 \le \max_{k \in \N_0^d, |k|_\infty < \frac{L}{r}+1} \lambda(rk + [-a,r]^d).
\end{align*}
\begin{calc}
$\sum_{k \in \N_0^d: |k|_\infty < \frac{L}{r}+1} \eta_k^2$ indeed equals $1$ on $[0,L]^d$: Let us check this for $d=1$. 
This is the case when for $m=\max\{ k \in \N_0: k < \frac{L}{r}+1\}$ one has $L \le r(m+1)-a$. 
As $a<r$, this is the case when $\frac{L}{r} \le m$, which is the case by definition of $m$. 
\end{calc}
\end{proof}

\begin{calc}
\begin{obs}
An alternative way of proving \eqref{eqn:eigenvalue_overlapping_box_estimate} without \eqref{eqn:min_max_for_smooth} is as follows: \\
By the IMS-formula \eqref{eqn:IMS_formula} we have for all $\psi \in C_c^\infty(Q_L)$ with $\|\psi\|_{L^2} =1$, using $\sum_{k\in \N_0^d: |k|_\infty \le  \frac{L}{r} + 1 } \|\psi_k\|_{L^2}^2  =1$ for the second inequality and writing $\check \psi_k = \frac{\psi_k}{\|\psi_k\|_{L^2}}$ when $\psi_k \ne 0$ and $\check \psi_k =0$ otherwise, 
\begin{align}
\int_{\R^d} - |\nabla \psi|^2 + \left(\zeta-\Phi \right) \psi^2 
\cand \begin{calc}
\notag
 = \sum_{k\in \N_0^d: |k|_\infty \le  \frac{L}{r} + 1 } \int_{\R^d} - |\nabla \psi_k|^2 + \zeta \psi_k^2  
 \end{calc} \cnewline
 \label{eqn:formula_of_H_psi_paired_with_psi}
& = \sum_{k\in \N_0^d: |k|_\infty \le  \frac{L}{r} + 1 } \|\psi_k\|_{L^2}^2 
\int_{\R^d} - |\nabla \check \psi_k|^2 + \zeta  \check \psi_k^2 . 
\end{align}
Therefore the bound \eqref{eqn:eigenvalue_overlapping_box_estimate} follows by \eqref{eqn:min_max_for_smooth}.
\end{obs}
\end{calc}

\begin{theorem}
\label{theorem:n_eigenvalue_bound_from_below_of_principals}
Let $\zeta \in L^\infty(\R^d)$. 
Let $x,y_1,\dots,y_n \in \R^d$, $L>r>0$ be such that $(y_i + [0,r]^d)_{i=1}^n$ are pairwise disjoint subsets of $x+[0,L]^d$. 
Then 
\begin{align}
\label{eqn:lower_bound_n_eigenvalue_by_first}
\lambda_n(x+[0,L]^d,\zeta) \ge \min_{i\in \{1,\dots,n\}} \lambda(y_i +[0,r]^d ,\zeta). 
\end{align}
\end{theorem}
\begin{proof}
By \eqref{eqn:min_max_for_smooth} (see also \eqref{eqn:min_max_of_eigenvalue_with_smooth_and_convex_formulation})
\begin{align*}
\lambda_n(x+[0,L]^d,\zeta) 
\cand 
\begin{calc}
\ge \sup_{\substack{f_1,\dots,f_n, \\ f_i \in C_c^\infty(y_i+[0,r]^d), \|f_i\|_{L^2}=1} }  
\inf_{\substack{\psi = \sum_{i=1}^n \alpha_i f_i  \\ \alpha_i \in [0,1], \sum_{i=1}^n \alpha_i^2 =1 }} 
 \int -| \nabla \psi |^2 + \zeta \psi^2
\end{calc}
\cnewline
& 
\ge \sup_{\substack{f_1,\dots,f_n, \\ f_i \in C_c^\infty(y_i+[0,r]^d), \|f_i\|_{L^2}=1} }  \min_{i\in \{1,\dots,n\}}  \int -| \nabla f_i |^2 + \zeta f_i^2, 
\end{align*}
which proves \eqref{eqn:lower_bound_n_eigenvalue_by_first} by \eqref{eqn:min_max_for_smooth} with $n=1$. 
\end{proof}

\subsection{White noise as potential}
\label{subsection:white_noise_bounds}

In this section we prove analogous bounds to those in Lemma \ref{lemma:monotonicity_of_eigenvalues_in_radius}, Theorem \ref{theorem:eigenvalues_boxes} and Theorem \ref{theorem:n_eigenvalue_bound_from_below_of_principals} by replacing the bounded potential $\zeta$ by white noise, i.e., we prove Theorem \ref{theorem:lower_and_upper_bounds_white_noise}. 

\begin{theorem} \phantomsection 
\label{theorem:lower_and_upper_bounds_white_noise} 
Let $L \ge r \ge 1$. 
\begin{enumerate}
\item 
\label{item:monotonicity_of_eigenvalues_enhanced_WN}
For all $\kappa>0$  and $x,y\in \R^2$ such that $y+ Q_r\subset x+ Q_L$
\begin{align}
\label{eqn:monotonicity_of_eigenvalues_enhanced_WN}
&   \blambda_n(y+Q_r, \kappa )  \le \blambda_n(x+Q_L, \kappa ) \qquad \mbox{a.s.}
\end{align}
\item There exists a $K>0$ such that for all $\kappa>0$, $x\in \R^2$ and $a\in (0,r)$, 
\begin{align}
& \blambda( x+ Q_L, \kappa  ) 
\label{eqn:upper_bound_white_noise}
 \le \max_{k\in \N_0^2, |k|_\infty< \frac{L}{r} +1}
\blambda(x+ rk + Q_{r+a}, \kappa  ) + \tfrac{K}{a^2} \qquad \mbox{a.s.} 
\end{align}
\item For $\kappa>0$ and $x,y_1,\dots,y_n \in \R^2$ such that $(y_i + Q_r)_{i=1}^n$ are pairwise disjoint subsets of $x+Q_L$
\begin{align}
\label{eqn:lower_bound_n_eigenvalue_by_first_white_noise}
\blambda_n(x+Q_L,\kappa) \ge \min_{i\in \{1,\dots,n\}} \blambda(y_i +Q_r ,\kappa) \qquad \mbox{a.s.}
\end{align}
\end{enumerate}
\end{theorem}

Let us describe how the proof of Theorem \ref{theorem:lower_and_upper_bounds_white_noise} follows from the following theorem. 
Let $L \ge r \ge 1,\kappa>0$. 
By performing a translation over $x$ we may assume $x=0$. 

It is sufficient to show that for all $y\in \R^2$ and $r>0$ such that $y+Q_r \subset Q_L$  one has the following convergences in probability (and thus almost surely along a sequence $(\epsilon_n)_{n\in\N}$ in $(0,1)$ that converges to $0$) 
\begin{align}
\label{eqn:convergence_ev_on_smaller_box}
\lambda_n(y+Q_r, \kappa (\xi'_{L,\epsilon}  - c_\epsilon')) \xrightarrow{\P} \blambda_n(y+Q_r, \kappa),
\end{align}
for the right choices $\xi'_{L,\epsilon}$ and $c_\epsilon'$.  
Indeed, for \eqref{eqn:monotonicity_of_eigenvalues_enhanced_WN} and \eqref{eqn:lower_bound_n_eigenvalue_by_first_white_noise} this is clearly sufficient. For \eqref{eqn:upper_bound_white_noise} this is sufficient by ``replacing $L$'' in \eqref{eqn:convergence_ev_on_smaller_box} by ``$3L$'' and ``replacing $r$'' by either ``$L$'' or ``$r+a$''. 

In this case, we choose $\xi'_{L,\epsilon}$ like $\xi_{L,\epsilon}$ in \eqref{eqn:xi_epsilon_dir} but with $\tau' = \1_{(-1,1)^2}$ instead of $\tau$ and $c_\epsilon'= \frac{1}{2\pi} \log \frac{1}{\epsilon} + c_{\tau'}$ (the choice of $\tau'= \1_{(-1,1)^2}$ is convenient for  calculations in Section~\ref{section:proof_convergence_diff_box_sizes}). 
Observe that 
\begin{align*}
\lambda_n(y+Q_r, \kappa \xi_{L,\epsilon}') 
= \lambda_n(y+Q_r, \kappa \theta_\epsilon^y)
= \lambda_n(y+Q_r, (\kappa \theta_\epsilon^y, \kappa^2 \theta_\epsilon^y \reso \sigma(\rD) \theta_\epsilon^y)) ,
\end{align*}
for $\theta_\epsilon^y$ (which equals $\xi_{L,\epsilon}|_{y+Q_r}$ in $L^2(y+Q_r)$) given by 
\begin{align}
\notag
\theta_\epsilon^y 
& = 
\sum_{k\in \N_0^2} \langle \xi_{L,\epsilon}  , \cT_y \fn_{k,r} \rangle_{L^2(y+Q_r)} \cT_y \fn_{k,r} \\
\label{eqn:theta_epsilon}
& = \sum_{k\in \N_0^2} \sum_{m\in \N_0^2} 
\1_{(-1,1)^2}(\tfrac{\epsilon}{L}m) \langle \cW, \fn_{m,L} \rangle
 \langle \fn_{m,L}  , \cT_y \fn_{k,r} \rangle_{L^2(y+Q_r)} \cT_y \fn_{k,r} . 
\end{align}
Therefore the following theorem resembles the missing part of the proof. 
Observe that $\theta_\epsilon^y \reso \sigma(\rD) \theta_\epsilon^y \in H_\fn^1 \subset \cC_\fn^{0}$ as $\theta_\epsilon^y \in L^2= H_\fn^0$ (see also \ref{obs:rough_distributions_for_H_elements}).

\begin{theorem}
\label{theorem:theta_epsilon_converges_to_xi_r}
Let $L > r \ge 1$ and $x,y\in \R^2$ be such that $y+Q_r \subset x+ Q_L$.
Let  $\theta_\epsilon^y$ be as in \eqref{eqn:theta_epsilon}. 
Then $(\xi_{L,\epsilon}', \xi_{L,\epsilon}' \reso \sigma(\rD) \xi_{L,\epsilon}' - c_\epsilon') \xrightarrow{\P} \bxi_L$ in $\fX_\fn^{\alpha}(Q_L)$ and $(\theta_\epsilon^{y}, \theta_\epsilon^{y} \reso \sigma(\rD) \theta_\epsilon^{y} - c_\epsilon' ) \xrightarrow{\P} \bxi_r^{y}$
 in $\fX_\fn^\alpha(y + Q_r) $. 
\end{theorem}

We prove Theorem \ref{theorem:theta_epsilon_converges_to_xi_r} in Section \ref{section:proof_white_noise_conv}: it follows from Theorem \ref{theorem:convergence_prima_and_theta}.

%
%
%

\section{Large deviation principle of the enhancement of white noise}
\label{section:ldp_enhancement}

In this section we assume $L>0$ and write $\bxi = (\xi, \Xi)$ for the limit $\bxi_L$ as in Theorem \ref{theorem:convergence_in_enhanced_space_to_white_noise}. 
We prove the following theorem.

\begin{theorem}
\label{theorem:ldp_white_noise_enhancement_dir}
$(\sqrt{\epsilon}\xi, \epsilon \Xi)$ satisfies the large deviation principle with rate $\epsilon$ and rate function $\fX_\fn^\alpha \rightarrow [0,\infty]$, 
$ 
(\psi_1, \psi_2) \mapsto \tfrac12 \|\psi_1\|_{L^2}^2. 
$ 
\end{theorem}

\begin{remark}
Analogously, by some lines of the proof in a straightforward way, the statement in Theorem \ref{theorem:ldp_white_noise_enhancement_dir} holds with underlying space the torus and $(\xi,\Xi)$ being the analogue limit as in Theorem \ref{theorem:convergence_in_enhanced_space_to_white_noise} as is considered in \cite{AlCh15}.
\end{remark}

As a direct consequence of this large deviation principle and the continuity of the eigenvalues in the (enhanced) noise (see \eqref{eqn:locally_lipschitz_eigenvalue_of_dirichlet}), we obtain the following by an application of the contraction principle (see \cite[Theorem 4.2.1]{DeZe10}). 

\begin{corollary}
\label{cor:ldp_eigenvalues}
$\blambda_n(Q_L,\epsilon)=\lambda_n(Q_L, (\epsilon\xi_L, \epsilon^2 \Xi_L) )$ satisfies the large deviation principle with rate $\epsilon^2$ and rate function $I_{L,n} : \R \rightarrow [0,\infty]$ given by 
\begin{align}
\label{eqn:rate_function_n_th_eigenvalue}
I_{L,n}(x) = \inf_{\overset{V\in L^2(Q_L)}{ \lambda_n(Q_L,V) = x}} \tfrac12 \|V\|_{L^2}^2. 
\end{align}
\end{corollary}

Theorem \ref{theorem:ldp_white_noise_enhancement_dir} is an extension of the following theorem. 
A proof can be given by using \cite[Theorem 3.4.5]{DeSt89}, but as our proof is rather simple and -- to our knowledge -- different from proofs in literature, we include it. 

\begin{theorem}
\label{theorem:ldp_white_noise_dir}
$\sqrt{\epsilon}\xi$ satisfies the large deviation principle with rate function $\cC_\fn^\alpha([0,L]^d) \rightarrow [0,\infty]$ given by 
$
\psi \mapsto \tfrac12 \|\psi\|_{L^2}^2. 
$ 
\end{theorem}
\begin{proof}

We use the Dawson-G\"artner projective limit theorem \cite[Theorem 4.6.1]{DeZe10} and the inverse contraction principle \cite[Theorem 4.2.4]{DeZe10}. 
Let $J = \N$ with its natural ordering. 
Let $\cY_i = \R^i$ for all $i \in J$. 
Let $p_{ij}$ be the projection $\cY_j \rightarrow \cY_i$ on the first $i$-coordinates. 
Let $\cY$ be the projective limit $\lim_{\leftarrow} \cY_j$ (see \cite[above Theorem 4.6.1]{DeZe10}, it is a subset of $\prod_{j\in J} \cY_j$). 
Let $p_j : \cY \rightarrow \cY_i$ be the canonical projection. 

Let $\fs: \N \rightarrow \N_0^d$ be a bijection. 
Write $ \fd_n' = \fd_{\fs(n)}$. 
Let $\Phi: \cC_\fn^\alpha([0,L]^d) \rightarrow \cY$ be given by $\Phi(u) = ( \langle u, \fd_1' \rangle , \dots, \langle u, \fd_n'\rangle )_{n\in\N}$. This $\Phi$ is continuous and injective. 
We first prove that $\Phi \circ \xi$ satisfies the large deviation principle. 

For every $n\in\N$ the vector $( \langle \xi, \fd_1' \rangle , \dots, \langle \xi, \fd_n'\rangle )$ is an $n$-dimensional standard normal variable, whence $\sqrt{\epsilon}( \langle \xi, \fd_1' \rangle , \dots, \langle \xi, \fd_n'\rangle ) = ( \langle \sqrt{\epsilon}\xi, \fd_1' \rangle , \dots, \langle \sqrt{\epsilon}\xi, \fd_n'\rangle )$ satisfies a large deviation principle on $\R^n$ with rate function given by 
$
I_n(y) :=  \frac12 |y|^2 = \frac12 \sum_{i=1}^n y_i^2.
$ 
By the Dawson-G\"artner projective limit theorem the sequence $\sqrt{\epsilon}( \langle \xi, \fd_1' \rangle , \dots, \langle \xi, \fd_n'\rangle )_{n\in\N}$ satisfies the large deviation principle on $\cY$ with rate function 
\begin{align*}
I( (y_1,\dots,y_n)_{n\in\N}) = \sup_{n\in\N} I_n(y_1,\dots,y_n) = \sup_{n\in\N} \frac12 \sum_{i=1}^n y_i^2. 
\end{align*}
The image of $\cC_\fn^\alpha$ under $\Phi$ is measurable, which follows from the following identity
\begin{align*}
\Phi(\cC_\fn^\alpha) = 
\Big\{ 
(a_1,\dots, a_n)_{n\in\N}: \sup_{i\in \N_{-1}} \Big\| \sum_{n\in\N } \rho_i( \tfrac{\fs(n)}{L} ) a_n \fd_n'  \Big\|_\infty <\infty 
\Big\}.
\end{align*}
\begin{calc}
Indeed, 
\begin{align*}
(a_1,\dots, a_n)_{n\in\N} \mapsto \Big\| \sum_{n\in\N } \rho_i( \tfrac{\fs(n)}{L} ) a_n \fd_n'  \Big\|_\infty
\end{align*}
is continuous as $\rho_i( \tfrac{\fs(n)}{L} ) \ne 0$ for only finitely many $n$. 
\end{calc}
As $\P( \Phi (\sqrt{\epsilon} \xi ) \in \Phi (\cC_\fn^\alpha) ) =1$, and the domain on which $I$ is finite is contained in $\Phi(\cC_\fn^\alpha)$, i.e., $\{ y\in \cY : I(y) <\infty\} \subset \Phi (\cC_\fn^\alpha)$, by \cite[Theorem 4.1.5]{DeZe10} $\Phi (\sqrt{\epsilon} \xi)$ satisfies the large deviation principle on $\Phi(\cC_\fn^\alpha)$ with rate function $I$ (restricted to $\Phi(\cC_\fn^\alpha)$). 

Now we apply the inverse contraction principle. $\Phi : \cC_\fn^\alpha \rightarrow \Phi(\cC_\fn^\alpha)$ is a continuous bijection. Also $I \circ \Phi (\psi) = \frac12 \|\psi\|_{L^2}^2$ (by Parseval's identity). Hence the proof is finished by showing that $\sqrt{\epsilon} \xi$ is exponentially tight in $\cC_\fn^\alpha$. 
Let $m>0$ and $K_m := \{ \psi \in \cC_\fn^\alpha : I \circ \Phi(\psi) \le m\}$. 
As $ L^2$ is compactly embedded in $H_\fn^{\alpha+1}$ by Theorem \ref{theorem:compact_embeddings_besov_spaces}, which is continuously embedded in $\cC_\fn^\alpha$ (by \cite[Theorem 2.71]{BaChDa11}, $K_m$ is relatively compact in $\cC_\fn^\alpha$. By the large deviation principle of $\Phi(\sqrt{\epsilon}\xi)$ on $\Phi(\cC_\fn^\alpha)$, and because $\overline{K_m}^c \subset K_m^c$, it follows that 
\begin{align*}
\limsup_{\epsilon \downarrow 0} \epsilon \log \P( \sqrt{\epsilon} \xi \in \overline{K_m}^c) 
\cand \begin{calc}
\le 
\limsup_{\epsilon \downarrow 0} \epsilon \log \P( \sqrt{\epsilon} \xi \in K_m^c) 
\end{calc} \cnewline
&
= \limsup_{\epsilon \downarrow 0} \epsilon \log \P( \Phi(\sqrt{\epsilon} \xi)  \in \{ y \in \cY : I \le m\}^c)
\le -m. 
\end{align*}
This proves the exponential tightness of $\sqrt{\epsilon} \xi$ in $\cC_\fn^\alpha$, which finishes the proof. 
\end{proof}

To prove Theorem \ref{theorem:ldp_white_noise_enhancement_dir} we use Theorem \ref{theorem:ldp_white_noise_dir}  and the extension of the contraction principle: 

\begin{theorem}
\cite[Theorem 4.2.23]{DeZe10}
\label{theorem:extended_contraction}
Let $\cX $ be a Hausdorff space and $(\cY,d)$ be a metric space. 
Suppose that $(\eta_\epsilon)_{\epsilon >0}$ are random variables with values in $\cX$ that satisfy the large deviation principle with (rate $\epsilon$ and) rate function $I: \cX \rightarrow [0,\infty]$. 
Suppose furthermore that 
$F_\delta : \cX \rightarrow \cY$ is a continuous map for all $\delta>0$, $F: \cX \rightarrow \cY$ is measurable and that for all $q\in [0,\infty)$ 
\begin{align}
\label{eqn:f_epsilon_close_to_f_on_sublevel_sets}
\lim_{\delta \downarrow 0} \sup_{x\in \cX : I(x) \le q} d( F_\delta (x), F(x) ) =0, 
\end{align}
and that $F_\delta(\eta_\epsilon)$ are exponential good approximations for $F(\eta_\epsilon)$, i.e., if for all $\kappa>0$ 
\begin{align}
\label{eqn:exponential_good_approximations}
\lim_{\delta \downarrow 0} \limsup_{\epsilon \downarrow 0} 
\epsilon \log \P ( d ( F_\delta (\eta_\epsilon) , F(\eta_\epsilon))> \kappa) = -\infty. 
\end{align}
Then $F(\eta_\epsilon)$ satisfies the large deviation principle with rate function $\cY \rightarrow [0,\infty]$ given by 
\begin{align*}
y \mapsto \inf_{x\in \cX : F(x) = y} I(x). 
\end{align*}
\end{theorem}

\begin{lemma}
\label{lemma:uniform_bound_mollifier}
Let $\alpha \in (-\frac43, -1)$. 
Let $\tau : \R^2 \rightarrow [0,1]$ be a compactly supported function that equals $1$ on a neighbourhood of $0$. 
Write $h_\delta = \tau(\delta \rD)h$. 
There exists a $C>0$ such that for all $\delta >0$ and $h\in L^2$
\begin{align}
\label{eqn:uniform_bound_reso_prod_molification}
\| h_\delta \reso \sigma(\rD) h_\delta - h \reso \sigma(\rD) h\|_{\cC_\fn^{2\alpha+2}} \le C \delta^{-\alpha-1} \|h\|_{L^2}^2.
\end{align}
\end{lemma}
\begin{proof}
This follows by Theorem \ref{theorem:bony_estimates} (note $2\alpha +4>0$),  Theorem \ref{theorem:inclusion_of_H_into_C} (also using $\|h_\delta\|_{H_\fn^{\alpha+1}}\lesssim \|h\|_{H_\fn^{\alpha+1}} \lesssim \|h \|_{L^2}$; see also \ref{obs:H_0_equals_L_2}) and Theorem \ref{theorem:fourier_mollifiers_converge_in_B_22_space}: 
\begin{align*}
& \| h_\delta \reso \sigma(\rD) h_\delta - h \reso \sigma(\rD) h\|_{\cC_\fn^{2\alpha+2}} \\
&
 \qquad \le 
\| (h-h_\delta) \reso \sigma(\rD) h_\delta \|_{H_\fn^{2\alpha+4}} 
+ \| h \reso \sigma(\rD) ( h_\delta - h) \|_{H_\fn^{2\alpha+4}} 
\\ 
& \qquad \lesssim 
\| h-h_\delta \|_{H_\fn^{\alpha +1}} \|h \|_{H_\fn^{\alpha+1}}
\lesssim \delta^{- \alpha -1} \|h\|_{L^2}^2. 
\end{align*}
\end{proof}

\begin{proof}[Proof of Theorem \ref{theorem:ldp_white_noise_enhancement_dir}]
For $\delta>0$ we write $h_\delta = \tau(\delta \rD) h$ for $\tau$ as in \ref{obs:overview_white_noise_section}
 and define $F_\delta : \cC_\fn^\alpha(Q_L) \rightarrow \fX_\fn^\alpha(Q_L)$ by 
\begin{align*}
F_\delta(h) = (h, h_\delta \reso \sigma(\rD) h_\delta).
\end{align*}
We define $F : \cC_\fn^\alpha(Q_L) \rightarrow \fX_\fn^\alpha(Q_L)$ as follows. 
If for $h\in \cC^\alpha_\fn(Q_L)$ the function $h_\delta \reso \sigma(\rD) h_\delta$ converges in $\cC_\fn^{2\alpha+2}$, then $F(h) = \lim_{\delta \downarrow 0} (h,h_\delta \reso \sigma(\rD) h_\delta)$;
if $h_\delta \reso \sigma(\rD) h_\delta$ does not converge, but $h_\delta \reso \sigma(\rD) h_\delta - c_\delta$ does (where $c_\delta=\tfrac{1}{2\pi} \log(\tfrac{1}{\delta}) + c_\tau$), then  define $F(h)= \lim_{\delta \downarrow 0} (h,h_\delta \reso \sigma(\rD) h_\delta- c_\delta)$; whereas if $h_\delta \reso \sigma(\rD) h_\delta - c_\delta$ also does not converge, then $F(h) = 0$. 

With $\cX = \cC_\fn^\alpha(Q_L)$ and $\cY = \fX_\fn^\alpha(Q_L)$ and $\eta_\epsilon = \sqrt{\epsilon}\xi$, by Theorem \ref{theorem:ldp_white_noise_dir} and Theorem \ref{theorem:extended_contraction} it is sufficient to prove that \eqref{eqn:f_epsilon_close_to_f_on_sublevel_sets} and \eqref{eqn:exponential_good_approximations} hold because when  $F(\phi) = (\psi_1,\psi_2) \ne 0$ then $\phi = \psi_1$. 

$\bullet$ First we check \eqref{eqn:f_epsilon_close_to_f_on_sublevel_sets}. 
By Lemma \ref{lemma:uniform_bound_mollifier} we have ($F(h) = (h,h\reso \sigma(\rD) h)$ and) 
\begin{align*}
 \sup_{h\in \cC_\fn^\alpha (Q_L) : \|h\|_{L^2}\le q} \|F_\delta(h) - F(h)\|_{\fX_\fn^\alpha} \lesssim \delta^{-\alpha-1} q^2, 
\end{align*}
for all $q \ge 0$, i.e., \eqref{eqn:f_epsilon_close_to_f_on_sublevel_sets} holds. 

$\bullet$ Now we check \eqref{eqn:exponential_good_approximations}. 
Let $\kappa>0$. 
We have that $\Xi:= \lim_{\delta\downarrow 0 } \xi_\delta \reso \sigma(\rD) \xi_\delta -c_\delta$ exists almost surely  by Theorem \ref{theorem:convergence_in_enhanced_space_to_white_noise}.
Hence, for $p>1$
\begin{align*}
\P\left( \| F_\delta(\sqrt{\epsilon}\xi) - F(\sqrt{\epsilon}\xi)\|_{\fX_\fn^\alpha} >\kappa \right)
& \le \frac{\epsilon^{p}}{\kappa^p} \E \left[ 
\| \xi_\delta \reso \sigma(\rD) \xi_\delta - \Xi  \|_{\cC_\fn^{2\alpha+2}}^p \right] \\
& \le \frac{\epsilon^{p} 2^{p} }{\kappa^p} (c_\delta^p + \E \left[ 
 \| \xi_\delta \reso \sigma(\rD) \xi_\delta -c_\delta  - \Xi \|_{\cC_\fn^{2\alpha+2}}^p \right]) 
\end{align*}
Let $\eta = - (2\alpha +2) $. 
By Lemmas \ref{lemma:gaussian_besov_bound},  \ref{lemma:expectation_reso_xi_eps_C_0_plus_reso_constant},  \ref{lemma:renormalisation_constant} and  \ref{lemma:bound_on_expect_mod_Delta_i_difference_Xi_in_x_squared} 
 there exists a $C>0$ such that for all $p>1$ 
\begin{align*}
\E \left[ 
 \| \xi_\delta \reso \sigma(\rD) \xi_\delta -c_\delta  - \Xi \|_{\cC_\fn^{2\alpha+2}}^p \right]
 \le C^p p^p \delta^{\eta p}.
\end{align*}
Therefore (using that $a^p + b^p \le (a+b)^p$) 
\begin{align*}
\P\left( \| F_\delta(\sqrt{\epsilon}\xi) - F(\sqrt{\epsilon}\xi)\|_{\fX_\fn^\alpha} >\kappa \right) \le \left[ \frac{2\epsilon}{\kappa} (c_\delta + Cp \delta^\eta ) \right]^p
\end{align*}
Hence with $p= \frac{1}{\epsilon}$ we obtain
\begin{align*}
\limsup_{\epsilon\downarrow 0} \epsilon \log \P\left( \| F_\delta(\sqrt{\epsilon}\xi) - F(\sqrt{\epsilon}\xi)\|_{\fX_\fn^\alpha} >\kappa \right)
& \le 
\limsup_{\epsilon\downarrow 0}  \log \left[ \tfrac{2}{\kappa} ( \epsilon c_\delta + C\delta^\eta ) \right] \\
& \le  \log (\tfrac{2C}{\kappa} \delta^\eta ).
\end{align*}
So that 
\begin{align*}
\lim_{\delta \downarrow 0} \limsup_{\epsilon\downarrow 0} \epsilon \log \P\left( \| F_\delta(\sqrt{\epsilon}\xi) - F(\sqrt{\epsilon}\xi)\|_{\fX_\fn^\alpha} >\kappa \right) = - \infty, 
\end{align*}
i.e., \eqref{eqn:exponential_good_approximations} holds. 
\end{proof}

\section{Infima over the large deviation rate function}
\label{section:ldp_results}

In this section we consider infima over sets of the rate function $I_{L,n}$ as in  \eqref{eqn:rate_function_n_th_eigenvalue}. 
We prove the results summarized in Theorem \ref{theorem:convergence_infima_rate_function}.

\begin{lemma}
\label{lemma:infima_rate_function_translated_by_constant}
For $a,b\in \R$ and all $\delta>0$ 
\begin{align*}
(1 - \delta) \inf I_{L,n} [b,\infty)  + \tfrac12 (1- \tfrac{1}{\delta})  L^2 a^2  
 &\le 
 \inf I_{L,n} [b+a,\infty) \\
&  \le (1+ \delta) \inf I_{L,n} [b,\infty)  + \tfrac12 (1+ \tfrac{1}{\delta})  L^2 a^2. 
\end{align*}
Consequently, 
for 
$(a_L)_{L>0}$ in $\R$ with $\lim_{L\rightarrow \infty} L a_L =0$, 
\begin{align*}
\lim_{L\rightarrow \infty} \inf I_{L,n} [b, \infty) 
&= \lim_{L\rightarrow \infty} \inf I_{L,n} [b+a_L, \infty) \\
&= \lim_{L\rightarrow \infty} \inf I_{L,n} (b+a_L, \infty) 
= \lim_{L\rightarrow \infty} \inf I_{L,n} (b, \infty) .
\end{align*}
\end{lemma}
\begin{proof}
As  $\lambda_n( Q_L, V) + a = \lambda_n(Q_L, V+a \1_{Q_L})$, 
 $\| a \1_{Q_L} \|_{L^2} = a L$, and  $2 \langle V, a \1_{Q_L} \rangle \le \delta \|V\|_{L^2}^2 + \frac{1}{\delta} a^2 L^2$ for all $\delta >0$; 
\begin{align*}
 \inf I_{L,n} [b+a,\infty) 
\cand 
\begin{calc}
  = \inf_{ \overset{V\in L^2(Q_L)}{ \lambda_n(Q_L,V-a\1_{Q_L}) \ge b} } \tfrac12 \|V\|_{L^2(Q_L)}^2 
\end{calc}
\cnewline
&  = \inf_{ \overset{V\in L^2(Q_L)}{ \lambda_n(Q_L,V) \ge b} } \tfrac12 \|V + a \1_{Q_L} \|_{L^2(Q_L)}^2  \\
& \le (1+ \delta) \inf_{ \overset{V\in L^2(Q_L)}{ \lambda_n(Q_L,V) \ge b} } \tfrac12  \|V  \|_{L^2(Q_L)}^2  + \tfrac12 (1+ \tfrac{1}{\delta}) a^2 L^2 . 
\end{align*}
The lower bound can be proven similarly. 
\end{proof}

We define 
\begin{align}
\label{eqn:mu_nL}
\mu_{L,n} &  := 
\inf I_{L,n} [1, \infty), \qquad 
\varrho_n  :=\inf_{L>0} \mu_{L,n}. 
\end{align}
We prove that $\varrho_n$ is bounded away from $0$ uniformly in $n$ (Lemma \ref{lemma:infima_rate_function_over_larger_than_q_downarrow_limit}) and give an alternative variational formula for $\varrho_n$ (Lemma \ref{lemma:formula_2_of_rho_n}) from which we conclude Theorem \ref{theorem:convergence_infima_rate_function}. 

\begin{lemma}
\label{lemma:infimum_over_smooth}
$ 
\mu_{L,n} 
= \inf I_{L,n} (1, \infty) 
= \inf_{\overset{V\in C_c^\infty(Q_L)}{ \lambda_n(Q_L,V) \ge 1}} \tfrac12 \|V\|_{L^2}^2. 
$ 
\end{lemma}
\begin{proof}
The first equality follows by Lemma \ref{lemma:infima_rate_function_translated_by_constant}. 
The second  follows by 
Lemma \ref{lemma:continuity_of_eigenvalues_on_H_00}. 
\end{proof}

We will use Ladyzhenskaya's inequality \cite{La58}, which is a special case of the Gagliardo--~Nirenberg interpolation inequality \cite{Ni59}. 

\begin{lemma}[Ladyzhenskaya's inequality]
\label{lemma:ladyzhenskaya}
There exists a $C>0$ such that for $f\in H^1(\R^2)$, 
\begin{align}
\label{eqn:name_it_later}
\|f\|_{L^4}^4 \le C \|\nabla f\|_{L^2}^2 \|f\|_{L^2}^2. 
\end{align}
\end{lemma}

\begin{lemma}
\label{lemma:infima_rate_function_over_larger_than_q_downarrow_limit}
Let $C>0$ be as in Lemma~\ref{lemma:ladyzhenskaya}. Then $\varrho_n \ge \frac{2}{C}$ for all $n\in\N$. 
\end{lemma}
\begin{proof}
Let $n\in\N$. 
Let $L>0$ and $\epsilon>0$. 
Let $V\in C_c^\infty(Q_L)$ 
 be such that $\lambda_n(Q_L,  V)\ge 1$ and $\frac12 \|V\|_{L^2}^2 \le \mu_{L,n} + \epsilon$. 
By \eqref{eqn:min_max_for_smooth} there is a $\psi \in C_c^\infty(Q_L)$ with $\|\psi\|_{L^2} =1$ such that (by integration by parts)
\begin{align*}
1 - \epsilon  \le - \|\nabla \psi\|_{L^2}^2 + \int V\psi^2 
\le 
- \|\nabla \psi\|_{L^2}^2 + \|V\|_{L^2} \|\psi\|_{L^4}^2.
\end{align*}
Hence by using Ladyzhenskaya's inequality \eqref{eqn:name_it_later}, which implies $\| \nabla \psi\|_{L^2}^2 \ge \frac{1}{C} \|\psi\|_{L^4}^4$, 
\begin{align*}
\|V\|_{L^2}
\ge 
\frac{1-\epsilon +\|\nabla \psi\|_{L^2}^2}{ \|\psi\|_{L^4}^2} 
\ge 
\frac{1-\epsilon }{ \|\psi\|_{L^4}^2} 
 +\frac{1}{C} \| \psi\|_{L^4}^2
\end{align*}
As $a^2 + b^2 \ge 2 a b$ we have 
\begin{calc}
\begin{align*}
\|V\|_{L^2} \ge 2 \sqrt{\frac{1-\epsilon}{C}}, 
\end{align*}
and thus 
\end{calc}
$ 
\mu_{L,n} + \epsilon \ge \tfrac12 \|V\|_{L^2}^2  \ge 2 \frac{1-\epsilon}{C}.
$ 
As this holds for all $\epsilon>0$ we conclude that $\mu_{L,n} \ge \frac{2}{C}$ for all $L>0$. Hence $\varrho_n \ge \frac{2}{C}$. 
\end{proof}

\begin{lemma}
\label{lemma:formula_2_of_rho_n}
For all $n\in\N$, $a>0$,   
\begin{align}
\label{eqn:formula_2_of_rho_n}
 \inf_{L>0} \inf_{ \overset{V\in C_c^\infty(Q_L)}{ \lambda_n(Q_L,V) \ge a} } \tfrac12 \|V\|_{L^2(Q_L)}^2
 =
 \inf_{L>0} \inf_{ \overset{ V\in C_c^\infty(Q_L)}{ \|V\|_{L^2}^2 \le  \frac{1}{a} } } \tfrac{1}{2 \lambda_n(Q_L, V) }. 
\end{align}
Moreover, $\mu_{L,n}$ is decreasing in $L$, and one could replace  ``$\inf_{L>0}$'' in \eqref{eqn:formula_2_of_rho_n} by ``$\lim_{L\rightarrow \infty}$''. In particular, $\varrho_n = \lim_{L\rightarrow \infty} \mu_{L,n}$.
\end{lemma}
\begin{proof}
With $W= L^2 V(L \cdot)$ we have $W\in C_c^\infty(Q_1)$, $\|W\|_{L^2(Q_1)}^2 = L^2 \|V\|_{L^2(Q_L)}^2$ and by Theorem \ref{lemma:scaling_smooth_potentials_dir}
$\lambda_n(Q_L,V)  
= \lambda_n(Q_L, \tfrac{1}{L^2} W(\tfrac{1}{L} \cdot) ) 
= \tfrac{1}{L^2} \lambda_n(Q_1,  W )$.
Therefore 
\begin{align}
\inf_{ \overset{V\in C_c^\infty(Q_L)}{ \lambda_n(Q_L,V) \ge a} } \tfrac12 \|V\|_{L^2(Q_L)}^2
& = \inf_{ \overset{W\in C_c^\infty(Q_1)}{ \lambda_n(Q_1,W) \ge a L^2 } } \tfrac12 \tfrac{1}{L^2} \|W\|_{L^2(Q_1)}^2, \\
\label{eqn:inf_over_one_over_lambda}
\inf_{ \overset{ V\in C_c^\infty(Q_L)}{ \|V\|_{L^2}^2 \le  \frac{1}{a} } } \tfrac{1}{2 \lambda_n(Q_L, V) }
&=\inf_{ \overset{ W\in C_c^\infty(Q_1)}{ \|W\|_{L^2}^2 \le  \frac{L^2}{a} } } \tfrac{L^2}{2 \lambda_n(Q_1, W) }. 
\end{align}
With this, \eqref{eqn:formula_2_of_rho_n} follows directly from Lemma \ref{lemma:infimum_duality}. That $\mu_{L,n}$ and the left-hand side of \eqref{eqn:inf_over_one_over_lambda} 
are decreasing in $L$ follows from Lemma \ref{lemma:monotonicity_of_eigenvalues_in_radius}. 
\end{proof}

\begin{lemma}
\label{lemma:infimum_duality}
Let $\cY$ be a topological space and $f,g : \cY \rightarrow \R$ be continuous functions. Let $a >0$ and suppose that 
$\varrho:=\inf_{L>0} \inf_{  w\in \cY : f(w) \ge a L } \frac{g(w)}{L} >0$. 
Then
\begin{align*}
\inf_{L>0} \inf_{ \overset{ w\in \cY}{ f(w) \ge a L} } \frac{g(w)}{L} 
=
\inf_{L>0} \inf_{ \overset{ w\in \cY}{ g(w) \le \frac{L}{a} } }  \frac{L}{  f( w)}.
\end{align*}
\end{lemma}
\begin{proof}
By definition we have 
$\forall L>0 \ \forall w\in \cY 
: \frac{1}{L} g(w) < \varrho  \Longrightarrow f(w) < a L$, 
by continuity of $f$ and $g$ we obtain (by taking $K= L \varrho a$)
\begin{align*}
\forall K>0 \ \forall w\in \cY :  g(w)  \le \frac{K}{a} \Longrightarrow \frac{ f(w)}{K} \le \frac{1}{\varrho} . 
\end{align*}
Let $\epsilon>0$. 
Then there exists an $L>0$ and $w_L \in \cY $ such that $f(w_L) \ge a L$ and $\frac{1}{L} g(w_L)  \le \varrho+ \epsilon$. 
Then with $K= L a(\varrho+\epsilon)$ we have for $w=w_L$ that 
$\frac{g(w)}{K}   \le \frac{1}{a}$ and $\frac{f(w)}{K} \ge  \frac{1}{\varrho+ \epsilon} $. 
So that 
\bfold 
\sup_{K>0} \sup_{ \overset{ w\in \cY}{ g(w) \le \frac{K}{a} } } \tfrac{f(w)}{K} = \frac{1}{  \varrho }. 
\efold  
\end{proof}

\begin{proof}[Proof of Theorem~\ref{theorem:convergence_infima_rate_function}]
By  \eqref{eqn:min_max_for_smooth} and Lemma \ref{lemma:formula_2_of_rho_n} (for $a=1$) we have 
\begin{align*}
\frac{2}{\varrho_n} 
\cand 
\begin{calc}
=  \sup_{L >0} \sup_{ \overset{ V\in C_c^\infty(Q_L)}{ \|V\|_{L^2}^2 \le 1} } 4 \lambda_n(Q_L, V) 
\end{calc}  \cnewline
& = 4
\sup_{L >0} \sup_{ \overset{ V\in C_c^\infty(Q_L)}{ \|V\|_{L^2}^2 \le 1} } 
\sup_{ \overset{ F \sqsubset C_c^\infty(Q_L) }{ \dim F = n} } 
\inf_{ \overset{ \psi \in F }{ \|\psi \|_{L^2}^2 =1} } 
\int_{Q_L} - | \nabla \psi|^2 + V \psi^2, 
\end{align*}
from which \eqref{eqn:2_div_varrho_n_var_formula} follows. 
By Cauchy-Schwarz, for $\psi \in C_c^\infty(\R^2)$, the supremum  of $\int V \psi^2$ with respect to $V\in C_c^\infty(\R^2)$ with $L^2$ norm equal to $1$ is attained at $V= \frac{\psi^2}{\|\psi^2\|_{L^2}}$; therefore this supremum equals $\|\psi\|_{L^4}^2$ and hence we derive the first equality in \eqref{eqn:2_div_varrho_1_var_formula_and_lady_constant}. 
In Lemma~\ref{lemma:infima_rate_function_over_larger_than_q_downarrow_limit} we have already seen that $\frac{2}{\rho_1}\le \chi$. For the other inequality, we refer to \cite[Theorem C.1]{Ch10} (basically the trick is to replace ``$\psi$'' by ``$\lambda f(\lambda \cdot)$'' and optimise over $\lambda>0$ first, then over $f\in L^2$ with $\|f\|_{L^2}=1$).
\begin{calc}
First, let us check that the $L^2$ norm of the rescaled function equals the $L^2$ norm of the original function. 
\begin{align*}
\|\lambda^{\frac{d}{2}} f(\lambda \cdot) \|_{L^2}^2
= \int \lambda^{d} |f(\lambda x )|^2 \dd x 
= \int |f(y )|^2 \dd y = \|f\|_{L^2}. 
\end{align*}
Therefore
\begin{align*}
\frac{2}{\varrho_1} 
& = 4 \sup_{L >0} \sup_{ \overset{ V\in C_c^\infty(Q_L)}{ \|V\|_{L^2}^2 \le 1} } 
\sup_{ \overset{ \psi \in F }{ \|\psi \|_{L^2}^2 =1} } 
\int_{Q_L} - | \nabla \psi|^2 + V \psi^2, \\
& =
4 \sup_{ \overset{ \psi \in C_c^\infty(\R^2) }{ \|\psi \|_{L^2}^2 =1} } 
\|\psi\|_{L^4}^2 - 
\int_{\R^2} | \nabla \psi|^2 \\
& = 
4 \sup_{ \overset{ f \in C_c^\infty(\R^2) }{ \|\psi \|_{L^2}^2 =1} } \sup_{\lambda>0}
\left( \int_{\R^2} \lambda^{4} |f(\lambda x)|^4 \dd x \right)^{\frac12}
 - 
 \int_{\R^2} \lambda^4 | \nabla f(\lambda x)|^2 \dd x \\
 & = 
4 \sup_{ \overset{ f \in C_c^\infty(\R^2) }{ \|\psi \|_{L^2}^2 =1} } \sup_{\lambda>0}
\lambda \left( \int_{\R^2}  |f(y)|^4 \dd y \right)^{\frac12}
 - 
\lambda^{2} \int_{\R^2}  | \nabla f(y)|^2 \dd y. 
\end{align*}
The concave function $a\lambda - b\lambda^2$ attains its maximum where the derivative equals $0$: At $\lambda = \frac{a}{2b}$. 
Hence the maximum equals $\frac{a^2}{2b} - \frac{a^2}{4b}= \frac{a^2}{4b}$. 

Hence 
\begin{align*}
\frac{2}{\varrho_1} 
 & =  \sup_{ \overset{ f \in C_c^\infty(\R^2) }{ \|\psi \|_{L^2}^2 =1} } 
\frac{\|f\|_{L^4}^4}{\|\nabla f\|_{L^2}^2}. 
\end{align*}

\end{calc} 
\end{proof}

\section{Convergence of Gaussians}
\label{section:proof_white_noise_conv}

In this section we prove the convergence of Gaussians mentioned in Section \ref{section:white_noise} and Section \ref{section:eigenvalues_on_boxes}. 
We bundle the proofs together in a general setting as they rely on similar techniques.

For $r\ge 1$ we let 
$X_{k,r}^\epsilon$ and $Y_{k,r}^\epsilon$ be centered Gaussian variables for $k\in \N_0^d$, $\epsilon>0$ such that every finite subset of $\{Y_{k,r}^\epsilon:  k\in \N_0^d, \epsilon>0\} \cup \{ X_{k,r}^\epsilon : k\in \N_0^d, \epsilon>0\}$ is jointly Gaussian for all $r\ge 1$. 
We write 
\begin{align}
\xi_{r,\epsilon} = \sum_{k\in \N_0^d}  Y_{k,r}^\epsilon \fn_{k,r}, \qquad 
\theta_{r,\epsilon} = \sum_{k\in \N_0^d} X_{k,r}^\epsilon \fn_{k,r}.
\end{align}
Also, we introduce the notation
\begin{align*}
& \rho^\reso :  \R^d  \times  \R^d  \rightarrow \R, 
\qquad 
\rho^\reso(x,y) = 
\sum_{{\substack{ i,j\in \N_{-1} \\ |i-j|\le 1 }}} \rho_i(x) \rho_j(y), \\
& \Theta_{r,\epsilon} = \theta_{r,\epsilon}\reso \sigma(\rD) \theta_{r,\epsilon} - \E[ \theta_{r,\epsilon}\reso \sigma(\rD) \theta_{r,\epsilon}  ] , \\
& \Xi_{r,\epsilon} = \xi_{r,\epsilon}\reso \sigma(\rD) \xi_{r,\epsilon} - \E[ \xi_{r,\epsilon}\reso \sigma(\rD) \xi_{r,\epsilon}  ] .
\end{align*}

\begin{lemma}
\label{lemma:bound_theta_moments}
Let $d=2$. 
Write $F_{r,\epsilon}(k,l) = \E[ X_{k,r}^\epsilon X_{l,r}^\epsilon]$. 
Let $I\subset [1,\infty)$.
Suppose that 
\begin{align}
\notag 
& \forall \delta>0 \  \exists C>0\  \forall r \in I \  \forall k,l \in \N_0^2 \ \forall \epsilon >0  :\\
& 
\label{eqn:bound_F}
\qquad |F_{r,\epsilon}(k,l)| 
 \le  C \prod_{i=1}^d  (1+ |k_i - l_i | )^{\delta-1}.
\end{align}
For all $\gamma\in (0,1)$ there exists a $\fC>0$ such that for all $r \in I$, 
$i \in \N_{-1}$, $\epsilon >0$, $x\in Q_r$
\begin{align}
\label{eqn:second_moments_first_and_second_chaos}
\E[ | \Delta_i \theta_{r,\epsilon} |(x)^2]
 \le \fC r^{2\gamma}
2^{(2+\gamma )i}, 
\qquad 
\E [ |\Delta_i  \Theta_{r,\epsilon} |(x)^2] 
\le \fC r^{2\gamma} 2^{\gamma i }. 
\end{align}
\end{lemma}
\begin{proof}
This follows from Lemma \ref{lemma:expect_theta_eps} and Lemma \ref{lemma:expectation_of_THETA_eps}. 
\end{proof}

Observe that 
\begin{align}
 \label{eqn:about_difference_theta_and_xi}
&  \theta_{r,\epsilon} \reso \sigma(\rD) \theta_{r,\epsilon}
- 
\xi_{r,\epsilon} \reso \sigma(\rD) \xi_{r,\epsilon}  
 = 
\sum_{k,l\in \N_0^2} 
\frac{ \rho^\reso( \frac{k}{r},\frac{l}{r}) }{1+\frac{\pi^2}{r^2} |l|^2} 
\fn_{k,r}  \fn_{l,r} 
[X_{k,r}^\epsilon X_{l,r}^\epsilon - 
Y_{k,r}^\epsilon Y_{l,r}^\epsilon].
 \end{align}

\begin{theorem}
\label{theorem:convergence_enhanced_pair_theta}
Let $d=2$, $I\subset [1,\infty)$. 
We write $\fR = \{ (k,l)  \in \N_0^2 \times \N_0^2 : k_1 \ne l_1, k_2 \ne l_2 \}$. 
Let $G_{r,\epsilon}(k,l) = \E[X_{k,r}^\epsilon X_{l,r}^\epsilon - Y_{k,r}^\epsilon Y_{l,r}^\epsilon] $. 
Consider the following conditions. 
\begin{align}
\label{eqn:convergence_X_k_eps_to_Z_k}
& \forall k \in \N_0^2 \ \forall r \in I; \quad \E[ |X_{k,r}^\epsilon - Y_{k,r}^\epsilon|^2 ] \xrightarrow{\epsilon\downarrow 0} 0, \\
\notag &  \forall r \in I \  \forall \delta >0 \ \exists C>0  \ \exists \epsilon_0 >0 \ \forall \epsilon \in (0,\epsilon_0) \    \forall k,l \in \N_0^2 : \ 
 \\
\label{eqn:bound_G} 
& \quad 
 |G_{r,\epsilon}(k,l)|   
\le C \begin{cases}
\prod_{i=1}^2 \frac{1}{1+ |k_i - \frac{r}{\epsilon}|)^{1-\delta}} +  \frac{1}{1+ |l_i - \frac{r}{\epsilon}|)^{1-\delta}}
& (k,l) \in \fR,  \\
\sum_{i=1}^2 \frac{1}{1+ |k_i - \frac{r}{\epsilon}|)^{1-\delta}} +  \frac{1}{1+ |l_i - \frac{r}{\epsilon}|)^{1-\delta}}
& (k,l) \in \N_0^2 \times \N_0^2 \setminus \fR.
\end{cases}
\end{align}
\begin{enumerate}
\item 
\label{item:convergence_Theta_min_Xi_eps}
Suppose that \eqref{eqn:convergence_X_k_eps_to_Z_k} holds and that \eqref{eqn:bound_F} holds for $F_{r,\epsilon}(k,l)$ being either $\E[X_{k,r}^\epsilon X_{l,r}^\epsilon], \E[X_{k,r}^\epsilon Y_{l,r}^\epsilon]$ or $\E[Y_{k,r}^\epsilon Y_{l,r}^\epsilon]$. Then for $r\in I$, $\alpha<-1$, in $\fX_n^\alpha$ we have 
\begin{align*}
(
\theta_{r,\epsilon} 
 - 
 \xi_{r,\epsilon} 
,
\Theta_{r,\epsilon}
 - 
 \Xi_{r,\epsilon} 
) \xrightarrow{\P} 0. 
\end{align*}
\item 
\label{item:convergence_expectation_difference}
Suppose \eqref{eqn:bound_G} holds. Then $\E[ \theta_{r,\epsilon}\reso \sigma(\rD) \theta_{r,\epsilon}  -  \xi_{r,\epsilon} \reso \sigma(\rD)  \xi_{r,\epsilon} ] \rightarrow 0$ in $\cC_\fn^{-\gamma}$ for all $\gamma>0$ and $r\in I$. 
\end{enumerate}
Consequently, if the above assumptions in \ref{item:convergence_Theta_min_Xi_eps} and \ref{item:convergence_expectation_difference} hold, then with $c=0$, for $r\in I$, $\alpha<-1$, in $\fX_n^\alpha$
\begin{align}
\label{eqn:difference_first_and_second_chaos}
(
\theta_{r,\epsilon} 
 - 
 \xi_{r,\epsilon} 
,
\theta_{r,\epsilon}\reso \sigma(\rD) \theta_{r,\epsilon}
 - 
 \xi_{r,\epsilon} \reso \sigma(\rD)  \xi_{r,\epsilon}
) \xrightarrow{\P} (0,c). 
\end{align}
\end{theorem}
\begin{proof}
\ref{item:convergence_Theta_min_Xi_eps} 
We use Lemma \ref{lemma:gaussian_besov_bound}\ref{item:zeta_convergence_hypercontractivity}.  
By Lemma \ref{lemma:bound_theta_moments} we obtain \eqref{eqn:bound_Delta_i_zeta} for $\zeta = \theta_{r,\epsilon}-\xi_{r,\epsilon}$ with $a= 2+\gamma$ and for $\zeta=\Theta_{r,\epsilon}-\Xi_{r,\epsilon}$ with $a= 2\gamma$ for $\gamma \in (0,1)$. 
\eqref{eqn:convergence_X_k_eps_to_Z_k} implies that $\E[ |\langle \theta_{r,\epsilon} -  \xi_{r,\epsilon} , \fn_{k,r} \rangle|^2] \rightarrow 0$, i.e., \eqref{eqn:zeta_epsilon_paired_with_fn_k_conv_0} holds for $\zeta_\epsilon= \theta_{r,\epsilon}-\xi_{r,\epsilon}$. 
In Lemma \ref{lemma:difference_resos_paired_in_square_to_zero}
we show that \eqref{eqn:zeta_epsilon_paired_with_fn_k_conv_0} holds for $\zeta_\epsilon= \Theta_{r,\epsilon}-\Xi_{r,\epsilon}$. 

\ref{item:convergence_expectation_difference} is shown in Lemma \ref{lemma:convergence_difference_expectations}. 
\end{proof}

\begin{theorem}
\label{theorem:convergence_prima_and_theta}
Let $\tau \in C_c^\infty(\R^2,[0,1])$ and $\tau':\R^2 \rightarrow [0,1]$ be compactly supported functions. 
Suppose $\tau$ and $\tau'$ are equal to $1$ on a neighbourhood of $0$. 
\begin{enumerate}
\item 
\label{item:smooth_vs_table_tau}
For all $r\ge 1$ \eqref{eqn:difference_first_and_second_chaos} holds 
with $c= c_{\tau'} - c_{\tau}$
in case $X_{k,r}^\epsilon = \tau'(\frac{\epsilon}{r} k) Z_k$ and $Y_{k,r} = \tau(\frac{\epsilon}{r} k) Z_k$. 
\item 
\label{item:comparing_box_sizes}
Let $L>r \ge 1$ and $y \in \R^2$ be such that $y + Q_r \subset Q_L$.
With $\cW$ as in \ref{obs:bounded_lin_op_corresponding_to_white_noise}, 
for
\begin{align*}
& \cZ_m = \langle \cW, \fn_{m, L} \rangle, 
 \qquad
Z_k = \langle \cW, \cT_y \fn_{k,r} \rangle = \sum_{m \in \N_0^2}  \cZ_m \langle \fn_{m,L}, \cT_y \fn_{k,r} \rangle_{L^2(Q_r)}, 
\\
& 
X_{k,r}^\epsilon = \sum_{m\in \N_0^2} 
\1_{(-1,1)^2} (\tfrac{\epsilon}{L}m)  \cZ_m 
 \langle \fn_{m,L}  , \cT_y \fn_{k,r} \rangle_{L^2(Q_r)} , 
 \qquad Y_{k,r}^\epsilon = \1_{(-1,1)^2}(\tfrac{\epsilon}{r}k) Z_k. 
\end{align*}
\eqref{eqn:difference_first_and_second_chaos} holds with $c=0$. 
\end{enumerate}
\end{theorem}
\begin{proof}
\ref{item:smooth_vs_table_tau}
That \eqref{eqn:convergence_X_k_eps_to_Z_k} holds is clear. 
As $|\E[ X_{k,r}^\epsilon X_{l,r}^\epsilon]|\vee|\E[ X_{k,r}^\epsilon Y_{l,r}^\epsilon]|\vee |\E[ Y_{k,r}^\epsilon Y_{l,r}^\epsilon]|\le 2 \delta_{k,l}$, also \eqref{eqn:bound_F} holds for each of those expectations and thus the conditions of Theorem \ref{theorem:convergence_enhanced_pair_theta}\ref{item:convergence_Theta_min_Xi_eps} hold. 
Therefore it is sufficient to show that $\E[ \theta_{r,\epsilon}\reso \sigma(\rD) \theta_{r,\epsilon}  -  \xi_{r,\epsilon} \reso \sigma(\rD)  \xi_{r,\epsilon} ] \xrightarrow{\P} c_{\tau'} - c_\tau$ in $\cC_\fn^{-\gamma}$ for all $\gamma>0$. 
This follows by Lemma \ref{lemma:expectation_reso_xi_eps_C_0_plus_reso_constant} and Lemma \ref{lemma:renormalisation_constant} as they show that $\E[\theta_{r,\epsilon}\reso \sigma(\rD) \theta_{r,\epsilon}] - c_\epsilon - c_{\tau'}$ and $\E[\xi_{r,\epsilon}\reso \sigma(\rD) \xi_{r,\epsilon}] - c_\epsilon - c_{\tau}$ converge to the same limit in $\cC_\fn^{-\gamma}$. 

\ref{item:comparing_box_sizes} 
We prove this in Theorem \ref{theorem:bounds_F_and_G_hold}. 
\end{proof}

\subsection{Terms in the first Wiener chaos}
\label{subsec:first_chaos}

\begin{lemma}
\label{lemma:bound_on_expect_mod_Delta_i_difference_xi_in_x_squared}
Consider the setting of \ref{obs:overview_white_noise_section}, i.e., $Y_{k,r}^\epsilon = \tau(\frac{\epsilon}{r}k) Z_k$ for i.i.d. standard normal random variables $(Z_k)_{k\in \N_0^d}$ and $\tau \in C_c^\infty(\R^2,[0,1])$. 
For all $\gamma\in (0,1)$ there exists a $C>0$ such that for all $r\ge 1$, $i \in \N_{-1}$, $\epsilon,\delta>0$, $x\in Q_r$ 
\begin{align}
\label{eqn:expect_mod_Delta_i_difference_xi_in_x_squared}
\E[ | \Delta_i ( \xi_{r,\epsilon} - \xi_{r,\delta}) (x)|^2]
& \le C 
2^{(d+2\gamma )i} |\epsilon - \delta|^{\gamma}. 
\end{align}
\end{lemma}
\begin{proof}
Let $\gamma \in (0,1)$. 
 As 
$
\Delta_i (\xi_{r,\epsilon} - \xi_{r,\delta})(x) =  \sum_{k\in  \N_0^2} \rho_i( \frac{k}{r} ) 
( \tau(\epsilon  \frac{k}{r} ) - \tau( \delta   \frac{k}{r} )) Z_k \fn_{k,r}(x), 
$ 
and $\|\fn_{k,r}\|_\infty^2 \le (\frac2r)^d$,
\begin{calc}
\begin{align*}
\E[ | \Delta_i (\xi_{r,\epsilon} - \xi_{r,\delta}) (x)|^2 ]
 \le (\tfrac2r)^d  \sum_{k\in  \frac{1}{r}  \N_0^d} \rho_i(k)^2 
( \tau(\epsilon k) - \tau( \delta  k))^2
\end{align*}
\end{calc}
by \eqref{eqn:polynomial_bound_on_rho_i} we have 
\begin{align*}
\E[ | \Delta_i (\xi_{r,\epsilon} - \xi_{r,\delta}) (x)|^2 ]
 \lesssim  r^{-d} 
 2^{(d+2\gamma)i}  \sum_{k\in  \frac{1}{r}  \N_0^d} \frac{( \tau(\epsilon k) - \tau(\delta  k))^2}{(1+|k|)^{d+2\gamma}}.
\end{align*}
As $|\tau(\epsilon k) - \tau(\delta  k)| \le \|\nabla \tau\|_\infty |\epsilon - \delta| |k|$ and $\|\tau\|_\infty = 1$, 
\begin{align}
\label{eqn:bound_difference_tau_squared}
(\tau(\epsilon k) - \tau(\delta  k))^2 
\lesssim \|\nabla \tau\|_\infty^\gamma |\epsilon - \delta|^\gamma |k|^\gamma. 
\end{align}
Therefore, as $\sum_{k\in  \frac{1}{r}  \Z^d}  r^{-d}  \frac{|k|^\gamma }{(1+|k|)^{d+2\gamma}}  <\infty$ , we obtain \eqref{eqn:expect_mod_Delta_i_difference_xi_in_x_squared}. 
\begin{calc}
Indeed the sum is bounded by a constant not depending on $r$: 
For $k\in \Z^d$ and $x\in R^d$ with $|x- k|_\infty \le \frac12$ we have $|x-k| \le \frac{\sqrt{d}}{2}$, 
\begin{align*}
\sum_{k\in  \frac{1}{r}  \N_0^d} \frac{1}{(1+|k|)^{d+\eta}} r^{-d} 
& \le \int_{\R^d} \frac{1}{(1+ \frac{1}{\sqrt{d}} | \lfloor \frac{x}{r} \rfloor  |)^{d+\eta}} r^{-d} \dd x
\le \int_{\R^d} \frac{1}{(\frac12+ \frac{1}{\sqrt{d}} | \frac{x}{r}   |)^{d+\eta}} r^{-d} \dd x \\
& = \int_{\R^d} \frac{1}{(\frac12+ \frac{1}{\sqrt{d}} | y   |)^{d+\eta}}  \dd y  < \infty.
\end{align*}
\end{calc}
\end{proof}

\begin{lemma}
\label{lemma:expect_theta_eps}
Suppose that \eqref{eqn:bound_F} holds for $F_{r,\epsilon}(k,l) = \E[ X_{k,r}^\epsilon X_{l,r}^\epsilon]$.
For all $\gamma \in (0,1)$ there exists a $C>0$ (independent of $r$) such that for all 
$i \in \N_{-1}$, $\epsilon >0$, $x\in Q_r$ 
\begin{align}
\E[ | \Delta_i \theta_{r,\epsilon} (x)|^2]
& \le C r^{d\gamma}
2^{(d+\gamma )i} . 
\end{align}
\end{lemma}
\begin{proof}
By \eqref{eqn:polynomial_bound_on_rho_i} $ 2^{-\beta i} \|\Delta_i \fn_{k,r} \|_{L^\infty} \lesssim  r^{-\frac{d}{2}}  (1+|\frac{k}{r}|)^{-\beta} \le   r^{-\frac{d}{2}}  \prod_{i=1}^d  (\frac{1}{r}+ \frac{k_i}{r})^{-\frac{\beta}{d} }$. 
Let $\delta >0$ be such that $\delta< \gamma$ (so that in particular $\delta < \frac{1+\gamma}{2}$). As 
$|\E[ X_{k,r}^\epsilon X_{l,r}^\epsilon]|  \lesssim \prod_{i=1}^d (1+|k_i-l_i|)^{\delta-1}= \prod_{i=1}^d r^{\delta-1} (\frac{1}{r}+|\frac{k_i}{r}-\frac{l_i}{r}|)^{\delta-1}$, we have by using Lemma \ref{lemma:bound_sum_gamma_delta}
\begin{align*}
2^{-(d+d\gamma)i} \E[ | \Delta_i \theta_{r,\epsilon}(x) |_{L^\infty}^2]
\cand \begin{calc}
\notag
\lesssim r^{-d} \sum_{  k,l \in  \N_0^d} 
\frac{1}{(1+| \frac{k}{r} |)^{\frac{d+d\gamma}{2}}}
\frac{1}{(1+| \frac{l}{r} |)^{\frac{d+d\gamma}{2}}}  
|\E[ X_{k,r}^\epsilon X_{l,r}^\epsilon]| 
\end{calc}\cnewline
& \lesssim 
\Big(
 \sum_{  k,l \in \frac{1}{r} \N_0} 
\frac{1}{(\frac{1}{r}+ k )^{\frac{1+\gamma}{2}}}
\frac{1}{(\frac{1}{r}+ l )^{\frac{1+\gamma}{2}}}
 \frac{r^{\delta-2}}{ (\frac{1}{r}+|k-l |)^{1-\delta} } \Big)^d \\
& \lesssim 
\Big( 
\sum_{  k \in \frac{1}{r} \N_0} 
\frac{r^{\delta-1} }{(\frac{1}{r}+ k )^{\frac{1+\gamma}{2}}}
\frac{1}{(\frac{1}{r}+ k )^{\frac{1+\gamma}{2}-\delta}} \Big)^d
\lesssim \big( r^\delta (\tfrac{1}{r})^{\delta - \gamma} \big)^d
\lesssim r^{d\gamma}. 
\end{align*}
\end{proof}

In the following two lemmas we present tools to bound sums by integrals, which will be frequently used. 

\begin{lemma}
\label{lemma:monotone_functions_sums_and_integrals}
Let $M\in \N$ and 
 $f : [0,M] \rightarrow \R$ be a decreasing measurable function. 
 Then 
$
 \sum_{m=1}^M f(m) \le \int_0^{M} f(x) \dd x \le \sum_{m=0}^{M-1} f(m) . 
$ 
If $f$ instead is increasing, then 
$
 \sum_{m=0}^{M-1} f(m) \le \int_0^M f(x) \dd x \le \sum_{m=1}^M f(m). 
$ 
\end{lemma}

\begin{lemma}
\label{lemma:bound_sum_gamma_delta}
Let $\gamma, \delta>0$ be such that $\delta < \gamma<1$. 
There exists a $C>0$ such that for all  $r \ge 1$, $b>0$ and  $u,v  \in \R$,
\begin{align}
\label{eqn:sum_bound_with_gamma_and_delta}
\sum_{k \in \N_0} \frac{1}{r} \frac{1}{(b+| \frac{k}{r}- u |)^\gamma} \frac{1}{(b+| \frac{k}{r}-v|)^{1-\delta}} \le C (b+|u-v|)^{\delta-\gamma}. 
\end{align}
and for all $l \in \R^2$
\begin{align}
\label{eqn:bound_sum_k_with_min_gamma_power}
\sum_{k \in  \frac{1}{r} \N_0^2} 
 \frac{1}{r^2} 
\frac{ \rho^\reso(k,l)
 }{ (1+ |k-l|)^{\gamma}  } 
 \le C (1+|l|)^{ 2 -\gamma }.
\end{align}
\end{lemma}
\begin{proof}
We can bound both sums by ``their corresponding integral'' by observing the following. 
For $k\in \Z^d$ and $x\in \R^d$ with $|x-\frac{k}{r}|_\infty < \frac{1}{2r}$ and thus $|x-\frac{k}{r}|\le \frac{\sqrt{d}}{2r}$, for $u \in \R^d$ 
\begin{align}
\label{eqn:abs_difference}
\big| x - u \big| 
& \le \big| \tfrac{k}{r} - u\big| + \big| x - \tfrac{k}{r} \big| 
\le \big| \tfrac{k}{r} -u\big| + \tfrac{\sqrt{d}}{2r}.
\end{align}
So that 
\begin{align*}
\frac{1}{(b+|\frac{k}{r}-u|)^{\gamma}} 
\le 
\frac{(b+\frac{\sqrt{d}}{2r})^\gamma}{(b+\frac{\sqrt{d}}{2r}+|\frac{k}{r}-u|)^{\gamma}} 
\le \frac{ (b+\frac{\sqrt{d}}{2})^\gamma }{(b+|x-u|)^{\gamma}} . 
\end{align*}
Then
\eqref{eqn:sum_bound_with_gamma_and_delta} follows by Lemma \ref{lemma:integral_bound_with_gamma_and_theta} 
and by Lemma  \ref{lemma:resonance_sum_estimate} we have $\sum_{k \in  \frac{1}{r} \N_0^2} 
 \frac{1}{r^2} 
\frac{ \rho^\reso(k,l)
 }{ (1+ |k-l|)^{\gamma}  } 
 \lesssim 
 1+ 2\pi \int_{\frac{1}{c} |l|}^{c |l|}
\frac{x}{ (1+|x-| l || )^{\gamma}} \dd x  \lesssim (1+|l|)^{ 2 -\gamma } $.
\end{proof}

\subsection{Terms in the second Wiener chaos}
\label{subsec:second_chaos}

In order to bound terms in the second Wiener chaos, i.e., $\Xi_{r,\epsilon}$, $\Theta_{r,\epsilon}$ and $\E[\theta_{r,\epsilon}\reso \sigma(\rD) \theta_{r,\epsilon}  - \xi_{r,\epsilon} \reso \sigma(\rD) \xi_{r,\epsilon}]$, we start by presenting auxiliary lemma's and observations.

\begin{theorem}[Wick's theorem] \cite[Theorem 1.28]{Ja97}
\label{theorem:wick}
Let $A,B,C,D$ be jointly Gaussian random variables. Then 
\begin{align*}
\E[ABCD] = \E[AB]\E[CD]+\E[AC]\E[BD] + \E[AD]\E[BC]. 
\end{align*}
\end{theorem}

\begin{lemma}
\label{lemma:resonance_sum_estimate}
There exist $b >0$ and $c>1$ such that 
\begin{align*}
\supp \rho^\reso  \subset  B(0,b)^2 \cup \{ (x,y) \in\R^d \times \R^d : \tfrac{1}{c} |x| 
\le  |y| 
\le c |x|  \} 
\end{align*}
Consequently, uniformly in  $x,y\in \R^d$ 
\begin{align}
\label{eqn:getting_rid_of_one_1_plus_k_squared}
 \frac{\rho^\reso(x,y) }{(1+|x|^2)} \eqsim  \frac{\rho^\reso(x,y) }{(1+|y|^2)} .
\end{align}
\end{lemma}
\begin{proof}
Let $0<a<b$ be such that $\supp \rho_0 \subset \{ x \in \R^d : a\le |x| \le b\}$ and $\supp \rho_{-1} \subset B(0,b)$. 
Let $i, j \in \N_{-1}$ and $x,y \in \R^2$ be such that $\rho_i(x) \rho_j(y) \ne 0$. 
If $i,j \in \{-1,0\}$, then $x,y \in B(0,b)$. 
Suppose $i,j \ge 0$ and $|i-j| \le 1$. 
Then $|x| \in [2^i a, 2^i b]$ and $|y| \in [2^j a, 2^j b] \subset [2^{i-1} a, 2^{i+1} b]$. 
This in turn implies 
\begin{align*}
\tfrac{a}{2b} |x| 
\le 
\tfrac{a}{2b} 2^i b = 2^{i-1} a \le  |y| 
\le 2^{i+1} b \le 
\tfrac{2b}{a} 2^i a \le \tfrac{2b}{a} |x|. 
\end{align*}
\begin{calc}
For $x \in B(0,b)$ one has $1+|x|^2 \le 1+b^2 \le (1+b^2) ( 1+|y|^2)$ and for $(x,y) \in\R^d \times \R^d$ with $\tfrac{1}{c} |x| 
\le  |y| 
\le c |x|  $ one has $1+|x|^2 \lesssim 1+ c^2 |y|^2 \le c^2 (1+|x|^2)$. 
\end{calc}
\end{proof}

\begin{obs}
\label{obs:inner_prod_of_prod_f_k_f_l_with_f_z}
Let $k,l,z\in \N_0^d$. We write $\fn_k = \fn_{k,r}$ here. 
By \eqref{eqn:basis_prod_even_even} (and using \eqref{eqn:pairing_on_box_compared_with_extensions_on_torus}) and as $\fn_{\fq \circ k} = \fn_k $ for all $\fq \in \{-1,1\}^d$, 
\begin{align}
\langle \fn_k \fn_l , \fn_z \rangle_{L^2(Q_r)}
\cand 
\begin{calc} \notag 
 = 2^{-d} \langle \overline \fn_k \overline \fn_l , \overline \fn_z \rangle_{L^2(Q_r)}
 = 2^{-d} (2r)^{-\frac{d}{2}} \sum_{\fp \in \{-1,1\}^d} \frac{\nu_k \nu_l}{\nu_{k+\fp \circ l}} \langle \overline \fn_{ k + \fp \circ l}, \overline \fn_z \rangle_{L^2(Q_r)}
 \end{calc} \cnewline
\notag & = (2r)^{-\frac{d}{2}} \sum_{\fp \in \{-1,1\}^d} \frac{\nu_k \nu_l}{\nu_{k+\fp \circ l}} \langle \fn_{ k + \fp \circ l}, \fn_z \rangle_{L^2(Q_r)} \\
\label{eqn:innerprod_f_k_prod_f_l_with_f_z}
&  = (2r)^{-\frac{d}{2}}  \sum_{\fp,\fq \in \{-1,1\}^d} \frac{\nu_k \nu_l}{\nu_{k+\fq \circ \fp \circ l}} \delta_{\fq \circ k + \fp \circ l , z}. 
\end{align} 
By combining this with \eqref{eqn:polynomial_bound_on_rho_i}, using that $|\fn_k(x)| \le (\frac{2}{r})^{-\frac{d}{2}}$, we have 
for $x \in (0,r)^d$ and $\gamma>0$
\begin{align}
\label{eqn:Delta_q_f_k_f_l}
 r^d  |\Delta_i ( \fn_k \fn_l)(x)| 
\lesssim  \sum_{\fp,\fq \in \{-1,1\}^d} \rho_i \Big( \frac{\fq\circ k + \fp \circ l}{r}  \Big) 
 \frac{2^{\gamma i} }{ (1+ | \frac{k}{r}-\frac{l}{r} |)^\gamma }.
\end{align}
\end{obs}

\begin{lemma}
\label{lemma:bound_on_expect_mod_Delta_i_difference_Xi_in_x_squared}
Let $d=2$. 
Consider the setting of \ref{obs:overview_white_noise_section} as we did in Lemma \ref{lemma:bound_on_expect_mod_Delta_i_difference_xi_in_x_squared}.
For all $\gamma \in (0,1)$ there exists a $C>0$ (independent of $r$) such that for all  $i\in \N_{-1}$, $\epsilon,\delta>0$, $x\in Q_r$
\begin{align}
\label{eqn:bound_on_expect_mod_Delta_i_difference_Xi_in_x_squared}
\E [ |\Delta_i ( \Xi_{r,\epsilon} - \Xi_{r,\delta})(x)|^2] 
&\le C   |\epsilon - \delta|^{\gamma}  2^{2\gamma i }.
\end{align}
\end{lemma}
\begin{proof}
First observe 
 $\Xi_{r,\epsilon}  =  \sum_{k,l\in \N_0^2} \rho^\reso( \frac{k}{r},\frac{l}{r}) \frac{\tau(\epsilon \frac{k}{r}) \tau(\epsilon \frac{l}{r})}{1+\frac{\pi^2}{r^2} |l|^2} [Z_k Z_l -\delta_{k,l}] \fn_k \fn_l $. 
By Theorem \ref{theorem:wick} and \eqref{eqn:Delta_q_f_k_f_l} (as both contributions $\delta_{k,m} \delta_{l,n}$ and $\delta_{k,n}\delta_{m,l}$ can be bounded by the same expression by Lemma \ref{lemma:resonance_sum_estimate})
\begin{align*}
&  2^{-2\gamma i} \E [ |\Delta_i ( \Xi_{r,\epsilon} - \Xi_{r,\delta})(x)|] \\
& \lesssim 
\sum_{k,l \in  \frac{1}{r}  \N_0^2} 
 \frac{1}{r^4} 
\frac{\rho^\reso(k,l)^2 }{(1+ \pi^2|l|^2)^2}  
\frac{ [\tau(\epsilon k) \tau(\epsilon l)-\tau(\delta k) \tau(\delta l)]^2 }{ (1+ |k-l|)^{2\gamma} }. 
\end{align*}
As $2(ab-cd) = (a-c) (b+d) + (a+c)(b-d)$
\begin{calc}
\begin{align*}
(a-c) (b+d) + (a+c)(b-d)
= ab +ad-cb-cd + ab -ad + cb-cd
\end{align*}
we use this as follows:
\begin{align*}
& 2[\tau(\epsilon k) \tau (\epsilon l)  - 
\tau(\delta k) \tau (\delta l) ]\\
& = 
\big(\tau(\epsilon k) - \tau(\delta k)\big)
\big(\tau (\epsilon l)  +  \tau (\delta l) \big)
+\big(\tau(\epsilon k) + \tau(\delta k)\big)
\big(\tau (\epsilon l)  -  \tau (\delta l) \big),
\end{align*}
\end{calc}
similar to \eqref{eqn:bound_difference_tau_squared} as in the proof of Lemma \ref{lemma:bound_on_expect_mod_Delta_i_difference_xi_in_x_squared} we obtain 
\begin{align*}
|\tau(\epsilon k) \tau (\epsilon l)  - 
\tau(\delta k) \tau (\delta l) |^2
\le 4 \|\nabla \tau\|_\infty^{\gamma}
|\epsilon - \delta|^{\gamma} (|k|^{\gamma} + |l|^{\gamma}). 
\end{align*}
Using Lemma \ref{lemma:resonance_sum_estimate} and \eqref{eqn:bound_sum_k_with_min_gamma_power} we obtain
\begin{align*}
 2^{-2\gamma i} \E [ \|\Delta_i ( \Xi_{r,\epsilon} - \Xi_{r,\delta})\|_{L^\infty}^2] 
& \lesssim 
|\epsilon - \delta|^{\gamma}
\sum_{l \in  \frac{1}{r}  \N_0^2} 
\frac{  r^{-4}  }{(1+ |l|)^{4-\gamma} }  
\sum_{k \in  \frac{1}{r}  \N_0^2} 
\frac{ \rho^\reso(k,l)^2 }{ (1+ |k-l|)^{2\gamma} }
\\
& \lesssim 
|\epsilon - \delta|^{\gamma}
\sum_{l \in  \frac{1}{r}  \N_0^2} 
\frac{  r^{-2}  }{(1+ |l|)^{2 + \gamma} }  . 
\end{align*}
\end{proof}

\begin{lemma}
\label{lemma:expectation_of_THETA_eps}
Suppose that \eqref{eqn:bound_F} holds for $F_{r,\epsilon}(k,l) = \E[ X_{k,r}^\epsilon X_{l,r}^\epsilon]$.
For all $\gamma \in (0,\infty)$ there exists a $C>0$ (independent of $r$)  such that for all  $i\in \N_{-1}$, $\epsilon>0$ 
\begin{align}
\label{eqn:bound_on_expect_mod_Delta_i_Theta}
\E [ |\Delta_i  \Theta_{r,\epsilon} (x)|^2] 
\le C r^{2\gamma} 2^{\gamma i }. 
\end{align}
\end{lemma}
\begin{proof}
First note that $\Theta_{r,\epsilon} = \sum_{k,l\in \N_0^2} 
 \frac{ \rho^\reso(\frac{k}{r},\frac{l}{r}) }{1+ \pi^2 |\frac{l}{r}|^2} 
\fn_{k,r}  \fn_{l,r} 
[X_{k,r}^\epsilon X_{l,r}^\epsilon - \E[ X_{k,r}^\epsilon X_{l,r}^\epsilon] ] $. 
By Theorem \ref{theorem:wick}
\begin{align*}
&  \E \left(
[X_{k,r}^\epsilon X_{l,r}^\epsilon - \E[ X_{k,r}^\epsilon X_{l,r}^\epsilon] ]
[X_{m,r}^\epsilon X_{n,r}^\epsilon - \E[ X_{m,r}^\epsilon X_{n,r}^\epsilon] ] 
  \right)  \\
&   = \E[ X_{k,r}^\epsilon X_{m,r}^\epsilon]\E[ X_{l,r}^\epsilon X_{n,r}^\epsilon] 
  + \E[ X_{k,r}^\epsilon X_{n,r}^\epsilon]\E[ X_{l,r}^\epsilon X_{m,r}^\epsilon]. 
\end{align*}
By exploiting symmetries using Lemma \ref{lemma:resonance_sum_estimate} and by \eqref{eqn:Delta_q_f_k_f_l} we have 
\begin{align*}
2^{-2\gamma i}  \E [ |\Delta_i  \Theta_{r,\epsilon} (x)|^2] 
\lesssim
%
 \sum_{k,l,m,n\in  \frac{1}{r} \N_0^2} 
\frac{ r^{-4} \rho^\reso(k,l) \rho^\reso(m,n) 
|\E[ X_{rk}^\epsilon X_{rm}^\epsilon]\E[ X_{rl}^\epsilon X_{rn}^\epsilon] |
}{(1+|k-l|)^\gamma (1+|m-n|)^\gamma(1+|l|^2)(1+|m|^2)}. 
\end{align*}
We will bound the $\rho^\reso$ function by $1$, 
use the bound \eqref{eqn:bound_F} for some $\delta>0$ (will be chosen small enough later) and we `separate the dimensions' by using that $1+|k|^2 \gtrsim (1+k_1) (1+k_2)$ and $(1+ |k-l|)^\gamma \gtrsim (\frac{1}{r}+ |k_1-l_1|)^{\frac{\gamma}{2}} (\frac{1}{r}+ |k_2-l_2|)^{\frac{\gamma}{2}} $ and  obtain 
\begin{align}
2^{-2\gamma i}  \E [ |\Delta_i  \Theta_{r,\epsilon} (x)|^2] 
\lesssim 
\bigg( 
\sum_{k,l,m,n\in \frac1r \N_0} 
\frac{r^{2\delta-4} (\frac{1}{r}+|k-m|)^{\delta-1 } (\frac{1}{r}+|l-n|)^{\delta-1 }
}{(\frac{1}{r}+|k-l|)^{\frac{\gamma}{2}} (\frac{1}{r}+|m-n|)^{\frac{\gamma}{2}} (\frac{1}{r}+l)(\frac{1}{r}+m)} 
\bigg)^2. 
 \label{eqn:bound_of_second_moment_Theta_by_square_4_dim_sum}
\end{align}
For $\delta < \frac{\gamma}{2}$ we have by Lemma \ref{lemma:bound_sum_gamma_delta}
\begin{align*}
\sum_{n\in \frac{1}{r}\N_0} 
\frac{r^{-1}(\frac{1}{r}+|l-n|)^{\delta-1 }}{(\frac{1}{r}+|m-n|)^{\frac{\gamma}{2}} } \vee 
\sum_{k\in \frac{1}{r} \N_0} 
\frac{r^{-1}(\frac{1}{r}+|k-m|)^{\delta-1 }}{(\frac{1}{r}+|k-l|)^{\frac{\gamma}{2}} } 
\lesssim \frac{1}{(\frac{1}{r}+|m-l|)^{\frac{\gamma}{2}-\delta} }, 
\end{align*}
and for $\delta < \frac{\gamma}{4}$ the square root of the right-hand side of \eqref{eqn:bound_of_second_moment_Theta_by_square_4_dim_sum} can be bounded by
\begin{align*}
\sum_{m,l \in \frac{1}{r} \N_0} \frac{r^{2\delta-2}}{(\frac{1}{r}+|m-l|)^{\gamma-2\delta} } \frac{1}{\frac{1}{r}+m} 
\frac{1}{\frac{1}{r}+l} 
\lesssim 
\sum_{l \in \frac{1}{r} \N_0}
\frac{r^{2\delta-1}}{(\frac{1}{r}+l)^{\gamma-3\delta} } 
\frac{1}{\frac{1}{r}+l} 
\lesssim r^{\gamma - 2\delta} \lesssim r^\gamma. 
\end{align*}
Hence we obtain \eqref{eqn:bound_on_expect_mod_Delta_i_Theta}.
\end{proof}

\begin{lemma}
\label{lemma:difference_resos_paired_in_square_to_zero}
Suppose that \eqref{eqn:convergence_X_k_eps_to_Z_k} holds and that \eqref{eqn:bound_F} holds for $F_{r,\epsilon}(k,l)$ being either 
 $\E[X_{k,r}^\epsilon X_{l,r}^\epsilon]$, $ \E[X_{k,r}^\epsilon Y_{l,r}^\epsilon]$ or $\E[Y_{k,r}^\epsilon Y_{l,r}^\epsilon]$. 
 Then 
$\E [ | \langle \Theta_{r,\epsilon} - \Xi_{r,\epsilon} , 
\fn_z \rangle |^2 ] \rightarrow 0$
for all $z\in \N_0^2$. 
\end{lemma}
\begin{proof}
Fix $z\in \N_0^2$. 
Given a function $H: (\N_0^2)^4 \rightarrow \R$ let us use the following (formal) notation 
\begin{align*}
\fS(H) = \sum_{k,l,m,n\in \N_0^2} 
\frac{ \rho^\reso(k,l) }{1+\frac{\pi^2}{r^2} |l|^2} 
\frac{ \rho^\reso(m,n) }{1+\frac{\pi^2}{r^2} |n|^2} 
\sum_{\fp,\fr , \fq , \fs\in \{-1,1\}^2} \delta_{\fr \circ k + \fp \circ l , z}
 \delta_{\fs \circ m + \fq \circ n , z} H(k,l,m,n).
\end{align*}
By \eqref{eqn:innerprod_f_k_prod_f_l_with_f_z}, as  $\frac14 \le \nu_k \le 1$ for all $k\in\N_0^2$,
\begin{align*}
 & \E [ | \langle 
\Theta_{r,\epsilon} 
 - 
\Xi_{r,\epsilon} , 
\fn_z \rangle |^2 ] 
\lesssim 
\fS(E_\epsilon),
\end{align*}
where 
\begin{align*}
E_\epsilon(k,l,m,n) & = 
  \E \left(
[X_{k,r}^\epsilon X_{l,r}^\epsilon - Y_{k,r}^\epsilon Y_{l,r}^\epsilon ]
[X_{m,r}^\epsilon X_{n,r}^\epsilon - Y_{m,r}^\epsilon Y_{n,r}^\epsilon ] 
  \right) \\
& \quad   - 
    \E 
[X_{k,r}^\epsilon X_{l,r}^\epsilon -  Y_{k,r}^\epsilon Y_{l,r}^\epsilon  ]
\E
[X_{m,r}^\epsilon X_{n,r}^\epsilon - Y_{m,r}^\epsilon Y_{n,r}^\epsilon ] . 
\end{align*}  
We decompose $E_\epsilon$ using Wick's theorem (Theorem \ref{theorem:wick}). Let us for a few lines write $A_k = X_{k,r}^\epsilon$ and $B_k =Y_{k,r}^\epsilon$, then we obtain 
\begin{align*}
& \E \left( [A_k A_l - B_k B_l ][ A_m A_n - B_m B_n ] \right)
- \E  [A_k A_l - B_k B_l ] \E[ A_m A_n - B_m B_n ]  \\
& = \E [A_k A_l  A_m A_n ] 
- \E [A_k A_l  B_m B_n ] 
- \E [ B_k B_l  A_m A_n ] 
+ \E [B_k B_l  B_m B_n ] \\
& \quad - (\E  [A_k A_l] - \E[ B_k B_l ])(\E[ A_m A_n] -\E[ B_m B_n ] ) \\
& = 
 \E [A_k A_m] \E[  A_l A_n ] 
- \E [A_k B_m ] \E[ A_l B_n ] 
- \E [ B_k A_m] \E[ B_l A_n ] 
+ \E [B_k B_m ] \E[B_l B_n ] \\
& \quad +  \E [A_k A_n]\E[  A_m A_l ] 
- \E [A_k B_n ] \E[ B_m A_l ] 
- \E [ B_k A_n] \E[ A_m B_l ] 
+ \E [B_k B_n ] \E[B_m B_l ] .
\end{align*}
Observe that
\begin{align*}
&  \E [A_k A_m] \E[  A_l A_n ] - \E [A_k B_m ] \E[ A_l B_n ] \\
\cand \begin{calc}
=  \E [A_k A_m] \E[  A_l A_n ] - 
\E [A_k (B_m-A_m) ] \E[ A_l B_n ] 
-\E [A_k A_m ] \E[ A_l B_n ] 
\end{calc}\cnewline
\cand \begin{calc}
=  \E [A_k A_m] \E[  A_l A_n ] 
- \E [A_k (B_m-A_m) ] \E[ A_l (B_n -A_n)] 
\end{calc}\cnewline
\cand \begin{calc}
\quad -\E [A_k (B_m-A_m) ] \E[ A_l A_n ] 
-\E [A_k A_m ] \E[ A_l B_n ] 
\end{calc}\cnewline
&
=  \E [A_k A_m] \E[  A_l (A_n-B_n) ] 
- \E [A_k (B_m-A_m) ] \E[ A_l (B_n -A_n)] 
\\ &
\quad -\E [A_k (B_m-A_m) ] \E[ A_l A_n ].
\end{align*}
Hence, as $\E[|A_k-B_k|^2] = \E[|X_{k,r}^\epsilon- Y_{k,r}^\epsilon|^2] \rightarrow 0$ by \eqref{eqn:convergence_X_k_eps_to_Z_k}, 
we have $E_\epsilon(k,l,m,n) \rightarrow 0$ for all $k,l,m,n\in \N_0^2$.
We show that $\fS(E_\epsilon)$ converges to zero by a dominated convergence argument. 
Let us write 
\begin{align*}
J(k,l,m,n) := \prod_{i=1}^2 (1+|k_i - m_i|)^{\delta-1} (1+|l_i - n_i|)^{\delta-1}, 
\end{align*}
and $\tilde J(k,l,m,n) = J(k,l,n,m)$.
Then by \eqref{eqn:bound_F} we have $E_\epsilon \le J + \tilde J$ and by the symmetries obtained by Lemma \ref{lemma:resonance_sum_estimate} $\fS(\tilde J) \lesssim \fS(J)$. 
Moreover, by ``merging the $\fp,\fq,\fr,\fs$ and $k,l,m,n$ variables'' (in the sense of summing over $k \in \Z^2$ instead of $\fq \circ k$ with $\fq \in \{-1,1\}^2$ and $k\in \N_0^2$) we have 
\begin{align*}
\fS(J) 
\cand \begin{calc}
\lesssim \sum_{k,l,m,n\in \Z^2} 
\frac{ \delta_{ k + l , z} }{1+\frac{\pi^2}{r^2} |l|^2} 
\frac{  \delta_{ m + n , z} }{1+\frac{\pi^2}{r^2} |n|^2} 
  \prod_{i=1}^2 (1+|k_i - m_i|)^{\delta-1} (1+|l_i - n_i|)^{\delta-1}
\end{calc} \cnewline
& \lesssim \sum_{l,n\in \Z^2} 
\frac{ 1 }{1+\frac{\pi^2}{r^2} |l|^2} 
\frac{ 1 }{1+\frac{\pi^2}{r^2} |n|^2} 
  \prod_{i=1}^2 (1+|l_i - n_i|)^{2\delta-2} 
\end{align*}
which is finite by Lemma \ref{lemma:bound_sum_gamma_delta}. 
\end{proof}

\begin{lemma}
\label{lemma:convergence_difference_expectations}
If \eqref{eqn:bound_G} holds, then 
$\E[ \theta_{r,\epsilon}\reso \sigma(\rD) \theta_{r,\epsilon}
 - 
\xi_{r,\epsilon} \reso \sigma(\rD) \xi_{r,\epsilon} ] \rightarrow 0$ in $\cC_\fn^{-\gamma}$ for all $\gamma>0$. 
\end{lemma}
\begin{proof}
Let us abbreviate $G_{r,\epsilon}(k,l) = \E[X_{k,r}^\epsilon X_{l,r}^\epsilon - Y_{k,r}^\epsilon Y_{l,r}^\epsilon]$. 
By \eqref{eqn:about_difference_theta_and_xi} and \eqref{eqn:Delta_q_f_k_f_l}
\begin{align*}
& \sup_{i\in \N_{-1}} 2^{-\gamma i} \| \Delta_i \E[ \theta_{r,\epsilon}\reso \sigma(\rD) \theta_{r,\epsilon}
 - 
\xi_{r,\epsilon} \reso \sigma(\rD) \xi_{r,\epsilon} ] \|_\infty 
 \lesssim 
\sum_{k,l \in \N_0^2} 
\frac{\rho^\reso(k,l)}{(1+ |l|^2)}  
\frac{|G_{r,\epsilon}(k,l)|}{ (1+ |k-l|)^\gamma  }. 
\end{align*}
We use \eqref{eqn:bound_G} and consider the sums over $\fR$ and $\N_0^2 \times \N_0^2 \setminus \fR$ as in Theorem \ref{theorem:convergence_enhanced_pair_theta} separately. 

$\bullet$ \textbf{[Sum over $\fR$]} 
By exploiting symmetries using Lemma \ref{lemma:resonance_sum_estimate} 
\begin{align*}
& \sum_{(k,l) \in \fR} 
\frac{\rho^\reso(k,l)}{(1+ |l|^2)}  
\frac{|G_{r,\epsilon}(k,l)|}{ (1+ |k-l|)^\gamma  } 
 \lesssim \cS_{\epsilon,1} + \cS_{\epsilon,2}, \\
& \cS_{\epsilon,1}
 = \sum_{k,l \in \N_0^2} 
\frac{\rho^\reso(k,l)}{(1+ |l|^2)}  
\frac{1}{ (1+ |k-l|)^\gamma } 
\frac{1}{(1+|\frac{r}{\epsilon}-l_1 |)^{1-\delta}} \frac{1}{(1+|\frac{r}{\epsilon}-l_2 |)^{1-\delta}} \\
& \cS_{\epsilon,2}
 = \sum_{k,l \in \N_0^2} 
\frac{\rho^\reso(k,l)}{(1+ |l|^2)}  
\frac{1}{ (1+ |k-l|)^\gamma } 
\frac{1}{(1+|\frac{r}{\epsilon}-l_1 |)^{1-\delta}} \frac{1}{(1+|\frac{r}{\epsilon}-k_2 |)^{1-\delta}}.
\end{align*}
By \eqref{eqn:bound_sum_k_with_min_gamma_power}, by using that $(1+|l|^2)\ge (1+l_1)  (1+l_2)$ and by using \eqref{eqn:sum_bound_with_gamma_and_delta} with $\delta < \gamma$
\begin{align*}
 \cS_{\epsilon,1}
& \lesssim 
\left( \sum_{l \in \N_0} 
\frac{1}{(1+ l)^{ \frac{\gamma}{2} } }  
\frac{1}{(1+|\frac{r}{\epsilon}-l |)^{1-\delta}} \right)^2
\lesssim (1+\tfrac{r}{\epsilon})^{\delta -\gamma} \lesssim \epsilon^{\gamma -\delta }.
\end{align*}
For $\cS_{\epsilon,2}$ by Lemma \ref{lemma:resonance_sum_estimate} there exist $b>0,c>1$ such that (using that  $|k-l| \ge |k_1-l_1|$)
\begin{align*}
& \sum_{k \in \N_0^2} 
\frac{\rho^\reso(k,l)}{ (1+ |k-l|)^\gamma } 
 \frac{1}{(1+|\frac{r}{\epsilon}-k_2 |)^{1-\delta}} \\
& \lesssim 
\sum_{ \substack{k \in \N_0^2 \\  |k| \le b } } 
 \frac{1}{(1+|\frac{r}{\epsilon}-k_2 |)^{1-\delta}} 
  +  
  \sum_{ \substack{k_1 \in \N_0 \\ k_1 \le c |l|} } 
\frac{1}{ (1+ | k_1- l_1 | )^\gamma } 
  \sum_{ \substack{k_2 \in \N_0 \\ k_2 \le c |l|} } 
 \frac{1}{(1+|\frac{r}{\epsilon}-k_2 |)^{1-\delta}}.
\end{align*}
We will bound the second sum on the right-hand side by its corresponding integrals (see Lemma \ref{lemma:monotone_functions_sums_and_integrals})
and will bound these to get a bound on the sum over $k$. 
Straightforward calculations show 
\begin{align*}
\int_{0}^{c |l|} \frac{1}{ (1+ | x - l_1 |)^\gamma } \dd x 
\lesssim (1+|l|)^{1-\gamma}.
\end{align*}
\begin{calc}
Indeed 
\begin{align*}
\int_{l_1}^{c |l|} \frac{1}{ (1+ | x -l_1 |)^\gamma } \dd x 
= \frac{1}{1-\gamma} (1+x- l_1)^{1-\gamma} |_{ l_1}^{c |l|}
\lesssim (1+|l|)^{1-\gamma}, \\
\int_0^{l_1} \frac{1}{ (1+ | x - l_1 |)^\gamma } \dd x 
= \frac{-1}{1-\gamma} (1+ l_1 -x)^{1-\gamma} |_{ 0}^{ l_1}
\lesssim (1+|l|)^{1-\gamma}, \\
\end{align*}
\end{calc}
On the other hand, for $\delta>0$ and $z>0$
\begin{align*}
  \int_0^{z}  \frac{1}{1+|\frac{r}{\epsilon}-x |} \dd x 
  & \lesssim  \log (1+\tfrac{r}{\epsilon})^2 (1+z) 
  \lesssim (1+\tfrac{r}{\epsilon})^{2\delta} (1+z)^{\delta}.
\end{align*}
\begin{calc} 
Indeed, 
\begin{align*}
  \int_0^{z}  \frac{1}{1+|\frac{r}{\epsilon}-x |} \dd x = - \log (1+ \tfrac{r}{\epsilon}-x)|_0^z \le \log(1+\tfrac{r}{\epsilon}) \qquad \mbox{ for } z\le \tfrac{r}{\epsilon}, \\
    \int_{\frac{r}{\epsilon}}^{z}  \frac{1}{1+|\frac{r}{\epsilon}-x |} \dd x = \log (1+x-\tfrac{r}{\epsilon})|_{\frac{r}{\epsilon}}^{z} = \log (1+z -\tfrac{r}{\epsilon}) \qquad \mbox{ for } z \ge \tfrac{r}{\epsilon}.
\end{align*}
Hence for all $z\ge 0$
\begin{align*}
  \int_0^{z}  \frac{1}{1+|\frac{r}{\epsilon}-x |} \dd x 
&   \lesssim \log(1+\tfrac{r}{\epsilon}) +  \log (1+|z -\tfrac{r}{\epsilon}|) \\
& =  \log(1+\tfrac{r}{\epsilon})(1+|z -\tfrac{r}{\epsilon}|) 
   \lesssim  \log (1+\tfrac{r}{\epsilon})^2 (1+z) .
\end{align*}
\end{calc}
Hence for all $\delta>0$ (we use \eqref{eqn:sum_bound_with_gamma_and_delta} for the last inequality) 
\begin{align*}
\cS_{\epsilon,2}
& \lesssim 
\sum_{l \in \N_0^2} 
\frac{1}{(1+ |l|^2)}  
\frac{1}{(1+|\frac{r}{\epsilon}-l_1 |)^{1-\delta}} 
 (1+l_1+l_2)^{1-\gamma+\delta} (1+\tfrac{1}{\epsilon})^{2\delta }
 \\
 &\lesssim 
\sum_{l \in \N_0^2} 
\frac{1}{(1+ l_1 + l_2)^{1+\gamma-\delta}}  
\frac{1}{(1+|\frac{r}{\epsilon}-l_1 |)^{1-\delta}} 
 (1+\tfrac{1}{\epsilon})^{2\delta }\\
  &\lesssim 
\sum_{l_2 \in \N_0} 
\frac{1}{(1 + l_2)^{1+\delta}}  
\sum_{l_1 \in \N_0} 
\frac{1}{(1+ l_1 )^{\gamma-2\delta}}  
\frac{1}{(1+|\frac{r}{\epsilon}-l_1 |)^{1-\delta}} 
 (1+\tfrac{1}{\epsilon})^{2\delta }
 \lesssim 
 (1+\tfrac{1}{\epsilon})^{5\delta -\gamma}.
\end{align*}
Therefore, by choosing $\delta < \frac{\gamma}{5}$ we obtain also $\cS_{\epsilon,2} \rightarrow 0$. 

$\bullet$ \textbf{[Sum over $\N_0^2\times\N_0^2 \setminus \fR$]} 
Observe that $\N_0^2\times\N_0^2 \setminus \fR = \{ (k,l) \in \N_0^2 \times \N_0^2 : \exists i \in \{1,2\} : k_i = l_i\}$. 
Therefore, again by exploiting symmetries using Lemma \ref{lemma:resonance_sum_estimate} (we bound the sum over $\N_0^2\times \N_0^2 \setminus \fR$ by the sum over all $l\in\N_0^2$, $k_2\in\N_0$ and take $k_1=l_1$), using \eqref{eqn:sum_bound_with_gamma_and_delta} for $\delta < \frac{\gamma}{2}$
\begin{align*}
& \sum_{(k,l) \in \N_0^2 \times \N_0^2 \setminus \fR } 
\frac{\rho^\reso(k,l)}{(1+ |l|^2)}  
\frac{|G_{r,\epsilon}(k,l)|}{ (1+ |k-l|)^\gamma  } 
 \lesssim 
  \sum_{l \in \N_0^2} 
\frac{1}{(1+ |l|^2)}  
\sum_{k_2\in\N_0}
\frac{1}{ (1+ |k_2-l_2|)^\gamma } 
\frac{1}{1+|\frac{r}{\epsilon}-k_2 |} \\
& \lesssim 
  \sum_{l \in \N_0^2} 
\frac{1}{(1+ l_1)^{1+\delta} }  
\frac{1}{(1+ l_2)^{1-\delta} }  
\frac{1}{(1+|\frac{r}{\epsilon}-l_2 |)^{\gamma-\delta}} 
\lesssim (1+\tfrac{1}{\epsilon})^{2\delta-\gamma}. 
\end{align*} 
\end{proof}

\section{Proof of Theorem \ref{theorem:convergence_prima_and_theta}\ref{item:comparing_box_sizes}}
\label{section:proof_convergence_diff_box_sizes}

In this section we consider $d=2$, $L > r \ge 1$ and $y \in \R^2$ such that $y + Q_r \subset Q_L$. 
We write $\tau = \1_{(-1,1)^2}$. We consider $X_{k,r}^\epsilon$ and $Y_{k,r}^\epsilon$ as in Theorem \ref{theorem:convergence_prima_and_theta}\ref{item:comparing_box_sizes}.
For $m,l \in \N_0$ and $z\in [0,L-r]$ we write 
\begin{align}
\label{eqn:b_m_l_z}
b_{m,l}^z = \langle \fn_{m,L} , \cT_z \fn_{l,r} \rangle_{L^2 ([0,r])}.
\end{align}
Then we have 
\begin{align*}
X_{k,r}^\epsilon = \sum_{m\in \N_0^2} \tau(\tfrac{\epsilon}{L}m) \cZ_m \prod_{i=1}^2 b_{m_i,k_i}^{y_i}, 
\qquad 
Y_{k,r}^\epsilon = \tau(\tfrac{\epsilon}{r}k) \sum_{m\in \N_0^2}  \cZ_m \prod_{i=1}^2 b_{m_i,k_i}^{y_i}. 
\end{align*}
And so with $G_{r,\epsilon}(k,l) = \E[X_{k,r}^\epsilon X_{l,r}^\epsilon - Y_{k,r}^\epsilon Y_{l,r}^\epsilon] $ as in Theorem \ref{theorem:convergence_enhanced_pair_theta}
\begin{align}
\label{eqn:formula_G_eps_in_terms_of_b}
G_{r,\epsilon}(k,l) = \sum_{m  \in \N_0^2} 
\Big( \prod_{i=1}^2 b_{m_i,k_i}^{y_i} b_{m_i,l_i}^{y_i} \Big)
[\tau(\tfrac{\epsilon}{L} m)^2
- \tau(\tfrac{\epsilon}{r} k ) \delta_{k,l} ] .
\end{align}

\begin{theorem}
\label{theorem:bounds_F_and_G_hold}
 \eqref{eqn:convergence_X_k_eps_to_Z_k}, \eqref{eqn:bound_F} and \eqref{eqn:bound_G} hold (for $I= [1,L]$)
\end{theorem}
\begin{proof}
For \eqref{eqn:convergence_X_k_eps_to_Z_k}, we have
\begin{align*}
\E[ |\langle \theta_{r,\epsilon} - \xi_{r,\epsilon} , \fn_k \rangle|^2]
 = \E[ | X_{k,r}^\epsilon - Y_{k,r}^\epsilon|^2] 
\lesssim \sum_{m \in \N_0^d} (\tau(\tfrac{\epsilon}{L}m) - \tau(\tfrac{\epsilon}{r}k))^2 \prod_{i=1}^2 (b_{m_i,k_i}^{y_i})^2. 
\end{align*}
By Lebesgue's dominated convergence theorem this converges to zero. 
\eqref{eqn:bound_F} follows by Theorem   \ref{theorem:bound_on_sum_b_mk_b_ml} 
by observing that $\E[X_{k,r}^\epsilon Y_{l,r}^\epsilon ] = \tau(\frac{\epsilon}{r}k) \E[X_{k,r}^\epsilon X_{l,r}^\epsilon ]$, $\E[Y_{k,r}^\epsilon Y_{l,r}^\epsilon] \le 2 \delta_{k,l}$ and that $|\E[X_{k,r}^\epsilon X_{l,r}^\epsilon]|\le \prod_{i=1}^2 (\sum_{m\in \N_0} |b_{m,k_i}^{y_i} b_{m,l_i}^{y_i}|)$. 
\eqref{eqn:bound_G} follows by Lemma \ref{lemma:bound_G_eps_with_eps_dependence}. 
\end{proof}

\begin{obs}
The estimates \eqref{eqn:bound_F} and \eqref{eqn:bound_G} will rely on bounds on $b_{m,l}^{z}$ for $m,l \in \N_0$ and $z \in [0,L-r]$. Let us calculate $b_{m,l}^z$ here. 
For notational convenience we put $\frac{\sin(\pi x)}{x}$ and $\frac{1-\cos(\pi x)}{x}$ for $x=0$ equal to $1$ here. 
By using some trigonometric rules, one can compute that
\begin{align}
\label{eqn:pairing_n_m_l_and_shift_n_l_r}
b_{m,l}^z 
\begin{calc}
= \langle \fn_{m,L} , \cT_z \fn_{l,r} \rangle_{L^2 ([0,r])}
\end{calc}
= \sqrt{\tfrac{r}{L}}\tfrac{1}{\pi} \nu_m \nu_l [ f_{m,l}  \cos( \tfrac{\pi}{L} m z )  + g_{m,l}
  \sin( \tfrac{\pi}{L} m z ) ],
\end{align}
where 
\begin{align*}
f_{m,l} = \sum_{\fp\in \{-1,1\}} \frac{\sin( \pi  (\frac{r}{L}m + \fp l) )}{ \frac{r}{L}m +\fp l } , \qquad 
g_{m,l} = \sum_{\fp\in \{-1,1\}} \frac{1-\cos( \pi (\frac{r}{L}m + \fp l) )}{  \frac{r}{L}m + \fp l} .
\end{align*}
Let us demonstrate \eqref{eqn:pairing_n_m_l_and_shift_n_l_r} in the easier case $z=0$. 
Due to the identities $2 \cos(a) \cos(b) = \sum_{\fp \in \{-1,1\}} \cos(a+ \fp b) $  and 
$\sin(\pi (a\pm l))= (-1)^l \sin (\pi a)$ for $a,b\in \R$ and $ l\in \Z$, 
we obtain
\begin{align}
\langle \fn_{m, L} , \fn_{l,r} \rangle_{L^2([0,r])} 
\notag 
& = \frac{2}{\sqrt{Lr} } \nu_m \nu_l 
 \int_0^r \cos( \tfrac{\pi}{L} m x) \cos( \tfrac{\pi}{r} l x) \dd x \\
\label{eqn:b_expression}
&  = \sqrt{\frac{r}{L}} \frac{1}{\pi}  \nu_m \nu_l 
\sum_{\fp \in \{-1,1\}}    \frac{\sin( \pi ( \frac{m r}{L} +\fp l ))}{  \frac{m r}{L}  +\fp l } . 
\end{align}
\begin{calc}
\begin{align*}
\frac{1}{r} \int_0^r 2 \cos( a x) \cos( b x) \dd x
& =  \frac1r \int_0^r \cos( (a- b) x) +\cos( (a+b) x) \dd x \\
& =  \frac1r \left[ \frac{\sin((a-b)x)}{a-b} + \frac{\sin((a+b)x)}{a+b} \right]_0^r \\
& =  \left[ \frac{\sin((a-b)r)}{(a-b)r} + \frac{\sin((a+b)r)}{(a+b)r} \right] .
\end{align*}
\end{calc}
\begin{calc}
For general $z\in [0,L-r]$. First observe that 
\begin{align*}
 \langle \fn_{m, L} , \cT_{z} \fn_{l,r} \rangle_{L^2(\R)} 
& = \int_{z+[0,r]\cap [0,L]} \fn_{m,L}(x) \fn_{l,r}(x-z) \dd x \\
& =   \int_{0}^r \fn_{m,L}(x+z) \fn_{l,r}(x) \dd x . 
\end{align*}
Using the trigonometric identities for $\cos(a \pm b)$ and $\sin (a \pm b)$ 
we have 
\begin{align*}
& 2 \cos (a+d) \cos(b) 
 = \cos (a+ d -b) + \cos( a+d+b) \\
& = \cos(a-b) \cos(d) - \sin (a-b) \sin(d) 
+ \cos(a+b) \cos(d) - \sin(a+b)\sin(d) \\
& = \sum_{\fp \in \{-1,1\}} \cos(d) \cos(a+\fp b)  - \sin(d)  \sin(a+ \fp b), 
\end{align*}
Hence 
\begin{align*}
& \int_0^r 2\cos( \tfrac{\pi}{L} m x+ \tfrac{\pi}{L} m z ) \cos( \tfrac{\pi}{r} l x) \dd x \\
& =
\sum_{\fp\in \{-1,1\}}
\frac{\sin( \pi (\tfrac{m}{L} + \fp \tfrac{l}{r}) r)}{\pi (\tfrac{m}{L} + \fp \tfrac{l}{r})}  \cos( \tfrac{\pi}{L} m z ) \\
& \qquad +
\sum_{\fp\in \{-1,1\}}
\frac{1-\cos( \pi r (\tfrac{m}{L} + \fp \tfrac{l}{r}) )}{\pi (\tfrac{m}{L} + \fp \tfrac{l}{r})}  
  \sin( \tfrac{\pi}{L} m z ),
\end{align*}
from which we deduce \eqref{eqn:pairing_n_m_l_and_shift_n_l_r}. 
\end{calc}
\end{obs}

As a consequence we obtain the following. 

\begin{lemma}
\label{lemma:bound_b_mk}
There exists a $C>0$ (independent of $r$ and $L$) such that for all $z\in [0,L-r]$ and $m,l \in \N_0$, 
\begin{align}
\label{eqn:bound_b_mk}
|b_{m,l}^z| 
\le C \sqrt{\tfrac{r}{L}} \tfrac{1}{1+|\frac{r}{L} m - l|}. 
\end{align}
\end{lemma}
\begin{proof}
This follows from the expression \eqref{eqn:pairing_n_m_l_and_shift_n_l_r} by using that $|\frac{\sin(\pi x)}{x}| \lesssim \frac{1}{1+|x|}$ and $\frac{1-\cos(\pi x)}{x} \lesssim \frac{1}{1+|x|}$. 
\begin{calc}
The bounds $\frac{\sin(\pi x)}{x} \le \frac{\pi}{1+x}$ and $\frac{1-\cos(\pi x)}{x} \le \frac{2}{1+x}$ hold for $x \ge 1$, whereas for $x\in (0,1)$ we can use that $\sin(\pi x) \le \pi x$ and that $1- \cos(\pi x) \le 1- \cos(\pi x)^2 = \sin(\pi x)^2\le \sin (\pi x)$. 
\end{calc}
\end{proof}


\begin{theorem}
\label{theorem:bound_on_sum_b_mk_b_ml}
For all $\delta>0$ there exists a $C>0$  (independent of $L$ and $r$)  such that for all $k,l\in \N_0$ and $z\in [0,L-r]$
\begin{align}
\label{eqn:F_epsilon_bound}
\sum_{m\in \N_0} | b_{m,k}^z b_{m,l}^z|  
\le C   (1+ |k-l|)^{\delta-1}. 
\end{align}
\end{theorem}
\begin{proof}
This follows by Lemma \ref{lemma:bound_b_mk} and by \eqref{eqn:sum_bound_with_gamma_and_delta} as $1+| \frac{r}{L}m - u | \ge (1+| \frac{r}{L}m - u |)^{1-\frac{\delta}{2}} $ for $\delta>0$. 
\begin{calc}
With $C$ as in Lemma \ref{lemma:bound_b_mk}
\begin{align*}
\sum_{m\in \N_0} | b_{m,k}^z b_{m,l}^z|  
& \le C^2
\frac{r}{L} 
 \frac{1}{1+|\frac{r}{L}m - k | } 
  \frac{1}{1+|\frac{r}{L}m - l | } \\
& \le C^2   \sum_{m\in\N_0}
\frac{r}{L} 
 \frac{1}{(1+|\frac{r}{L}m - k | )^{1-\frac{\delta}{2}} }
  \frac{1}{(1+|\frac{r}{L}m - l | )^{1-\frac{\delta}{2}}}.
\end{align*}
Now apply \eqref{eqn:sum_bound_with_gamma_and_delta}. 
\end{calc}
\end{proof}

\begin{obs}
\label{obs:bound_absolute_value_sum_less_M_k=l_ge_M_N}
Let $C>0$ be as in Lemma \ref{lemma:bound_b_mk}.  
\begin{enumerate}
\item 
\label{item:k_ge_M_N}
For all  $z\in [0,L-r]$ and $M,k,l\in \N_0$ such that $ \frac{r}{L}M \le l \le k $,  by Lemma~\ref{lemma:monotone_functions_sums_and_integrals}
\begin{calc}
\begin{align*}
|b_{m,k}^z| \vee |b_{m,l}^z| \le C \frac{1}{1+l-\frac{r}{L}m },
\end{align*}
and thus 
\end{calc}
\begin{align*}
\sum_{m=0}^{M-1} 
|b_{m,k}^z b_{m,l}^z|
\cand \begin{calc}
\le C^2 \sum_{m=0}^{M-1} 
\frac{1}{1+(l-\frac{r}{L}m)^2 }
\end{calc} \cnewline
& \le C^2 \int_{0}^M \frac{1}{ (1+l-\frac{r}{L}x)^2} \dd x 
 \le C^2 \frac{1}{1+|l- \frac{r}{L}M|},
\end{align*}
\item 
\label{item:k_less_M_N}
Similarly, for all  $z\in [0,L-r]$ and $M,k,l\in \N_0$ such that  $l\le k \le  \frac{r}{L}M$,
 \begin{calc}
\begin{align*}
|b_{m,k}^z| \vee |b_{m,l}^z| \le C \frac{1}{1+\frac{r}{L}m -k},
\end{align*}
and thus 
\end{calc}
\begin{align*}
\sum_{m=M}^{\infty} 
|b_{m,k}^z b_{m,l}^z|
\cand \begin{calc}
\le C^2 \sum_{m=M+1}^{\infty} \frac{1}{1+\frac{r}{L}m -k} + \frac{1}{(1+\frac{r}{L}M -k)^2}
\end{calc}\cnewline
\cand \begin{calc} 
\le C^2 
\int_{M}^\infty \frac{1}{ (1+\frac{r}{L}x-k)^2} \dd x 
+ \frac{1}{1+|\frac{r}{L}M -k|}
\end{calc} \cnewline
& \le (C^2 +1) \frac{1}{1+|k- \frac{r}{L}M|} .
\end{align*}
%
\end{enumerate}
\end{obs}

As a consequence of the above and $\sum_{m\in \N_0} b_{m,k}^z b_{m,l}^z = \delta_{k,l}$ we obtain the following lemma. 

\begin{lemma}
\label{lemma:bounds_on_h_M}
There exists a $C>0$ such that for all  $z\in [0,L-r]$, $M \in [0,\infty) $ and $k,l \in \N_0$: If either $k\ne l$ or $k=l \le \frac{r}{L}M$, then 
\begin{align}
\label{eqn:bound_sum_ge_L_eps}
\frac{1}{C}\Big| \sum_{m\in \N_0, m \ge M} b_{m,k}^z b_{m,l}^z \Big|
 \le
 \frac{1}{(1+|k- \frac{r}{L}M|)^{1-\delta}} +  \frac{1}{(1+|l- \frac{r}{L}M| ) ^{1-\delta}}, 
\end{align}
and if either $k\ne l$ or $k =l \ge \frac{r}{L}M$ 
\begin{align}
\label{eqn:bound_sum_less_L_eps}
\frac{1}{C} \Big| \sum_{m\in N_0, m <M } b_{m,k}^z b_{m,l}^z \Big|
 \le
 \frac{1}{(1+|k- \frac{r}{L}M|)^{1-\delta}} +  \frac{1}{(1+|l- \frac{r}{L}M| ) ^{1-\delta}} .
\end{align}
\end{lemma}
\begin{proof}
By \eqref{eqn:bound_b_mk} we may assume $M\in \N_0$. 
The statements for $k=l$ follow immediately by the bounds in \ref{obs:bound_absolute_value_sum_less_M_k=l_ge_M_N}.
For $k \ne l$ we have $\sum_{m\in N_0, m <M }  b_{m,k}^z b_{m,l}^z = \sum_{m\in N_0, m \ge M }  b_{m,k}^z b_{m,l}^z$ so that the rest follows by  \ref{obs:bound_absolute_value_sum_less_M_k=l_ge_M_N} and by observing that if $l \le \frac{r}{L}M \le k$ that $|\sum_{m\in N_0, m <M } b_{m,k}^z b_{m,l}^z|   \lesssim  \frac{1}{(1+|k- l|)^{1-\delta}}$ by Theorem \ref{theorem:bound_on_sum_b_mk_b_ml}, which is less than the right-hand side of both \eqref{eqn:bound_sum_ge_L_eps} and \eqref{eqn:bound_sum_less_L_eps}.
\end{proof}

\begin{lemma}
\label{lemma:bound_G_eps_with_eps_dependence}
Write 
$G_{r,\epsilon}(k,l)  = \E[X_{k,r}^\epsilon X_{l,r}^\epsilon - Y_{k,r}^\epsilon Y_{l,r}^\epsilon  ]$. 
There exists a $C>0$ such that for all $\epsilon>0$ and $k,l \in \N_0^2$
\begin{align}
\label{eqn:bound_G_eps_with_eps_dependence}
\frac{1}{C} |G_{r,\epsilon}(k,l)|
\le  
\begin{cases}
\prod_{i=1}^2 \frac{1}{(1+ |k_i - \frac{r}{\epsilon}|)^{1-\delta}} +  \frac{1}{(1+ |l_i - \frac{r}{\epsilon}|)^{1-\delta}} 
& \mbox{if for } i \in \{1,2\} \mbox{ either } k_i \ne l_i \\
& \mbox{ or }  k_i = l_i \ge \frac{r}{\epsilon}, \\
 \frac{1}{(1+ |k_i - \frac{r}{\epsilon}|)^{1-\delta}} +  \frac{1}{(1+ |l_i - \frac{r}{\epsilon}|)^{1-\delta}} 
&  \mbox{if either }
 k_i \ne l_i \mbox{ or }  k_i = l_i \ge \frac{r}{\epsilon} \\
 & 
  \mbox{ and } k_{3-i} = l_{3-i} < \frac{r}{\epsilon}, \\
 \frac{1}{(1+ |k_1 - \frac{r}{\epsilon}|)^{1-\delta}}  + \frac{1}{(1+ |k_2 - \frac{r}{\epsilon}|)^{1-\delta}}  
& k_i = l_i < \frac{r}{\epsilon} \mbox{ for } i \in \{1,2\}. 
\end{cases}
\end{align}
\end{lemma}
\begin{proof}
Let $(k,l) \in \N_0^2 \times \N_0^2 $ be such that  $k=l$ with $|k|_\infty < \frac{r}{\epsilon}$. Then (see \eqref{eqn:formula_G_eps_in_terms_of_b})
\begin{align*}
|G_{r,\epsilon}(k,l)| 
\cand \begin{calc}
 = 
\Big| \sum_{m  \in \N_0^2} 
\prod_{i=1}^2 b_{m_i,k_i}^{y_i} b_{m_i,l_i}^{y_i}
[\tau(\tfrac{\epsilon}{L} m)^2
- 1 ] \Big|
\end{calc} \cnewline
& = 
\Big| \sum_{m  \in \N_0^2: |m|_\infty \ge \frac{L}{\epsilon} } 
\prod_{i=1}^2 b_{m_i,k_i}^{y_i} b_{m_i,l_i}^{y_i}
 \Big| 
 \lesssim 
\Big|
\sum_{m\in\N_0, m \ge \frac{L}{\epsilon}} (b_{m,k_1}^{y_1})^2
\Big| 
+ 
\Big|
\sum_{m\in\N_0, m \ge \frac{L}{\epsilon}} (b_{m,k_2}^{y_2})^2
\Big|.
\end{align*}
If $k$ and $l$ are not like that, then 
\begin{align*}
G_{r,\epsilon}(k,l)
& = \Big(\sum_{m\in \N_0, m < \frac{L}{\epsilon} } b_{m,k_1}^{y_1} b_{m,l_1}^{y_1}\Big) 
\Big( \sum_{m\in \N_0, m < \frac{L}{\epsilon} } b_{m,k_2}^{y_2} b_{m,l_2}^{y_2}\Big). 
\end{align*}
So that the bound \eqref{eqn:bound_G_eps_with_eps_dependence} follows from Lemma \ref{lemma:bounds_on_h_M}. 
\end{proof}

\appendix

\section{The min-max formula for smooth potentials}

\begin{lemma}
\label{lemma:smooth_approximation_orthogonal_sobolev_functions}
Let $f_1,\dots, f_n$ be pairwise orthogonal in $H_0^2$. 
There exist pairwise orthogonal $f_{1,k}, \dots, f_{n,k}$ in $C_c^\infty$ for $k\in\N$ such that for all $i$
\begin{align}
\label{eqn:convergence_f_i_k}
f_{i,k} \xrightarrow{k \rightarrow \infty} f_i \quad \mbox{ in } H_0^2. 
\end{align}
\end{lemma}
\begin{proof}
Let $g_{i,k} \in C_c^\infty$ be such that $g_{i,k} \rightarrow f_i$ in $H_0^2$ for all $i$.
By doing a Gram-Schmidt procedure on $g_{1,k}, \dots, g_{n,k}$ we can give the proof by induction. 
We prove the induction step, assuming that $f_{1,k}=g_{1,k}, \dots, f_{n-1,k}=g_{n-1,k}$ are pairwise independent. We define 
\begin{align*}
f_{n,k} = g_{n,k} - \sum_{i=1}^{n-1} \frac{ \langle g_{n,k}, f_{i,k} \rangle}{ \langle f_{i,k}, f_{i,k} \rangle } f_{i,k}. 
\end{align*}
Then $f_{n,k}$ is pairwise independent from $f_{1,k}, \dots, f_{n-1,k}$. As for $i \in \{1,\dots, n-1\}$ we have 
\begin{align*}
\langle g_{n,k} , f_{i,k} \rangle  \rightarrow \langle f_n, f_i \rangle =0, 
\end{align*}
it follows that $f_{n,k} \rightarrow f_n$. 
\end{proof}

\begin{lemma}
\label{lemma:min_max_over_smooth}
Let $\zeta \in L^\infty$, $n\in \N$ and $L>0$. 
Then (for notation see \ref{theorem:dirichlet_summary})
\begin{align}
\label{eqn:min_max_over_smooth}
\lambda_n(Q_L,\zeta) 
& = 
\sup_{\substack{F\sqsubset H_0^2 \\ \dim F = n}}
\inf_{\substack{\psi \in F \\ \|\psi\|_{L^2}=1}} \langle \sH_\zeta \psi , \psi \rangle 
= 
\sup_{\substack{F\sqsubset C_c^\infty \\ \dim F = n}}
\inf_{\substack{\psi \in F \\ \|\psi\|_{L^2}=1}} \langle \sH_\zeta \psi , \psi \rangle . 
\end{align}
\end{lemma}
\begin{proof}
First observe that 
\begin{align*}
\lambda_n(Q_L,\zeta) 
& = 
\sup_{\substack{ f_1,\dots, f_n \in  H_0^2 \\ \langle f_i, f_j \rangle_{H_0^2} = \delta_{ij} }}
\inf_{\substack{\psi = \sum_{i=1}^n \alpha_i f_i  \\ \alpha_i \in [0,1], \sum_{i=1}^n \alpha_i^2 =1 }} \langle \sH_\zeta \psi , \psi \rangle . 
\end{align*}
Let $f_1,\dots,f_n \in H_0^2$ with $\langle f_i, f_j \rangle_{H_0^2}= \delta_{ij}$. 
By Lemma \ref{lemma:smooth_approximation_orthogonal_sobolev_functions}
there exist $f_{1,k}, \dots, f_{n,k}$ in $C_c^\infty$ with $\langle f_{i,k} , f_{j,k} \rangle_{H^2_0} = \delta_{ij}$ (by renormalising) such that \eqref{eqn:convergence_f_i_k} holds. 
Then 
\begin{align*}
& \left|
\inf_{\substack{\psi = \sum_{i=1}^n \alpha_i f_i  \\ \alpha_i \in [0,1], \sum_{i=1}^n \alpha_i^2 =1 }} \langle \sH_\zeta \psi , \psi \rangle 
- 
\inf_{\substack{\psi = \sum_{i=1}^n \alpha_i f_{i,k}  \\ \alpha_i \in [0,1], \sum_{i=1}^n \alpha_i^2 =1 }} \langle \sH_\zeta \psi , \psi \rangle \right| \\
& \le 
\sup_{\substack{\psi = \sum_{i=1}^n \alpha_i f_i , \varphi = \sum_{i=1}^n \alpha_i f_{i,k}  \\ \alpha_i \in [0,1], \sum_{i=1}^n \alpha_i^2 =1 }}
\left|
 \langle \sH_\zeta \psi , \psi \rangle_{L^2} 
-
\langle \sH_\zeta \varphi , \varphi \rangle_{L^2}  \right| \\
& \lesssim 
\sup_{\alpha_i \in [0,1], \sum_{i=1}^n \alpha_i^2 =1 }
\left\| \sum_{i=1}^n \alpha_i f_i - \sum_{i=1}^n \alpha_i f_{i,k} \right\|_{H_0^2} 
\le  \sum_{i=1}^n \| f_i - f_{i,k} \|_{H_0^2} \rightarrow 0. 
\end{align*}
\begin{calc}
We used 
\begin{align*}
& \left|
\langle \sH_\zeta \psi , \psi \rangle_{L^2}
-
\langle \sH_\zeta \varphi , \varphi \rangle_{L^2}
\right| 
 = 
 \left| \langle \sH_\zeta (\psi -\varphi) , \psi \rangle_{L^2}
+
\langle \sH_\zeta \varphi , \psi- \varphi \rangle_{L^2} \right| 
\end{align*}
and $| \langle \sH_\zeta f, g \rangle_{L^2}| \le \| \sH_\zeta f\|_{L^2} \|g\|_{L^2} \lesssim \|f\|_{H_0^2} \|g\|_{H_0^2} $.
\end{calc}
This proves 
\begin{align}
\label{eqn:min_max_of_eigenvalue_with_smooth_and_convex_formulation}
\lambda_n(Q_L,\zeta) 
& = 
\sup_{\substack{ f_1,\dots, f_n \in  C_c^\infty \\ \langle f_i, f_j \rangle_{H_0^2} = \delta_{ij} }}
\inf_{\substack{\psi = \sum_{i=1}^n \alpha_i f_i  \\ \alpha_i \in [0,1], \sum_{i=1}^n \alpha_i^2 =1 }} \langle \sH_\zeta \psi , \psi \rangle , 
\end{align}
and therefore 
\eqref{eqn:min_max_over_smooth}. 
\end{proof}

\section{Useful bound on an integral}
\label{section:integral_bound}

\begin{lemma}
\label{lemma:integral_bound_with_gamma_and_theta}
Let $\gamma, \theta \in (0,1)$ and $\gamma+\theta >1$. 
There exists a $C>0$ such that for all $b>0$ and $u \in \R$ 
\begin{align}
\label{eqn:integral_bound_with_gamma_and_theta_with_u}
\int_0^\infty \frac{1}{(b+|x-u|)^\gamma} \frac{1}{(b+x)^\theta} \dd x 
\le C (b+|u|)^{1-\gamma-\theta}. 
\end{align}
Consequently, there exists a $C>0$ such that for all $b>0$ and $u,v \in \R$
\begin{align}
\label{eqn:integral_bound_with_gamma_and_theta_over_R_with_u_and_v}
\int_\R \frac{1}{(b+|x-u|)^\gamma} \frac{1}{(b+|x-v|)^\theta} \dd x 
\le C (b+|u-v|)^{1-\gamma-\theta}. 
\end{align}
\end{lemma}
\begin{proof}
By a simple substitution argument we may assume $b=1$. 
\begin{calc}
Indeed, by substituting $y= \frac1b x$ we have (assuming \eqref{eqn:integral_bound_with_gamma_and_theta_with_u} holds for $b=1$)
\begin{align*}
\int_0^\infty \frac{1}{(b+|x-u|)^\gamma} \frac{1}{(b+x)^\theta} \dd x 
& = \int_0^\infty \frac{b}{(b+|by-u|)^\gamma (b(1+y))^\theta} \dd y \\
& = b^{1- \gamma - \theta} \int_0^\infty \frac{1}{(1+|y-\frac{u}{b}|)^\gamma ((1+y))^\theta} \dd y \\
& \le C b^{1- \gamma - \theta} (1+|\frac{u}{b}|)^{1-\gamma-\theta}
=C (b+|u|)^{1-\gamma-\theta}. 
\end{align*}
\end{calc}
We have uniformly in  $a\in (0,1)$
\begin{align*}
\int_0^\infty \frac{1}{(a+x)^\gamma} \frac{1}{(1+x)^\theta} \dd x 
\le 
\int_1^\infty \frac{1}{x^{\gamma+\theta}} \dd x 
+ \int_0^1 \frac{1}{(a+x)^{\gamma}} \dd x \lesssim 1+ (1+a)^{1-\gamma} \lesssim 1. 
\end{align*}
Hence for all $u \ge 0$ 
\begin{align}
\notag & \int_u^\infty \frac{1}{(1+x-u)^\gamma} \frac{1}{(1+x)^\theta} \dd x 
= \int_0^\infty \frac{1}{(1+x)^\gamma} \frac{1}{(1+u+x)^\theta} \dd x \\
\label{eqn:integral_bound_with_u_negative}
& \qquad 
= (1+u)^{1-\gamma-\theta} 
\int_0^\infty \frac{1}{(\frac{1}{1+u}+x)^\gamma} \frac{1}{(1+x)^\theta} \dd x 
\lesssim (1+u)^{1-\gamma-\theta} .
\end{align}
On the other hand we have
\begin{align*}
\int_0^{\frac{u}{2}}  \frac{1}{(1+u-x)^\gamma} \frac{1}{(1+x)^\theta} \dd x
\le (1+ \frac{u}{2})^{-\gamma} \int_0^{\frac{u}{2}} \frac{1}{(1+x)^\theta} \dd x \lesssim (1+u)^{1-\gamma-\theta}, 
\end{align*}
and similarly $\int_{\frac{u}{2}}^u \frac{1}{(1+u-x)^\gamma} \frac{1}{(1+x)^\theta} \dd x \lesssim (1+u)^{1-\gamma-\theta}$. 
In case $u$ is negative, the bound is already proved in \eqref{eqn:integral_bound_with_u_negative} (by interchanging $\theta$ and $\gamma$). 

For \eqref{eqn:integral_bound_with_gamma_and_theta_over_R_with_u_and_v} it is sufficient to observe that 
\begin{align*}
\int_v^\infty \frac{1}{(1+|x-u|)^\gamma} \frac{1}{(1+|x-v|)^\theta} \dd x 
& = \int_0^\infty  \frac{1}{(1+|x+v-u|)^\gamma} \frac{1}{(1+x)^\theta} \dd x ,\\
\int_{-\infty}^v \frac{1}{(1+|x-u|)^\gamma} \frac{1}{(1+|x-v|)^\theta} \dd x 
& = \int_0^\infty  \frac{1}{(1+|x+u-v|)^\gamma} \frac{1}{(1+x)^\theta} \dd x. 
\end{align*}
\end{proof}

\begin{calc}

\section{Spectrum of an operator with compact resolvents}
\label{section:spectrum_compact_resolvents}

Let $H$ be a Hilbert space. 
Let $A : \cD \rightarrow H$ be a linear operator, where $\cD $ is a linear subspace of $H$. $\rho(A)$ denotes the resolvent set of $A$, $\sigma(A)$ the spectrum and $\sigma_p(A)$ the point spectrum, i.e., the set of eigenvalues.

\begin{theorem}
\label{theorem:spectrum_under_compact_resolvent}
Suppose $H$ is infinite dimensional. 
Let $\alpha>0$. 
Suppose $\mu \in (-\infty, \alpha] \subset \rho(A)$ and write $R_\mu = (\mu - A)^{-1}$. 
Suppose that $R_\mu$ is a self-adjoint compact operator as a map $H\rightarrow H$. 
Then 
\begin{align}
\label{eqn:spectrum_A_in_terms_of_a_spectrum_resolvent}
\sigma(A) = \sigma_p(A) = \{   \mu -\tfrac{1}{\lambda} : \lambda \in \sigma_p(R_\mu) \setminus \{0\} \}. 
\end{align}
Suppose moreover that $A$ is a closed symmetric (densely defined) operator. 
Then $A$ is self-adjoint and has an (at most) countable spectrum without accumulation points. 
For all $\lambda \in \sigma(A)$, $\ker (\lambda -A)$ is finite dimensional and 
\begin{align}
\label{eqn:domain_A_as_direct_sum_kernels_eigenvalue_min_A}
\cD= \bigoplus_{\lambda \in \sigma(A)} \ker ( \lambda - A),
\end{align}
Let $\lambda_1\le \lambda_2 \le \cdots$ be such that  $\sigma(A) = \{ \lambda_n : n\in\N\}$ and such that $\#\{n\in\N: \lambda_n = \lambda\} = \dim \ker (\lambda - A)$ for all $\lambda \in \sigma(A)$. 
Then with the notation $\sqsubset $ for ``is a linear subspace of'' 
\begin{align}
\label{eqn:min_max_for_lambda}
\lambda_n 
&= 
\inf_{\substack{F\sqsubset \cD \\ \dim F = n}}
\sup_{\substack{v\in F\\ \|v\|_{H}=1}} \langle Av, v\rangle, \\
\label{eqn:max_min_for_one_over_mu_min_lambda}
\frac{1}{\mu - \lambda_n} 
& =
 \sup_{\substack{F\sqsubset H \\ \dim F = n}}
\inf_{\substack{v\in F\\ \|v\|_{H}=1}} \langle R_\mu v, v\rangle. 
\end{align}
\end{theorem}

\begin{theorem}
\label{theorem:kato_operators_with_compact_resolvent}
\cite[Theorem 6.29]{Ka95}
Let $A$ be a closed operator and $\mu \in \rho(A)$. 
If $R_\mu$ is compact, then $\sigma(A)$ consists of countably many eigenvalues with finite multiplicities and has no  accumulation points. Moreover $R_\lambda$ is compact for all $\lambda \in \rho(A)$. 
\end{theorem}

\begin{theorem}[F. Riesz]
\label{theorem:Ru425}
\cite[Theorem 4.25]{Ru91}
\cite[Theorem VI.7.1]{Co07}
Let $R : H \rightarrow H$ be a compact operator. 
Then $\sigma(R) \setminus \{0\} = \sigma_p(R)$, $\sigma(R)$ is countable and has at most one limit point, namely $0$. If $\dim(H) = \infty$, then $0 \in \sigma(R)$. 
\end{theorem}

\begin{calc}
\begin{theorem}[Riesz-Schauder theorem]\cite[Theorem VI.15]{ReSi75} 
Let $A$ be a compact operator on $H$. 
Then $\sigma(A)$ is countable with no accumulation point except possibly $0$. 
Further, every $\lambda \in \sigma(A) \setminus \{0\}$ is an eigenvalue of finite multiplicity. 
\end{theorem}
\end{calc}

\begin{theorem}
\label{theorem:Ru1229}
\cite[Theorem 12.29]{Ru91}
If $T: H \rightarrow H$ is a normal operator and $\sigma(T)$ is countable, then $H= \bigoplus_{\lambda \in \sigma(T)} \ker(\lambda -T)$. 
\end{theorem}

%
%
%

\begin{lemma} 
[Fischer's principle] \cite[Section 28, Theorem 4, p. 318]{La02}\footnote{In this reference the operator is actually assumed to be compact and symmetric, but this is only done to guarantee that the spectrum is countable and ordered.} 
\label{lemma:min_max_and_max_min_shorter}
Suppose that $\sigma_p(A) = \{ \lambda_n : n\in\N\}$ and $\#\{n\in\N: \lambda_n = \lambda\} = \dim \ker (\lambda - A)$ for all $\lambda \in \sigma(A)$. 
If $\lambda_{n} \le \lambda_{n+1}$ for all $n\in\N$, then 
\begin{align*}
\lambda_n = 
\inf_{\substack{F\sqsubset \cD \\ \dim F = n}}
\sup_{\substack{v\in F\\ \|v\|_{H}=1}} \langle Av, v\rangle. 
\end{align*}
 If $\lambda_{n} \ge \lambda_{n+1}$ for all $n\in\N$, then 
\begin{align*}
\lambda_n = 
\sup_{\substack{F\sqsubset \cD \\ \dim F = n}}
\inf_{\substack{v\in F\\ \|v\|_{H}=1}} \langle Av, v\rangle. 
\end{align*}
\end{lemma}

These theorems can be used to give a short  proof of Theorem \ref{theorem:spectrum_under_compact_resolvent}. 

\begin{proof}[Proof of Theorem \ref{theorem:spectrum_under_compact_resolvent}]
For the most statements one can combine 
Theorems \ref{theorem:kato_operators_with_compact_resolvent} and \ref{theorem:Ru425}. 
\eqref{eqn:spectrum_A_in_terms_of_a_spectrum_resolvent} follows from
\begin{align*}
\lambda( \mu - \tfrac{1}{\lambda}  - A) 
&= \lambda (\mu - A) -1 = (\lambda - R_\mu) (\mu - A) \\
&=   (\mu - A)\lambda -1  = (\mu - A) (\lambda - R_\mu), 
\end{align*}
as this implies that
$\lambda - R_\mu$ is boundedly invertible (or injective)  if and only if $ \mu - \tfrac{1}{\lambda}  - A$ is. 
\begin{calc}
If $Q$ is an inverse for $\mu - \frac{1}{\lambda} - A$ then 
\begin{align*}
\frac{1}{\lambda}Q ( \lambda - R_\mu) (\mu - A) = (\lambda - R_\mu) ( \mu - A) \frac{1}{\lambda}Q = I
\end{align*}
Therefore $\frac{1}{\lambda}(\mu-A)Q$ is an inverse of $\lambda - R_\mu$, as
\begin{align*}
\frac{1}{\lambda}(\mu- A) Q (\lambda - R_u) 
& = \frac{1}{\lambda}(\mu- A) Q (\lambda - R_u) (\mu- A) R_\mu \\
& = 
(\mu- A)  (\lambda - R_u) (\mu- A) \frac{1}{\lambda}Q R_\mu 
= (\mu- A)  I R_\mu = I. 
\end{align*}
Vice versa if $Q$ is an inverse $\lambda - R_\mu$, then $\lambda Q R_\mu$ is an inverse for $\mu - \frac{1}{\lambda} - A$.
\end{calc}
Moreover, $\ker (\mu - \tfrac{1}{\lambda}  - A) = \ker  (\lambda - R_\mu)$. 
Then \eqref{eqn:domain_A_as_direct_sum_kernels_eigenvalue_min_A} follows by applying Theorem \ref{theorem:Ru1229} to $R_\mu$ and observing that $0 \in \sigma(R_\mu)$ because $\dim(H) = \infty$ and $(\ker R_\mu)^\perp = \ran(R_\mu) = \cD$. 
\eqref{eqn:min_max_for_lambda} and \eqref{eqn:max_min_for_one_over_mu_min_lambda} follow from Lemma \ref{lemma:min_max_and_max_min_shorter}.
\end{proof}

\begin{calc}
\begin{definition}
\label{def:A_bounded}
Let $A: \cD(A) \rightarrow H$ and $B: \cD(B) \rightarrow H$ be densely defined operators. We say that $B$ is $A$-\emph{bounded} when 
\begin{enumerate}
\item $\cD(A) \subset \cD(B)$,
\item There exist $a,b\ge 0$ such that for all $\varphi \in \cD(A)$ 
\begin{align}
\label{eqn:A_bounded}
\|B \varphi\| \le a \|A \varphi\| + b\|\varphi\|. 
\end{align}
\end{enumerate}
The infimum of such $a$ is called the \emph{relative bound} of $B$ with respect to $A$. If the relative bound equals $0$, we say that $B$ is infinitesimally small with respect to $A$. 
\end{definition}

\begin{theorem}[Hilbert-Schmidt theorem] \cite[Theorem VI.16]{ReSi75}
Let $A$ be a self-adjoint compact operator on $H$. 
Then there is a complete orthonormal basis $(\phi_n)_{n\in\N}$ for $H$ such that $A\phi_n = \lambda_n \phi_n$ and $\lambda_n \rightarrow 0$. 
\end{theorem}

\begin{theorem}[The Kato-Rellich theorem]
\label{theorem:kato_rellich}
\cite[Theorem X.12]{ReSi75}
Suppose $A$ is self-adjoint, $B$ is symmetric and $B$ is $A$-bounded with relative bound $a<1$. 
Then $A+B$ is self-adjoint on $\cD(A)$. 
Furthermore, if $\sigma(A) \subset [M,\infty)$, then $\sigma(A+B) \subset [M - \max\{ \frac{b}{1-a}, a|M|+ b\}, \infty)$ (where $a,b$ are as in \eqref{eqn:A_bounded}). 
\end{theorem}
\end{calc}

\end{calc}

\bibliographystyle{abbrv}
\bibliography{references}

\begin{thebibliography}{10}

\bibitem{AlCh15}
R.~Allez and K.~Chouk.
\newblock The continuous anderson hamiltonian in dimension two.
\newblock Preprint available at \url{https://arxiv.org/abs/1511.02718}.

\bibitem{BaChDa11}
H.~Bahouri, J.-Y. Chemin, and R.~Danchin.
\newblock {\em Fourier analysis and nonlinear partial differential equations},
  volume 343 of {\em Grundlehren der Mathematischen Wissenschaften [Fundamental
  Principles of Mathematical Sciences]}.
\newblock Springer, Heidelberg, 2011.

\bibitem{BiKo01LTT}
M.~Biskup and W.~K\"onig.
\newblock Long-time tails in the parabolic {A}nderson model with bounded
  potential.
\newblock {\em Ann. Probab.}, 29(2):636--682, 2001.

\bibitem{CaFrGa17}
G.~Cannizzaro, P.~K. Friz, and P.~Gassiat.
\newblock Malliavin calculus for regularity structures: the case of g{PAM}.
\newblock {\em J. Funct. Anal.}, 272(1):363--419, 2017.

\bibitem{CaMo95}
R.~A. Carmona and S.~A. Molchanov.
\newblock Stationary parabolic {A}nderson model and intermittency.
\newblock {\em Probab. Theory Related Fields}, 102(4):433--453, 1995.

\bibitem{Ch10}
X.~Chen.
\newblock {\em Random walk intersections}, volume 157 of {\em Mathematical
  Surveys and Monographs}.
\newblock American Mathematical Society, Providence, RI, 2010.
\newblock Large deviations and related topics.

\bibitem{Ch14}
X.~Chen.
\newblock Quenched asymptotics for {B}rownian motion in generalized {G}aussian
  potential.
\newblock {\em Ann. Probab.}, 42(2):576--622, 2014.

\bibitem{ChGaPe17}
K.~Chouk, J.~Gairing, and N.~Perkowski.
\newblock An invariance principle for the two-dimensional parabolic {A}nderson
  model with small potential.
\newblock {\em Stoch. Partial Differ. Equ. Anal. Comput.}, 5(4):520--558, 2017.

\bibitem{Co07}
J.~B. Conway.
\newblock {\em A Course in Functional Analysis}, volume~96 of {\em Graduate
  Texts in Mathematics}.
\newblock Springer-Verlag, New York, second edition, 1990.

\bibitem{DeZe10}
A.~Dembo and O.~Zeitouni.
\newblock {\em Large Deviations Techniques and Applications}, volume~38 of {\em
  Stochastic Modelling and Applied Probability}.
\newblock Springer-Verlag, Berlin, 2010.
\newblock Corrected reprint of the second (1998) edition.

\bibitem{DeSt89}
J.-D. Deuschel and D.~W. Stroock.
\newblock {\em Large Deviations}, volume 137 of {\em Pure and Applied
  Mathematics}.
\newblock Academic Press, Inc., Boston, MA, 1989.

\bibitem{Do69}
W.~F. Donoghue, Jr.
\newblock {\em Distributions and {F}ourier transforms}, volume~32 of {\em Pure
  and Applied Mathematics}.
\newblock Academic Press, New York, 1969.

\bibitem{DuLa17}
L.~Dumaz and C.~Labb\'{e}.
\newblock Localization of the continuous {A}nderson {H}amiltonian in 1-{D}.
\newblock {\em Probab. Theory Related Fields}, 176(1-2):353--419, 2020.

\bibitem{FuNa76}
M.~Fukushima and S.~Nakao.
\newblock On spectra of the {S}chr\"odinger operator with a white {G}aussian
  noise potential.
\newblock {\em Z. Wahrscheinlichkeitstheorie und Verw. Gebiete},
  37(3):267--274, 1976/77.

\bibitem{GaKo00}
J.~G\"artner and W.~K\"onig.
\newblock Moment asymptotics for the continuous parabolic {A}nderson model.
\newblock {\em Ann. Appl. Probab.}, 10(1):192--217, 2000.

\bibitem{GuImPe15}
M.~Gubinelli, P.~Imkeller, and N.~Perkowski.
\newblock Paracontrolled distributions and singular {PDE}s.
\newblock {\em Forum Math. Pi}, 3:e6, 75, 2015.

\bibitem{GuUgZa20}
M.~Gubinelli, B.~Ugurcan, and I.~Zachhuber.
\newblock Semilinear evolution equations for the {A}nderson {H}amiltonian in
  two and three dimensions.
\newblock {\em Stoch. Partial Differ. Equ. Anal. Comput.}, 8(1):82--149, 2020.

\bibitem{Ja97}
S.~Janson.
\newblock {\em Gaussian {H}ilbert spaces}, volume 129 of {\em Cambridge Tracts
  in Mathematics}.
\newblock Cambridge University Press, Cambridge, 1997.

\bibitem{Ka97}
O.~Kallenberg.
\newblock {\em Foundations of modern probability}.
\newblock Probability and its Applications (New York). Springer-Verlag, New
  York, 1997.

\bibitem{Ko16}
W.~K\"onig.
\newblock {\em The parabolic {A}nderson model}.
\newblock Pathways in Mathematics. Birkh\"auser/Springer, [Cham], 2016.
\newblock Random walk in random potential.

\bibitem{La19}
C.~Labb\'{e}.
\newblock The continuous {A}nderson {H}amiltonian in {$d\leq 3$}.
\newblock {\em J. Funct. Anal.}, 277(9):3187--3235, 2019.

\bibitem{La58}
O.~A. Lady\v{z}enskaja.
\newblock Solution ``in the large'' to the boundary-value problem for the
  {N}avier-{S}tokes equations in two space variables.
\newblock {\em Soviet Physics. Dokl.}, 123 (3):1128--1131 (427--429 Dokl. Akad.
  Nauk SSSR), 1958.

\bibitem{La02}
P.~D. Lax.
\newblock {\em Functional analysis}.
\newblock Pure and Applied Mathematics (New York). Wiley-Interscience [John
  Wiley \& Sons], New York, 2002.

\bibitem{Ni59}
L.~Nirenberg.
\newblock On elliptic partial differential equations.
\newblock {\em Ann. Scuola Norm. Sup. Pisa (3)}, 13:115--162, 1959.

\bibitem{Nu09}
D.~Nualart.
\newblock {\em Malliavin calculus and its applications}, volume 110 of {\em
  CBMS Regional Conference Series in Mathematics}.
\newblock Published for the Conference Board of the Mathematical Sciences,
  Washington, DC; by the American Mathematical Society, Providence, RI, 2009.

\bibitem{Pa65}
A.~Papoulis.
\newblock {\em Probability, random variables, and stochastic processes}.
\newblock McGraw-Hill Book Co., New York-Toronto-London-Sydney, 1965.

\bibitem{PeRo19}
N.~Perkowski and T.~Rosati.
\newblock A rough super-brownian motion.
\newblock Preprint, \url{http://arxiv.org/abs/1905.05825}.

\bibitem{PrTr16}
D.~J. Pr\"omel and M.~Trabs.
\newblock Rough differential equations driven by signals in {B}esov spaces.
\newblock {\em J. Differential Equations}, 260(6):5202--5249, 2016.

\bibitem{ReSi75}
M.~Reed and B.~Simon.
\newblock {\em Methods of modern mathematical physics. {II}. {F}ourier
  analysis, self-adjointness}.
\newblock Academic Press [Harcourt Brace Jovanovich, Publishers], New
  York-London, 1975.

\bibitem{ReSi78}
M.~Reed and B.~Simon.
\newblock {\em Methods of modern mathematical physics. {IV}. {A}nalysis of
  operators}.
\newblock Academic Press [Harcourt Brace Jovanovich, Publishers], New
  York-London, 1978.

\bibitem{ScTr87}
H.-J. Schmeisser and H.~Triebel.
\newblock {\em Topics in {F}ourier analysis and function spaces}.
\newblock A Wiley-Interscience Publication. John Wiley \& Sons, Ltd.,
  Chichester, 1987.

\bibitem{Si83}
B.~Simon.
\newblock Semiclassical analysis of low lying eigenvalues. {I}. {N}ondegenerate
  minima: asymptotic expansions.
\newblock {\em Ann. Inst. H. Poincar\'{e} Sect. A (N.S.)}, 38(3):295--308,
  1983.

\bibitem{StSh03}
E.~M. Stein and R.~Shakarchi.
\newblock {\em Fourier analysis}, volume~1 of {\em Princeton Lectures in
  Analysis}.
\newblock Princeton University Press, Princeton, NJ, 2003.
\newblock An introduction.

\bibitem{Tr83}
H.~Triebel.
\newblock {\em Theory of function spaces}, volume~78 of {\em Monographs in
  Mathematics}.
\newblock Birkh\"auser Verlag, Basel, 1983.

\bibitem{Ve11}
M.~C. Veraar.
\newblock Regularity of {G}aussian white noise on the {$d$}-dimensional torus.
\newblock In {\em Marcinkiewicz centenary volume}, volume~95 of {\em Banach
  Center Publ.}, pages 385--398. Polish Acad. Sci. Inst. Math., Warsaw, 2011.

\end{thebibliography}

\end{document}